\documentclass[12pt]{amsart}

\usepackage{graphicx, overpic}
\usepackage[below]{placeins}
\usepackage[colorlinks=true, linkcolor=blue, citecolor=blue]{hyperref}
\usepackage[]{algorithm2e}
\usepackage{comment}

\usepackage[T1]{fontenc}
\usepackage{amsmath,amsthm,amscd,amssymb,eucal}

\usepackage{enumerate, amsfonts, latexsym, color, url}
\usepackage{epstopdf}
\usepackage{pinlabel}
\usepackage{calrsfs}
\DeclareMathAlphabet{\pazocal}{OMS}{zplm}{m}{n}

\begin{document}

\newtheorem{conx}{Conjecture}
\renewcommand{\theconx}{\Alph{conx}} 

\newtheorem{theorem}{Theorem}[section]
\newtheorem{lemma}[theorem]{Lemma}
\newtheorem{proposition}[theorem]{Proposition}
\newtheorem{corollary}[theorem]{Corollary}
\newtheorem{conjecture}[theorem]{Conjecture}
\newtheorem{question}[theorem]{Question}
\newtheorem{problem}[theorem]{Problem}
\newtheorem*{claim}{Claim}
\newtheorem*{criterion}{Criterion}
\newtheorem*{cut_point_theorem}{Cut Point Theorem}
\newtheorem*{linearization_theorem}{Linearization Theorem}

\theoremstyle{definition}
\newtheorem{definition}[theorem]{Definition}
\newtheorem{provisional_definition}[theorem]{Provisional Definition}
\newtheorem{construction}[theorem]{Construction}
\newtheorem{notation}[theorem]{Notation}
\newtheorem{object}[theorem]{Object}
\newtheorem{operation}[theorem]{Operation}

\theoremstyle{remark}
\newtheorem{remark}[theorem]{Remark}
\newtheorem{example}[theorem]{Example}
\newtheorem*{case}{Case}

\numberwithin{equation}{subsection}

\def\Z{{\mathbb Z}}
\def\N{{\mathbb N}}
\def\R{{\mathbb R}}
\def\C{{\mathbb C}}
\def\CC{{\pazocal C}}
\def\D{{\mathbb D}}
\def\H{{\mathbb H}}
\def\SS{\Gamma}	
\def\I{{\pazocal I}}
\def\IG{IG}	
\def\L{{\pazocal L}}
\def\M{{\pazocal M}} %
\def\Mord{{\pazocal M}_{\ord}}
\def\P{{\pazocal P}}
\def\T{{\pazocal T}}

\def\diam{\textnormal{diam}}
\def\dist{\textnormal{dist}}
\def\ord{\textnormal{ord}}
\def\length{\textnormal{length}}
\def\fix{\textnormal{fix}}
\def\inte{\textnormal{int}}
\def\limit{\textnormal{limit}}
\def\cut{\textnormal{cut}}

\newcommand\numberthis{\addtocounter{equation}{1}\tag{\theequation}}
\newcommand{\marginal}[1]{\marginpar{\tiny #1}}

\title{Laminations and External Angles for Similarity Pairs}
\author{Danny Calegari}
\address{University of Chicago \\ Chicago, Ill 60637 USA}
\email{dannyc@uchicago.edu}

\author{Alden Walker}
\address{Center for Communications Research \\ La Jolla, CA 92121 USA}
\email{akwalke@ccr-lajolla.org}
\date{\today}

\begin{abstract}
A {\em similarity pair} is the dynamical system in $\C$ generated by two maps
$f:z \to sz-1$ and $g:z \to sz+1$ for $|s|<1$. Associated to the dynamical system
is an attractor $\Lambda$. The Barnsley--Harrington Mandelbrot set $\M$ is the
set of $s\in \D$ for which $\Lambda$ is connected. Let $K$ denote the filled set of a
connected $\Lambda$. For $s\in \partial \M$ we show that the (partially defined) action of 
the semigroup on $\partial K$ is topologically conjugate to a (discontinuous)
piecewise linear action of constant slope. Conditional on a conjecture (satisfied
for `most' $s\in \partial \M$) we give a necessary and sufficient condition in terms
of the dynamics on $\partial K$ for $K$ to contain cut points, and we describe the
set of all such cut points in terms of infinite walks in a directed graph $\IG$
obtained by an explicit recursive algorithm.

The structure of the `dynamical cut point set' for a 2-dimensional family of 
piecewise linear actions (containing those coming from $s\in \partial \M$) recovers
and generalizes the Douady--Hubbard--Thurston quadratic minor lamination for the
abstract Mandelbrot set.
\end{abstract}

\maketitle
\setcounter{tocdepth}{1}
\tableofcontents

\section{Introduction}

A {\em similarity pair} is an iterated function system (hereafter IFS) generated
by two conjugate similarities of the plane. Let $\SS$ denote the free semigroup on two
generators $f$ and $g$. If we fix a complex number $s$
with $|s|<1$, the similarity pair associated to $s$ is the action of $\SS$ on
$\C$ given by $f:z \to sz-1$ and $g:z \to sz+1$. 
In the sequel when we need to stress the dependence on $s$ we usually do this with 
subscripts.

Associated to a similarity pair is its {\em attractor} $\Lambda$, 
the unique compact, nonempty subset of $\C$ with 
$\Lambda = f\Lambda \cup g\Lambda$. Another way to describe $\Lambda$ is that it is
the closure of the fixed points of the nontrivial elements of the semigroup. See
Definition~\ref{definition:limit_set}.

Similarity pairs are a class of one dimensional holomorphic dynamical systems that lie 
somewhere intermediate between holomorphic dynamics, and Kleinian
groups. In this paper we will discuss a geometric object (a dynamical system on the
circle and an invariant lamination) that is analogous to the invariant 
dynamical laminations arising in holomorphic dynamics.

In 1985 Barnsley--Harrington \cite{Barnsley_Harrington} defined a `Mandelbrot set' 
$\M \subset \D$ for similarity pairs to be the set of complex parameters $s$ 
for which $\Lambda$ is connected; this is supposed to be very closely analogous to
the (ordinary) Mandelbrot set $\Mord$, the set of complex parameters $c$ for which the Julia
set $J_c$ of $h_c:z \to z^2+c$ is connected. Similarity pairs, the set $\M$ and
generalizations were subsequently studied by many people, 
e.g.\/ \cite{Bousch1,Bousch2,Bandt,Solomyak_local,Solomyak,Shmerkin_Solomyak,
Thurston_entropy,Calegari_Koch_Walker,Espigule_Juger_Saldana,Lindsay_Tiozzo_Wu} and others.

\medskip

In this paper we study the dynamics of similarity pairs and the topology of $\Lambda$
for parameters $s$ in the frontier $\partial \M$ of $\M$. Suppose $s$ is in $\M$ so that
$\Lambda$ is connected. In this case we may define $K$ to be
the {\em full} set associated to $\Lambda$, i.e.\/ the smallest compact simply-connected
subset of $\C$ containing $\Lambda$. It is a fact that $\Lambda$, and therefore
also $K$, is locally connected; this is in stark contrast to the case of Julia sets
of rational maps, which can fail to be locally connected even for $h_c:z \to z^2+c$,
for example when $c$ is a Cremer point, so that $h_c$ has an irrationally indifferent 
fixed point where the dynamics cannot be linearized, see e.g.\/ \cite{Milnor}, 
Corollary~18.6. 
In our case, since $K$ is simply-connected, its topology is entirely captured by 
the structure of its set of cut points. In this paper we are able to give a simple
sufficient condition for the existence of cut points in $K$ for $s \in \partial \M$, 
and conditional on a conjecture (Conjecture~\ref{conjecture:A}) which holds for 
more than $99\%$ of the points in $\partial \M$ (in a certain sense), this condition is
also necessary. Furthermore, also conditional on Conjecture~\ref{conjecture:A}, 
we give an explicit algorithm that describes the structure of the set of cut points
of $K$ in terms of infinite paths in a directed graph.

In the case of the ordinary Mandelbrot set and cut points of $J_c$, the story is
as follows. Isolated cut points in (filled) Julia sets $J_c$ for $h_c:z \to z^2+c$ appear or
disappear as $c$ moves between tangent bulbs of $\M$, and a periodic orbit in $J_c$
changes from attracting to repelling. Points $c$ for which $J_c$ has a rationally
indifferent periodic cycle are parameterized by isolated leaves of the Douady--Hubbard--Thurston
{\em quadratic minor lamination} \cite{Douady_Hubbard_1,Douady_Hubbard_2,Thurston};
and conditional on the notorious MLC conjecture (i.e.\/ the conjecture
that $\Mord$ is locally connected) the quotient of the closed disk by this lamination
gives an abstract topological parameterization of $\Mord$. Again conditional on MLC,
one obtains a surjective parameterization of $\partial \Mord$ by a circle, 
given by external angles of landing rays.

In this paper we associate to $s \in \partial \M$ a {\em pair} of parameters 
$(\theta,\lambda)$ consisting of an {\em external angle} 
$\theta \in \R/2\pi\Z$ and a {\em scale factor} $\lambda \in [1,2]$. From this
pair of numbers one can directly compute the structure of the cut points of $K$,
conditional on Conjecture~\ref{conjecture:A}. 

Not every parameter $(\theta,\lambda)$ is associated to some $s\in \partial \M$, but every 
parameter describes a 1-dimensional piecewise linear dynamical system of a certain
kind. Our cut point criterion for $K$ makes sense as a purely dynamical criterion for these
more general dynamical systems, and in the special case $\lambda = 2$ this dynamical
criterion completely recovers the structure of the 
Douady--Hubbard--Thurston lamination; see \S~\ref{subsection:innermost_leaves_and_DHT}.

\subsection{Uniformization and the Dynamical Lamination}

Suppose $s \in \M$ so that $\Lambda$ is connected. 
The complement $\C - \Lambda$ contains a unique unbounded component, and
we may define the {\em filled attractor} $K$ to be its complement. In other words,
$K$ is the full set associated to $\Lambda$.
For any $s\in \M$ the set $\Lambda$ is locally connected and therefore so is $K$.
Thus there is a conformal isomorphism $u:\C - \overline{\D} \to \C - K$ that extends 
continuously to $\partial u: S^1 \to \partial K$. Note that $\partial K \subset \Lambda$.
See Figure~\ref{uniformization}.

\begin{figure}[htpb]
\centering
\includegraphics[scale=0.2]{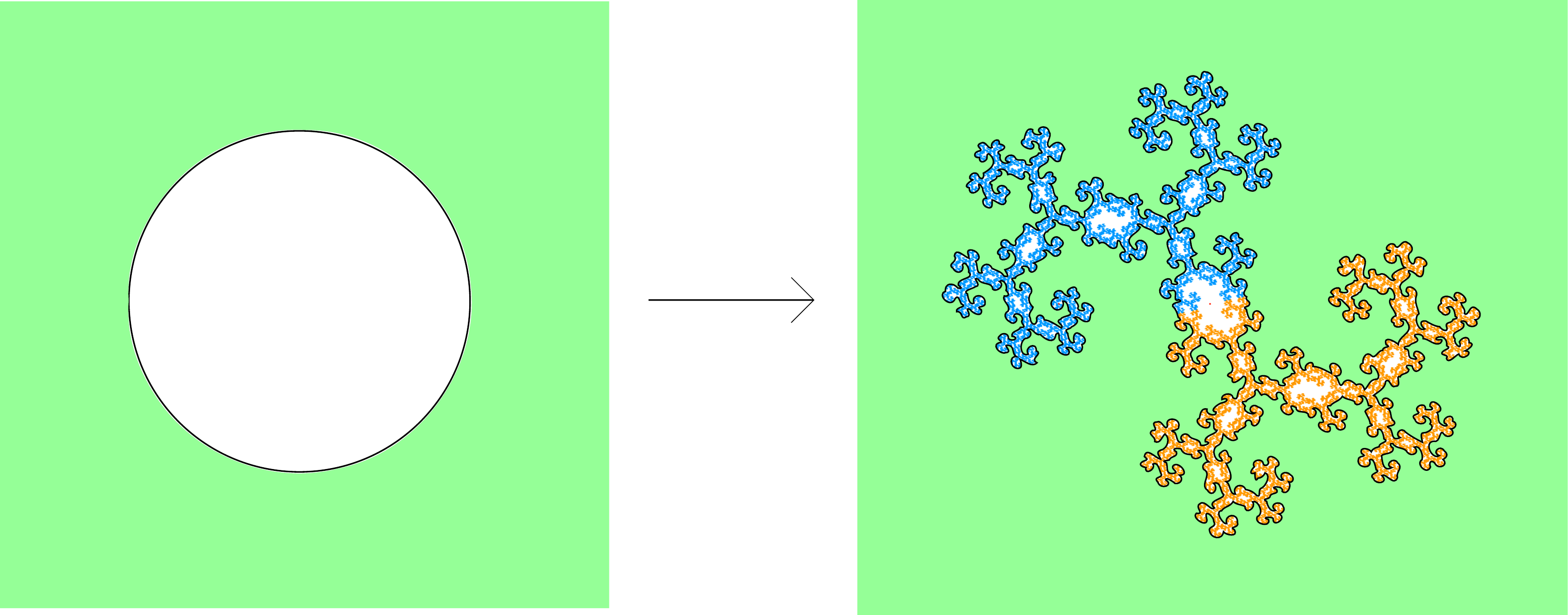}
\caption{The exterior $\C-K$ is uniformized by $\C-\overline{\D}$.}
\label{uniformization}
\end{figure}

The uniformizing map $u$ is unique up to precomposition with a rotation; we will
shortly describe how to make a specific choice. It turns out for $s\in \partial \M$
and for some (but not all) $s$ in the interior of $\M$
that we may divide $S^1$ into two halves, denoted $J_f$ and $J_g$, so that
$\partial u(J_f) \subset f\Lambda$ and $\partial u(J_g) \subset g\Lambda$. 

It is a rather subtle fact that this partition might not be unique, 
but it is possible to make a canonical choice; see \S~\ref{subsection:uniqueness}. 
Having made this choice, we may normalize $u$ so that $J_f$ and $J_g$ are respectively the
left and right half of a circle that has vertically bisected.

We conjecture (Conjecture~\ref{conjecture:A}) that with some
easily understood exceptions (when $s$ is real) providing $s\in \partial \M$
the choice of partition is unique. More explicitly, we conjecture that $\partial K$
decomposes into the images of two connected intervals 
$\partial K \cap fK = u(J_f)$ and $\partial K \cap gK = u(J_g)$
that intersect only at $u(\partial J_f) = u(\partial J_g)$. Our strongest results are conditional
on this conjecture, for which the numerical evidence is eloquent if not yet persuasive; 
see the brief discussion in \S~\ref{subsection:Conjecture_A}.

Here is the key point: since $\partial K$ is contained in $f\partial K \cup g\partial K$,
we may use the uniformizing map $u$ to pull back the dynamics of $\SS$ on $\partial K$ 
(where defined) to a dynamical system on $S^1$.
Specifically, there are intervals $I_f,I_g \subset S^1$ so that $f u(I_f) = u(J_f)$
and $g u(I_g) = u(J_g)$, and we may then obtain a (partially defined) action of $\SS$ on $S^1$
by $u f = f u$ on $I_f$ and $u g = g u$ on $I_g$.

We recall the definition of a lamination of the circle. A {\em leaf} in $S^1$ is
a distinct unordered pair of points in $S^1$. The space of all leaves in $S^1$ is
homeomorphic to an open Mobius band. A {\em lamination} (of the circle) is a closed 
union of leaves, no two of which link in $S^1$.

Now that we have associated a dynamical system on $S^1$ to a parameter $s\in \partial \M$
we may use this dynamics to define the so-called {\em dynamical lamination} $\L$
associated to $s$. Let $\ell$ be the vertical leaf that divides $J_f$ from $J_g$,
and say $\ell$ has {\em depth 0}.

We define $\L$ inductively as follows: if $\mu$ is a leaf in $\L$ of
depth $n$, and both endpoints of $\mu$ are in $I_f$ (resp. both endpoints are in $J_f$) 
then $f\mu$ is in $\L$ (resp. $g\mu$ is in $\L$) and has depth $n+1$. 
Leaves obtained in this way from $\ell$ are said to be of finite depth. Evidently
there are at most $2^n$ leaves of depth $n$. Let $\L$ denote the closure of the
set of finite depth leaves. It turns out (Proposition~\ref{proposition:props_of_L}) that
no two finite depth leaves link in $S^1$, and therefore $\L$ is a
lamination.

Figure~\ref{L_depth_2_with_limit} indicates leaves of $\L$ to depth $2$ in an example.
Corresponding to the leaf $\ell$ there is a crosscut in $K$ dividing
$u(J_f)$ from $u(J_g)$, and the finite depth leaves correspond to 
certain images of this crosscut under elements of $\SS$.

\begin{figure}[htpb]
\centering
\includegraphics[scale=0.2]{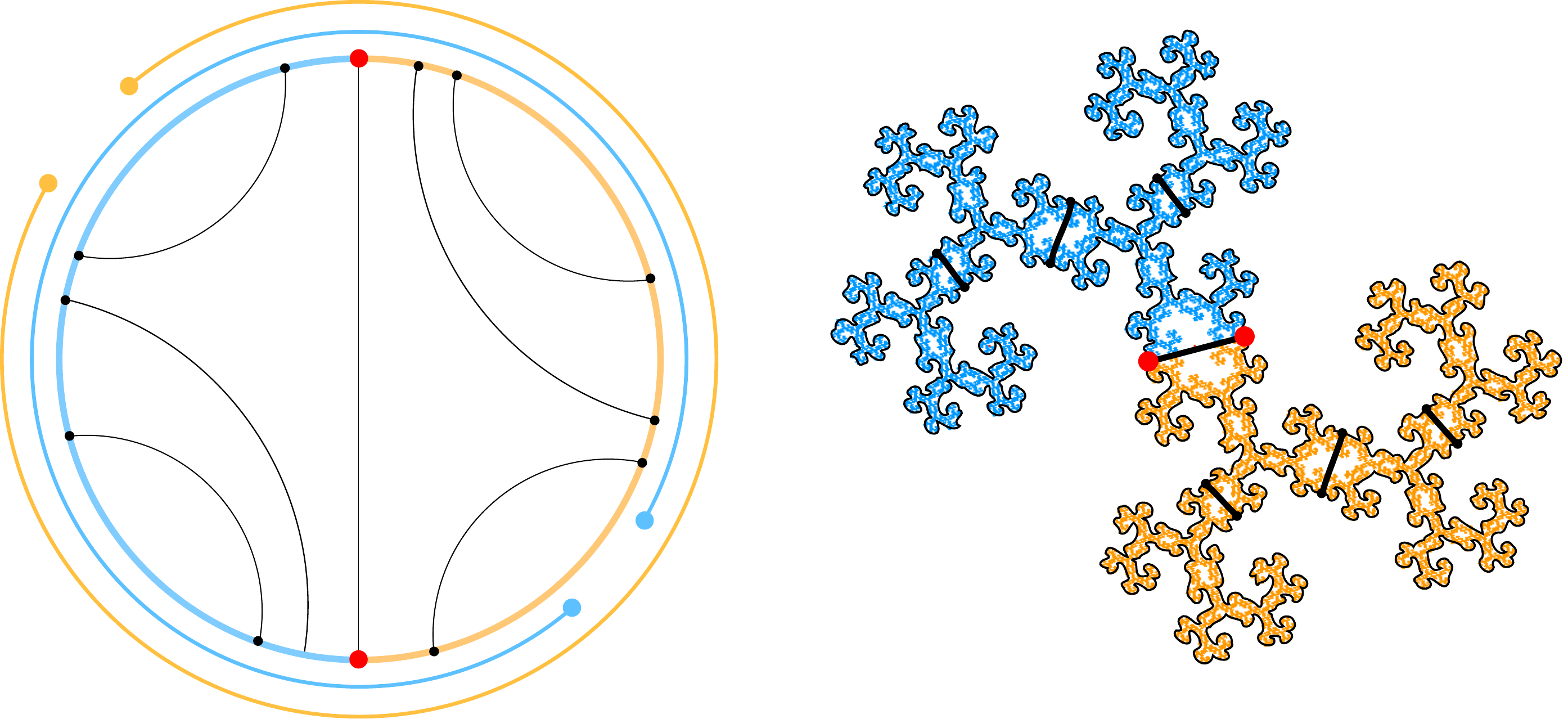}
\caption{Leaves of $\L$ to depth $2$ and corresponding crosscuts of $K$.}
\label{L_depth_2_with_limit}
\end{figure}

The domains of $f^{-1}$ and $g^{-1}$ are the intervals $J_f$ and $J_g$ respectively,
which are very nearly disjoint except on their common endpoints $J_f\cap J_g$.
We may therefore define a single map $H:S^1 \to S^1$ by
$$H(x) := f^{-1}(x) \text{ for } x\in J_f - J_g \text{ and } 
H(x):= g^{-1}(x) \text{ for } x\in J_g - J_f$$
and for $x\in J_f\cap J_g$ define $H(x)$ to be either $f^{-1}(x)$ or $g^{-1}(x)$.
Although it is annoying to deal with many-valued functions, this is the easiest way to
preserve the rotational symmetry, and capture the full dynamics of $f^{-1}$ and 
$g^{-1}$ as far as possible.

\subsection{Statement of Main Theorems}

Our first main theorem (although it is stated and proved second in the body of the paper)
says that the action of $H$ on $S^1$ is topologically conjugate
to a (piecewise, discontinuous) linear action. The following is
a simplified statement of the first main theorem; the more precise statement is 
given in Theorem~\ref{theorem:linearization}.

\begin{linearization_theorem}
Let $s\in \partial \M$. The action of $H$ on $S^1-\lbrace -\pi/2,\pi/2\rbrace$ 
is topologically conjugate to a (piecewise, discontinuous) linear action with constant
slope. Explicitly, there are unique parameters $\lambda \in [1,2]$ and 
$\theta \in \R/2\pi\Z$ and coordinates on $S^1$ so that 
$$H(x) = \lambda x + \theta \text{ for } x \in (-\pi/2,\pi/2) \text{ and }
H(x) = \lambda(x - \pi) + \theta + \pi \text{ for } x \in (\pi/2,3\pi/2)$$
\end{linearization_theorem}

In particular, the dynamics of $H$ on $S^1$ is entirely determined by $\lambda$
and $\theta$. For any pair $\lambda,\theta$ as above (not necessarily arising from some
$s$) let's denote by $E_{\lambda,\theta}:S^1 \to S^1$ the (piecewise, discontinuous)
linear map of this kind. The existence of a {\em semiconjugacy} from $H$ to $E_{\lambda,\theta}$
for some $\lambda,\theta$ follows from general principles from the Milnor--Thurston kneading 
theory. The fact that this is actually a {\em conjugacy} uses some elements of
the theory of Lorenz maps and $\beta$-transformations.

It is a complicated question which pairs $(\lambda,\theta)$ arise for $s\in \partial \M$.
The value $\lambda = 2$ is achieved for uncountably many $s\in \partial \M$.
Such examples were already studied in depth by Bousch \cite{Bousch1,Bousch2} and 
Solomyak \cite{Solomyak}. It turns out that $\lambda = 2$ if and only if $f\Lambda$
and $g\Lambda$ intersect at a single point. In this case $\Lambda$ is a dendrite,
and Eroglu--Rohde--Solomyak show \cite{Eroglu_Rohde_Solomyak} 
that $g:z \to sz+1$ and $\bar{f}:z \to -g(z)$ together are
quasiconformally conjugate to the branches of $h_c^{-1}$ on its (dendritic) Julia
set $c$ for some unique $c\in \partial \Mord$ with $\theta(c) = \theta(s)$. 
Points $s$ with this property are algebraic; the corresponding points $c$ are also
algebraic, and are examples of what are known as Misiurewicz points.

The further study of the space of dynamical systems of the form $S^1,E_{\lambda,\theta}$
is pursued in \S~\ref{section:linear_dynamics}. The connection between the
structure of dynamical cut leaves for $(\lambda,\theta)$ and the Douady--Hubbard--Thurston
lamination is explained in Theorem~\ref{theorem:DHT_lamination}.

We would further like to mention that the use of laminations in the study of IFS, particularly 
their relationship to Rauzy fractals, was developed in some depth in the work of
Victor Sirvent, see especially \cite{Sirvent_Pisot} and \cite{Sirvent_Fibonacci},
and there are connections to the material in this paper that deserve further exploration.

\medskip

How does the dynamics of $H$ on $S^1$ relate to the topology of $\Lambda$?
There is a partially-defined action of $H$ on leaves (i.e.\/ unordered pairs of
distinct points in $S^1$) as follows: if $\mu$ is a leaf with both endpoints in
$J_f$ resp. $J_g$, define $H(\mu) = f^{-1}(\mu)$ resp. $H(\mu) = g^{-1}(\mu)$,
and define $\ell_f: = H(\ell) := f^{-1}(\ell)$ in the special case that $\mu = \ell$ so
that the endpoints are exactly $J_f\cap J_g$. Note
that with this definition, every leaf of $\L$ is in the domain of definition of $H$.

Our second main theorem says that cut points of $\Lambda$ (which are the same as 
cut points of $K$) may be read off from the dynamics of $H$ on leaves. 
Say that a leaf $w\ell$ of $\L$ {\em interacts} with $\ell_f$ if they cross, or
if $w\ell$ separates $\ell_f$ from $\ell$ in $\D$.
Once again, the following is a simplification of the second main theorem; the
precise statement is Theorem~\ref{theorem:interaction}. 

\begin{cut_point_theorem}
Let $s\in \partial \M$ and let $w \in \SS$ have least
length so that $w\ell$ interacts with $\ell_f$, if any (we allow the
possibility that $w$ is empty if $\ell$ crosses $\ell_f$).
\begin{enumerate}
\item{it cannot be the case that $\ell_f$ hides $w\ell$ from $\ell$;}
\item{if $w\ell$ hides $\ell_f$ from $\ell$ then $K$ contains cut points; and}
\item{conditional on Conjecture~\ref{conjecture:A} if $w\ell$ crosses $\ell_f$ 
then $K$ contains no cut points (i.e.\/ it is homeomorphic to a disk).}
\end{enumerate}
\end{cut_point_theorem}

Combining this with the Linearization Theorem, one obtains a practical algorithm
to determine whether $K$ contains cut points for a given $s$. A refinement of this
algorithm lets us describe the structure of the set of cut points 
(conditional on Conjecture~\ref{conjecture:A}) in terms of infinite paths
in a certain directed graph $\IG$ which itself may be constructed algorithmically.
See \S~\ref{subsection:inclusion_graphs}.

\subsection{Acknowledgements}

We would like to thank Bernat Espigul\'e, Toby Hall, Curt McMullen, Steffen Rohde,
Victor Sirvent and Boris Solomyak 
for valuable comments, assistance and encouragement. We would also like to acknowledge
that Claude was used to write and/or modify code used for the numerical investigations
carried out in this paper, and also to create some of the figures that appear in the
Appendix.

\section{Definitions}

Let $\D \subset \C$ denote the open unit disk. In this paper we study a class of
one-dimensional holomorphic dynamical systems called {\em similarity pairs}:

\begin{definition}[Similarity Pair]\label{definition:similarity_pair}
For each $s \in \D$ the {\em similarity pair} associated to $s$
is the semigroup generated by $f:z \to sz-1$ and $g:z \to sz+1$. 
\end{definition}
Strictly speaking, $f$ and $g$ depend on $s$. If we need to emphasize this
dependence we write $f_s$ and $g_s$, though we suppress the subscripts
here and in the sequel whenever this is unambiguous.

We let $\SS$ denote the free semigroup on symbols $f$ and $g$. Thus $s\in \D$
parameterizes a 1 (complex) dimensional family of complex affine actions of $\SS$
on $\C$.

Associated to a parameter $s$ is a canonical subset $\Lambda \subset \C$
called the {\em attractor}, that plays a role in the theory of similarity pairs 
analogous to the Julia set of a rational map, or the limit set of a Kleinian group.

\begin{definition}[Limit Set]\label{definition:limit_set}
For $s\in \D$ the {\em attractor} $\Lambda$ is the unique nonempty compact 
subset of $\C$ with $\Lambda = f\Lambda \cup g\Lambda$. 
\end{definition}

As a set $\Lambda$ may be obtained as follows. First choose any compact set $D_0 \subset \C$
with $f(D_0),g(D_0) \subset D_0$; for instance, $D_0$ could be 
the round disk centered at $0$ of radius $1/(1-|s|)$. Then inductively
define $D_{n+1} = f(D_n) \cup g(D_n)$ and set $\Lambda = \cap_n D_n$.
It is easy to see that $\Lambda$ is equal to the closure of the set of fixed points of
nontrivial elements of $\SS$.

The most important topological property of $\Lambda$ is whether it is connected or
not. Barnsley--Harrington formalized this property in the following definition,
parallel to the definition of the (usual) Mandelbrot set $\Mord$ in terms of the
connectivity of the Julia set $J$ of the quadratic polynomial $h_c:z \to z^2+c$:

\begin{definition}[Barnsley--Harrington Mandelbrot set]
The {\em Barnsley--Harrington Mandelbrot set} $\M$
is the set of $s\in \D$ for which $\Lambda$ is connected.
\end{definition}

The set $\M$ is closed as a subset of the open unit disk $\D$; the union $\M \cup S^1$
is closed (and therefore compact) as a subset of $\C$. We define the boundary
$\partial \M$ to be its frontier in $\D$, i.e.\/ the set of points in $\M$ that
are in the closure of $\D - \M$ in $\D$. We are concerned in this paper with
the structure of $\partial \M$ as a set, and with the topology of $\Lambda$ 
and the dynamics of $\SS$ for $s\in \partial \M$.

It is easy to see that $\Lambda$ is either connected or a Cantor set (at least
if $|s|>0$), and if it is connected, it is locally connected. It will be very
convenient to obtain a priori estimates on $|s|$ for $s \in \M$ and $s \in \partial \M$.
The following is due to Bousch \cite{Bousch1} Proposition~2:

\begin{proposition}[Boundary estimate]\label{proposition:Bousch_bound}
The set $\M$ contains the entire annulus $1/\sqrt{2} \le |s|<1$ and is 
contained in the annulus $1/2 \le |s| < 1$. 
Consequently if $s\in \partial \M$ then $1/2 \le |s| \le 1/\sqrt{2}$.
\end{proposition}

In fact, it turns out to be surprisingly useful to improve the inequality
$|s| \le 1/\sqrt{2}$ very slightly. This improvement is delicate, since in fact there
are (exactly two) points $s \in \partial \M$ with $|s| = 1/\sqrt{2}$.

\begin{theorem}[Strict inequality]\label{theorem:strict_inequality}
If $s \in \partial \M$ then either $s = \pm i/\sqrt{2}$ or $|s| < 1/\sqrt{2}$.
\end{theorem}
We postpone the proof to Appendix~\ref{section:inequality_proof}.

\subsection{Symmetries of $\M$}\label{subsection:symmetries}

The set $\M$ is invariant under $s \to \bar{s}$. Indeed, 
$$f_{\bar{s}}\bar{z} = \overline{f_s(z)} \text{ and } g_{\bar{s}}\bar{z} = \overline{g_s(z)}$$
In particular, $\Gamma_s$ and $\Gamma_{\bar{s}}$ are conjugate, and $\Lambda_{\bar{s}}$ 
is the complex conjugate of $\Lambda_s$.

The set $\M$ is also invariant under the involution $s \to -s$. This is more
subtle, since
$$f_{-s}(-z) = f_s(z) \text{ and } f_{-s}(z) = -g_s(z)$$
and similarly for $g$. Consequently, for any word $u\in \SS$ we have
$$u_{-s}(z) = (-1)^{|u|} v_s(z)$$
where $v$ is the word obtained from $u$ by changing every second letter of $w$
from $f$ to $g$ or vice versa, starting at the last letter. For example
$$(ggffgfg)_{-s}(z) = - (fggffff)_s(z)$$
Thus as sets $\Lambda(s)$ and $\Lambda(-s)$ are equal, although the dynamics
of $\SS$ on these sets is not straightforwardly conjugate.

Suppose $s \in \M$, equivalently that $f\Lambda$ intersects $g\Lambda$. It follows
that at least one of $ff\Lambda$ and $fg\Lambda$ intersects $g\Lambda$. Consequently,
by replacing $s$ by $-s$ if necessary, we may assume that $fg\Lambda$ intersects
$g\Lambda$.

\section{Dynamics on the boundary of the filled set}\label{section:topological_model_for_boundary}

\subsection{The filled limit set}

\begin{definition}\label{def:filled_limit_set}
Suppose $s \in \M$ so that $\Lambda$ is connected. 
The {\em filled limit set} $K$ is the complement of the 
unique unbounded component of $\C - \Lambda$.
\end{definition}
Another way to define $K$ is to say it is the minimal compact
simply-connected subset of $\C$ that contains $\Lambda$.
We denote by $\partial K$ the boundary of $K$. 
Note that $fK\cup gK \subset K$, and that 
$\partial K \subset f\partial K \cup g\partial K$. Also note that if $f(p) \in
\partial K$ then $p\in \partial K$, and similarly for $g$. 

If $|s|<1/\sqrt{2}$, the Hausdorff dimension of $\Lambda$ is strictly less than 2,
so that $\Lambda$ has no interior. In this case, the interior of $K$ (which may 
be empty or not) contains an open dense subset which is in $\C-\Lambda$. 
On the other hand, if $|s|>1/\sqrt{2}$ then $s$ is in the interior of $\M$ by 
Proposition~\ref{proposition:Bousch_bound}.
Thus, if $s\in \partial \M$ but $\Lambda$ has
interior, then $|s|=1/\sqrt{2}$ which occurs if and only if $s=\pm i/\sqrt{2}$ by
Theorem~\ref{theorem:strict_inequality}. 

\begin{example}\label{example:rectangle}
The value $s = i/\sqrt{2}$ and its complex conjuate are exceptional, and must be
handled as a special case in some arguments in the sequel. Fortunately, $\Lambda$
and the dynamics of $\Gamma$ are very easy to describe for this $s$.
 
The set $\Lambda$ is the rectangle 
$$\Lambda = \lbrace z = x+iy \text{ such that } -2 \le x \le 2 \text{ and }
-\sqrt{2} \le y \le \sqrt{2} \rbrace$$
The set $f\Lambda$ is the `left side' of this rectangle, i.e.\/ the subset of
points with non-positive real part, and $g\Lambda$ is the `right side', i.e.\/
the subset of points with non-negative real part.
\end{example}

\begin{lemma}[Components of $K-\Lambda$]\label{lemma:components_of_complement}
Every component $U$ of $K-\Lambda$ is homeomorphic to a disk, and is of the form
$wV$ for some $w\in \SS$ and some component $V$ of $K-\Lambda$ for which $\partial V$ is neither
entirely contained in $f\Lambda$ nor in $g\Lambda$.
\end{lemma}
\begin{proof}
Since $\Lambda$ is always closed, and is connected for $s\in \M$, 
it follows that every component of $K-\Lambda$
is homeomorphic to a disk. 

If $U$ is a component of $K-\Lambda$ and $\partial U$ is entirely contained in
$f\Lambda$ resp. $g\Lambda$ then $f^{-1}(U)$ resp. $g^{-1}(U)$ is also a component of $K-\Lambda$.  
Since $f^{-1}$ multiplies the diameter of $U$ by $|s|^{-1} > 1$, we can take this 
inverse image only finitely many times before it fails to be the case that $\partial U$ 
is entirely contained in $f\Lambda$ or $g\Lambda$.  Thus
the components of $K-\Lambda$ are all of the form $wU$, where $w$ is a finite 
word in $f$ and $g$, and $U$ is some component of $K - \Lambda$ with
$\partial U$ contained neither entirely in $f\Lambda$ nor in $g\Lambda$.
\end{proof}

\subsection{Uniformization}

The complement of $K$ in $\hat{\C}$ is simply-connected, and since $K$ contains
more than one point, $\hat{\C}-K$ is conformally equivalent to the unit disk.
We choose some conformal uniformization $u:\C -\overline{\D} \to \C-K$ which takes
$\infty$ to $\infty$.

Bousch \cite{Bousch2}, \S~2.4 shows that if $\Lambda$ is connected, it is path
connected and locally connected, so that $u$ extends continuously
to a surjection $u:S^1 \to \partial K$, well-defined up
to the ambiguity of precomposition with a rotation. In particular, Lebesgue
measure on $S^1$ pushes forward under $u$ to a measure on
$\partial K$ of total mass $2\pi$ that we call {\em harmonic measure}. 

\begin{lemma}[Point preimages]\label{lemma:multiple_preimages_implies_cut_point}
For each $p\in \partial K$ the preimage $u^{-1}(p) \in S^1$ is compact,
totally disconnected, and nonempty, and it contains more than one point 
if and only if $p$ is a cut point of $K$.
\end{lemma}
\begin{proof}
These are well-known properties of the continuous extension of uniformization
maps, given that $K$ is locally connected. See e.g.\/ Pommerenke \cite{Pommerenke} 
Proposition~2.5 and Theorem~9.19.
\end{proof}

The involution $\iota:z \to -z$ on $\C$ 
takes $K$ (and hence also $\partial K$) to
itself, interchanging the images of $fK$ and $gK$. It extends to a
conformal symmetry of $\C-K$, and to rotation by $\pi$ on $S^1$. By abuse
of notation, we denote by $\iota:S^1 \to S^1$ rotation through $\pi$, so that
$u$ semiconjugates $\iota$ on $S^1$ to $\iota$ on $\partial K$.

\subsection{A partition of $S^1$}

For any $s \in \M$ we have $\partial K \subset f\partial K \cup g\partial K$ 
and therefore we may write
$\partial_f K: = \partial K \cap f\partial K$ and $\partial_g K: = \partial K \cap g\partial K$
and observe that $\partial_f K \cup \partial_g K = \partial K$. 

Define $D_f: = u^{-1} \partial_f K$ and $D_g: = u^{-1} \partial_g K$. These subsets of $S^1$
are both closed, their union is all of $S^1$, and $\iota D_f = D_g$; in particular,
each of them has Lesbesgue measure at least $\pi$. 
We shall show that if $s \in \partial \M$, the set $D_f$ contains a (necessarily unique)
connected component of length at least $\pi$. 

\begin{lemma}[Partition exists]\label{lemma:partition_exists}
For $s\in \partial \M$ the subsets $D_f$ and $D_g$ each contain unique
connected components of length at least $\pi$. Consequently it is possible to 
choose closed connected intervals $J_f\subset D_f$ and $J_g \subset D_g$ each
of length exactly $\pi$, and whose union is $S^1$.
\end{lemma}
\begin{proof}
The hypothesis that $s\in \partial \M$ enters in the following way.
Our proof depends on the notion of a {\em trap}, introduced in \cite{Calegari_Koch_Walker},
and alluded to in the proof of Theorem~\ref{theorem:strict_inequality} above.
A trap is a set of four distinct points in cyclic order $x^-,y^-,x^+,y^- \subset \partial K$
for which $x^\pm \in fK - gK$ and $y^\pm \in gK - fK$. The existence of a trap
for $K$ certifies that $s$ is in the interior of $\M$, see
\cite{Calegari_Koch_Walker} Proposition~7.1.6. Thus our goal is to show that either
we can find the desired subsets $J_f \subset D_f$ and $J_g \subset D_g$, or we can
find a trap.

The first remark is that if $D_f$ contains a component of length at least $\pi$ it
is unique, since $S^1$ has total length $2\pi$, so we just need to prove
existence. Let's suppose $A$ is a component of $D_f$ of maximal length.
If $\length(A)\ge \pi$ there is nothing to prove, so we suppose $\length(A) < \pi$.
We claim that there must be some $p^+ \in A - D_g$. For, otherwise,
$A$ is contained in a component $B$ of $D_g$ of with $\length(B) > \length(A)$,
and therefore $\iota B$ is a component of $D_f$ with $\length(\iota(B)) > \length(A)$
contrary to the maximality of $\length(A)$.

Since $\length(A) < \pi$ we may choose points $q^\pm \in D_g - D_f$ 
near $\partial A$ so that the oriented interval $[q^-,q^+] \subset S^1$ contains
$p^+$ and has length $<\pi$; in particular, $[q^-,q^+]$ is disjoint from the
oriented interval $[\iota q^-,\iota q^+]$. Define $p^-:=\iota q^-$ and observe
that $p^- \in D_f - D_g$. Note that the four points $p^-, q^-,p^+,q^+$ appear
in circular order in $S^1$, and that they alternate between points
of $D_f - D_g$ and $D_g - D_f$.

Now define $x^\pm = u(p^\pm)$ and $y^\pm = u(q^\pm)$. Observe that by 
construction, $x^\pm \in \partial_f K - \partial_g K$ and $y^\pm \in \partial_g K - \partial_f K$.
We will show that $x^\pm,y^\pm$ form a trap. 

First of all we claim that $x^-$ and $x^+$ are in distinct components of $\partial_f K -
\partial_g K$, and similarly for $y^-$ and $y^+$. 
For $p^-$ and $p^+$ are in distinct components of $D_f - D_g$, so if 
$x^-$ and $x^+$ are in the same component of $\partial_f K -
\partial_g K$ then $K$ must contain a cut point in the complement of $gK$, which
necessarily separates $gK$ into at least two components, each containing one of $y^-$ and
$y^+$. But $gK$ is connected if $K$ is, and the claim is proved.

But now $x^-,y^-,x^+,y^+$ form a trap, so that $s$ is an interior point of $\M$,
contrary to hypothesis.
\end{proof}

Once we have chosen $J_f$ and $J_g$ we may normalize the coordinates on $S^1$ so
that $J_f$ and $J_g$ are divided by the vertical bisector of $S^1$, with $J_f$ on
the left and $J_g$ on the right.

\subsection{Uniqueness of $J_f$}\label{subsection:uniqueness}

It is an awkward detail that the intervals $J_f$ and $J_g$ guaranteed by 
Lemma~\ref{lemma:partition_exists} are not necessarily
unique. We may make a canonical choice by choosing $J_f$ and $J_g$ to be
the leftmost intervals of length $\pi$ in the components of $D_f$ and $D_g$ that contain
them. 

\begin{example}[Real $s$]\label{example:real}
If $s$ is real and satisfies $|s| \in [1/2,1]$ then $\Lambda$ is a real interval.
The subintervals $f\Lambda$ and $g\Lambda$ have a nonempty intersection when
$|s| > 1/2$ so that the intervals $J'_f$ and $J_f$ are different.
\end{example}

Note that the converse of Lemma~\ref{lemma:partition_exists} is not true,
in the sense that there are $s$ in the interior of $\M$ for which 
$D_f$ and $D_g$ still contain (unique) connected components of length at
least $\pi$.

\begin{definition}[Splittable]\label{definition:splittable}
We say $s\in \M$ is {\em splittable} if $D_f$ and $D_g$ have (unique)
connected components of length at least $\pi$.
\end{definition}

Many results that we prove in the sequel apply to all splittable $s$,
and some only to splittable $s$ satisfying additional conditions.

\begin{example}[Interior $s$]\label{example:interior}
Let $s \sim 0.4334 + 0.5258i$ be a root of $z^4 + z^3 + z^2 - z + 1$. Then $s$ is
in the interior of $\M$, and the set $D_f$ is a connected interval of 
length $>\pi$. See Figure~\ref{circle_dividing_false}.

\begin{figure}[htpb]
\centering
\includegraphics[scale=0.2]{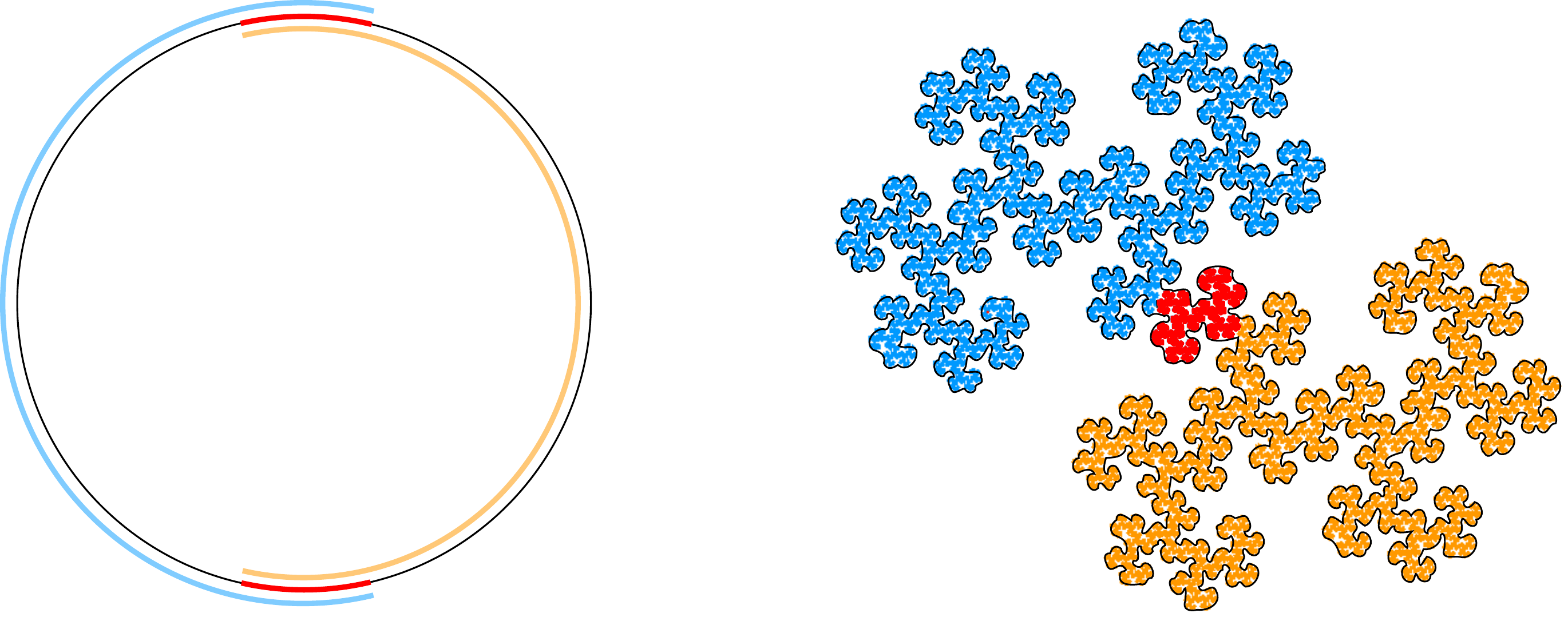}
\caption{For $s \sim 0.4334 + 0.5258i$ a root of $z^4 + z^3 + z^2 - z + 1$ 
in the interior of $\M$ the set $D_f$ is an interval of length $>\pi$.}
\label{circle_dividing_false}
\end{figure}

If $\Lambda$ is the limit set, then $fgfff\Lambda = gfggg\Lambda$, and these common
images of $\Lambda$ are equal to the entire intersection $f\Lambda \cap g\Lambda$. 
Note that $\Lambda$ is a dendrite in this example.
\end{example}

On the other hand, extensive computer experiments suggest the following conjecture:

\begin{conx}\label{conjecture:A}
If $s \in \partial \M$ is not real, then $D_f = J_f$.
\end{conx}

Our results in the sequel are strongest for $s$ satisfying 
Conjecture~\ref{conjecture:A}. We will discuss this conjecture and the
evidence for it in \S~\ref{subsection:Conjecture_A}.

\subsection{Dynamics on $S^1$}

There is a partially-defined action of $\SS$ on $\partial K$. We would like to transport
this action to $S^1$ via the map $u$. Suppose $x \in \partial_f K$ and let
$y = f^{-1}(x) \in \partial K$. If $\alpha$ is a properly embedded ray in $\C - K$
that lands at $x$, then $f^{-1}(\alpha)$ is a properly embedded ray in $\C - K$
that lands at $y$. The arcs $\alpha$ and $f^{-1}(\alpha)$ represent prime ends of 
$\C - K$ and therefore points $p,q \in S^1$ so that $u(p) = x$ and $u(q) = y$.
This lets us define an action of $f^{-1}$ resp. $g^{-1}$ on $D_f$ resp. $D_g$.
The action of each of $f^{-1}$ and $g^{-1}$ is continuous on its 
domain of definition. It might fail to be injective if there is a point
$x \in \partial_f K$ and two arcs $\alpha,\beta$ landing at $x$ and representing
different prime ends of $\C-K$ map to $f^{-1}(\alpha),f^{-1}(\beta)$ representing
the same prime end of $\C-K$. This can only be because $\alpha$ and $\beta$ are
separated by part of $gK - fK$. In particular, the prime ends $\alpha$ and $\beta$
must represent points in the frontier of $D$. Thus in particular $f^{-1}$ (and
likewise $g^{-1}$) is injective on the interior. 

Define $I_f$ to be the closure of $f^{-1}\inte(J_f)$ and define $I_g$
likewise. The image $f^{-1}\inte(J_f)$ is an open interval, so that
$I_f$ is either a closed interval, or it is all of $S^1$; in the latter case we
think of it as a `degenerate interval' and by abuse of notation we let $\partial I_f$ 
denote the single point $S^1 - f^{-1}\inte(J_f)$. If $I_f$ and $I_g$ are
not degenerate, we obtain well-defined and continuous inverses 
$f:I_f \to J_f$ and $g:I_g \to J_g$ so that $uf = fu$ and $ug = gu$ where defined.
If $I_f$ and $I_g$ are degenerate, then $f$ resp. $g$ is not defined on $\partial I_f$
resp. $\partial I_g$.

\begin{lemma}[Contractions]\label{lemma:expanding_map}
The maps $f,g$ are strictly distance decreasing 
for Lebesgue measure on $S^1$ in their domain of definition.
\end{lemma}
\begin{proof}
The domain $\hat{\C} - K$ includes into $\hat{\C} - fK$ and
therefore is strictly distance increasing in the hyperbolic metric. Thus
the harmonic measure on $f \partial K$ in $\hat{\C} - fK$ strictly dominates
the harmonic measure on $f \partial K$ in $\hat{\C} - K$; another way to
see this is that the harmonic measure of a subset $A\subset \partial K$
is the probability that Brownian motion on $\C$ starting at infinity intersects 
$\partial K$ first in $A$. Some set of walks
landing at any $p \in f \partial K$ first enter $K - fK$; the ratio
of the two measures is the fraction of walks that do not. Thus this ratio is
strictly less than $1$; the proof follows.
\end{proof}

\begin{lemma}[Degenerate intervals]\label{lemma:degenerate}
If $I_f$ and $I_g$ are degenerate (i.e.\/ they are both equal to all of $S^1$), 
then $u \partial J_f = 0$ which is a cut point of $\Lambda$. Furthermore in this
case, $\Lambda$ is a dendrite; i.e.\/ a compact path-connected subset which
contains no embedded Jordan curves.
\end{lemma}
\begin{proof}
If $I_f$ and $I_g$ are degenerate, then $f$ and $g$ are
defined on a dense open connected subset of $\partial K$, and they extend uniquely
to continuous functions defined on all of $\partial K$. Since $uf = fu$ on the interior
of $I_f$, it follows that $f u \partial I_f = u \partial J_f$ is a single point,
and since $u$ commutes with $\iota$ this point is fixed by $\iota$ and is therefore
equal to $0$. Since $0$ is in $\partial K$ it is a cut point of $K$ by symmetry,
and by construction it separates $f\Lambda$ from $g\Lambda$, i.e.\/ $fK$ and
$gK$ intersect only at $0$.  
 
We claim that $K$ (and therefore $\Lambda$) is a dendrite. 
For if not, let $A$ be a component of the interior of $K$ of 
diameter $>D-\epsilon$ where $D$ is the supremum of the diameters 
of components of the interior of $K$. Without loss of generality $A\subset fK$.
But then $f^{-1}A$ is a component of the interior of $K$ with diameter
at least $\diam(A)\cdot |s|^{-1}$, contrary to the definition of $A$
providing $\epsilon$ is small enough (depending on $|s|$).
\end{proof}

The symmetry of $\M$ under $s \to -s$ lets us make a simplifying assumption about
the dynamics of $\Gamma$ on $S^1$:

\begin{lemma}[Endpoint intersects]\label{lemma:simplifying_assumption_intervals}
Suppose $s \in \partial \M$. Then after replacing $s$ by $-s$ if necessary,
we may assume that at least one endpoint of 
$I_f$ lies in $J_g$ and similarly with $f$ and $g$ interchanged.
\end{lemma}
\begin{proof}
Recall from the discussion in \S~\ref{subsection:symmetries}
that $(fg)_{-s}(z) = (ff)_s(z)$. Thus replacing $s$ with $-s$ replaces $I_f$
with $\iota I_f$ (and similarly for $I_g$). The claim follows.
\end{proof}

In the sequel we will often prove our theorems under the assumption
that the conclusion of Lemma~\ref{lemma:simplifying_assumption_intervals} holds.
The generalization to arbitrary $s$ is typically omitted where routine.

\begin{figure}[htpb]
\centering
\includegraphics[scale=0.19]{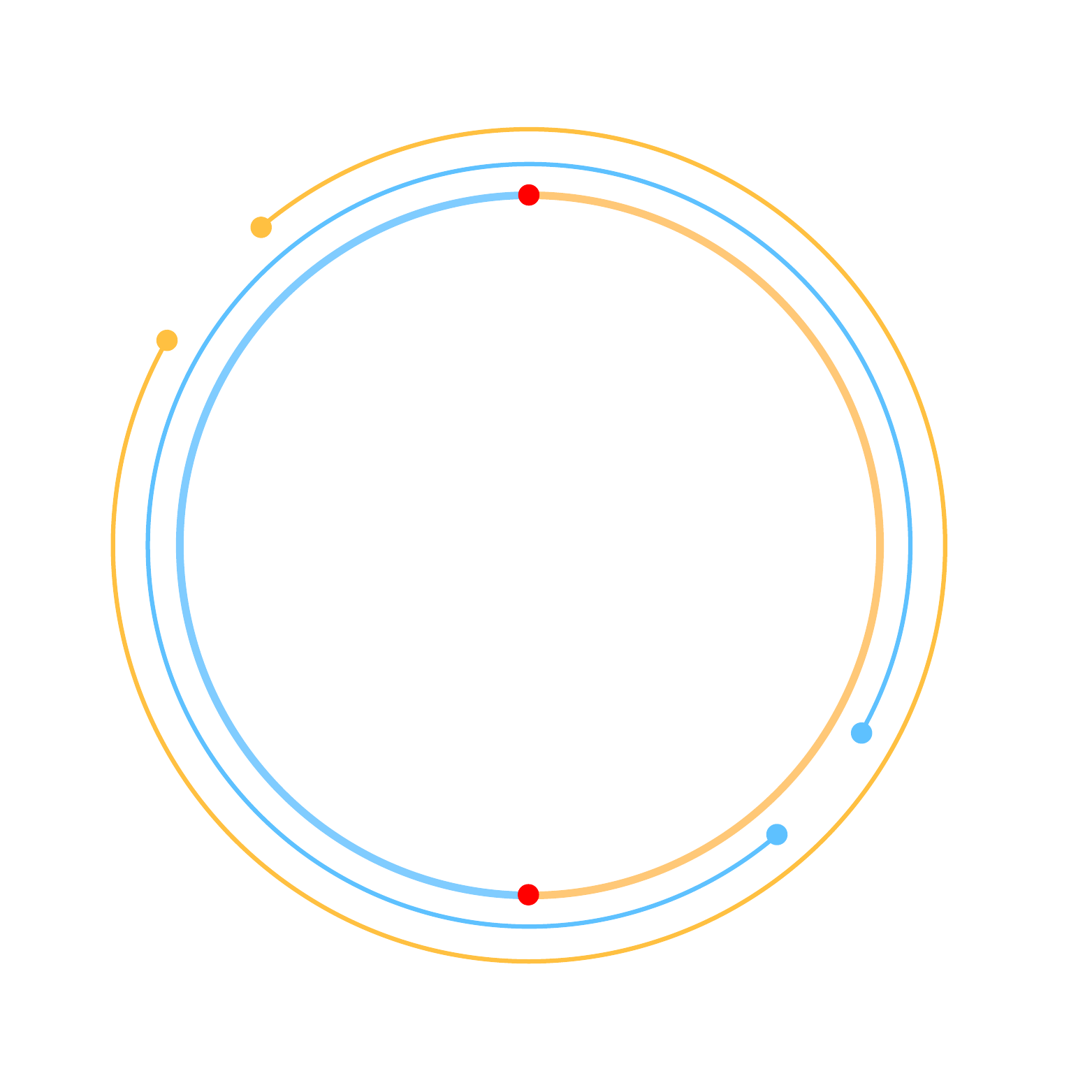}
\includegraphics[scale=0.44]{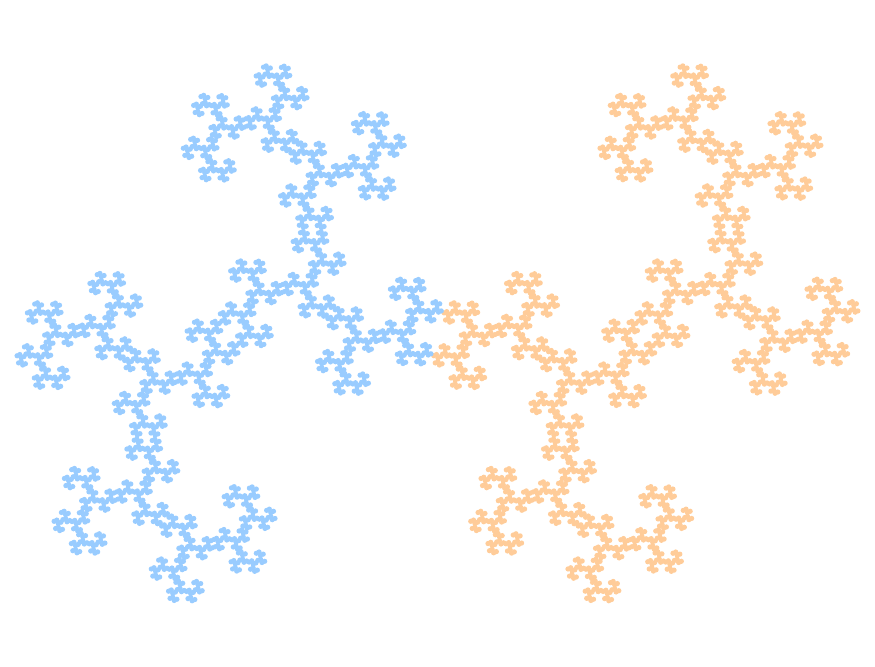}
\caption{A typical configuration of $I_f,I_g,J_f,J_g$
associated to $s \sim 0.399 + 0.498i \in \partial \M$.
The intervals $I_f$ and $J_f$ are in blue, and $I_g$ and $J_g$ are in orange.
The vertical bisector of $S^1$ divides $J_f$ from $J_g$ with $J_f$ on the left.}
\label{Is_and_Js}
\end{figure}

\begin{example}
Lemma~\ref{lemma:expanding_map} implies that the
intervals $I_f$ and $I_g$ both have length strictly $>\pi$.
Figure~\ref{Is_and_Js} shows a typical configuration of $I_f,I_g,J_f,J_g$
associated to $s \sim 0.399 + 0.498i \in \partial \M$.
\end{example}

\begin{remark}
There is some numerical evidence that the conclusion of 
Lemma~\ref{lemma:simplifying_assumption_intervals} holds for $s\in \partial \M$
whenever $\arg(s) \in [-\pi/2,\pi/2]$. We do not make use of this observation in
the sequel.
\end{remark}

\section{Laminations and cut points}
\label{section:laminations_and_cut_points}

\subsection{The dynamical lamination}\label{subsection:dynamical_lamination}

\begin{definition}\label{definition:lamination}
A \emph{leaf} in the circle $S^1$ is an unordered pair of 
distinct points. Two leaves $\mu,\mu'$ are said to {\em cross}
if they are disjoint as subsets of $S^1$, and link in $S^1$. A {\em lamination} $\L$
is a set of leaves, no two of which cross.
\end{definition}

The space of ordered distinct points in $S^1$ is an open annulus $S^1 \times S^1 - \Delta$
where $\Delta$ is the diagonal, and the space of unordered distinct points is the
open M\"obius band $M$ obtained from this annulus by quotienting by the involution that
exchanges the $S^1$ factors.

The definition of lamination that we give in Definition~\ref{definition:lamination} 
is nonstandard, in that the set of leaves is not required to be closed in this
M\"obius band. Note that if $\L\subset M$ is a collection of leaves no two of which 
cross, then the same is true of the closure $\overline{\L} \subset M$; thus if $\L$
is a lamination in our sense, $\overline{\L}$ is a lamination in the usual sense.

If we think of $S^1$ as the ideal boundary of the hyperbolic
plane in the Poincar\'e disk model, then each leaf $\mu \in M$ is the pair of
endpoints of a unique geodesic $\gamma_\mu$ in $\D$. Distinct leaves $\mu,\mu'$
do not cross if and only if the geodesics $\gamma_\mu,\gamma_{\mu'}$ are {\em disjoint}
(note that the geodesics are subsets of the open disk and do not contain their 
`endpoints' which lie at infinity). One sometimes distinguishes between a 
{\em lamination of $S^1$}, meaning a lamination in our sense, and a {\em lamination of
the disk}, meaning a collection of embedded disjoint geodesics in $\D$.
Laminations of $S^1$ are hard to draw, so one usually draws the associated lamination
of the disk; in the sequel we will move back and forth between these two points of 
view as convenient.  See Figure~\ref{lamination}.

\begin{figure}[htpb]
\centering
\includegraphics[scale=0.5]{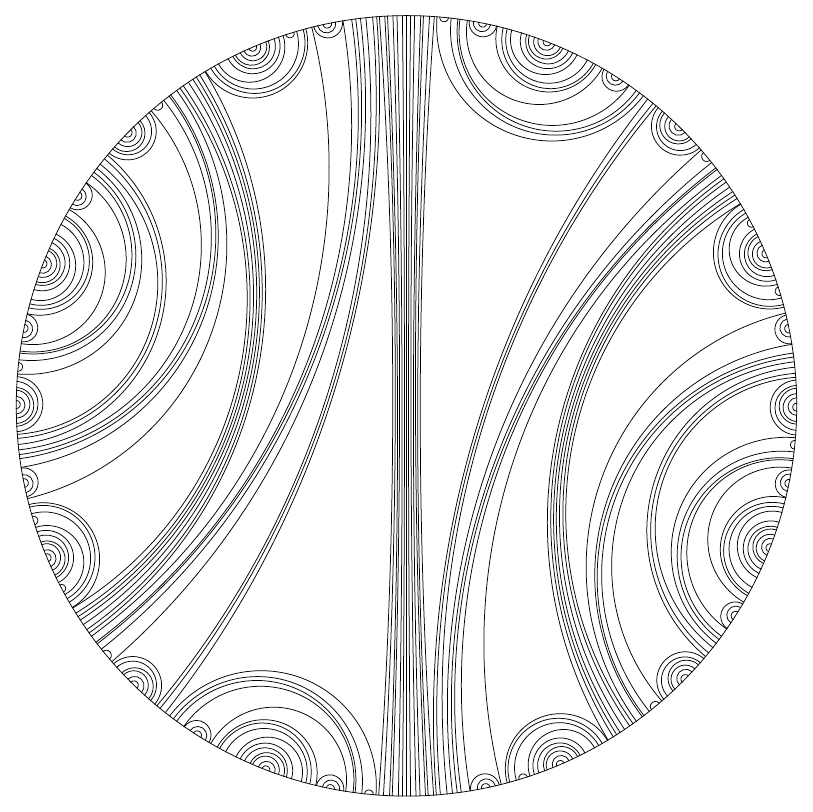}
\caption{A lamination of the circle is indicated by drawing the corresponding lamination of
the disk.}
\label{lamination}
\end{figure}

Associated to $s\in \partial \M$ we have constructed a circle with a partially
defined action of our semigroup $\SS$ on $S^1$, namely we have constructed intervals
$I_f,I_g$ and $J_f,J_g$ so that $f:I_f \to J_f$ and $g:I_g \to J_g$ are defined. 

If $\mu$ is a leaf, and $A \subset S^1$ then we say that $\mu$ is {\em contained in}
$A$ if the points of $\mu$ are in $A$. If $h:A \to B$ is a map between subsets of
$S^1$, and $\mu:=\lbrace x,y\rbrace$ is a leaf contained in $A$, then we define
$h\mu$ to be the leaf $\lbrace h(x),h(y)\rbrace$; note that $h\mu$ is contained in $B$.
If $\mu$ does not lie in $A$, then $h\mu$ is undefined.

The maps $f$ and $g$ have domains $I_f$ and $I_g$ respectively. Thus for a
leaf $\mu$ the leaf $f\mu$ resp. $g\mu$ is defined if and only if $\mu$ is contained
in $I_f$ resp. $I_g$. In particular, every leaf $\mu$ 
determines a suffix-closed subset $\SS_\mu$ 
of elements $w\in \SS$ (maybe empty) for which $w\mu$ is defined. `Suffix-closed' here
means that if $uv \in \SS_\mu$ for words $u,v\in \SS$ then $v\in \SS_\mu$.
The set of leaves $w\mu$ for $w$ ranging over $\SS_\mu$
is called the {\em orbit} of $\mu$.

\begin{definition}[Dynamical lamination]\label{definition:dynamical_lamination}
For $s \in \partial \M$ construct $I_f,I_g,J_f,J_g$ as in 
\S~\ref{section:topological_model_for_boundary}, where we normalize
$J_f$ and $J_g$ to lie to the left and the right (respectively) of the vertical
bisector of $S^1$. Let $\ell$ be the leaf which consists of the common endpoints
of $J_f$ and $J_g$, i.e.\/ $\ell = \lbrace \pi/2, 3\pi/2\rbrace$, and let $\L(s)$
(or just $\L$ if $s$ is understood)
be the orbit of $\ell$ under $\SS_\ell$. We call $\L$ the
{\em dynamical lamination}, although we have not yet justified the name.
\end{definition}
Note that as a set of leaves, $\L = f\L \cup g\L \cup \ell$.
If $\ell' \in \L$ is equal to $w\ell$ for some $w\in \SS_\ell$ with $|w|=n$ minimal,
we say $\ell'$ has {\em depth $n$}. In particular, $\ell$ has depth zero, and
every $\ell' \in \L$ has some well-defined (finite) depth.

\begin{proposition}[Lamination]\label{proposition:props_of_L}
The set of leaves $\L$ is a lamination. It is infinite if and only if $\ell$ is contained
in $I_f$, which happens if and only if $\ell$ is contained in $I_g$. In particular,
either $\L$ is infinite, or it consists entirely of the single leaf $\ell$.
\end{proposition}
\begin{proof}
The leaf $\ell$ is contained in $I_f$ if and only if it is contained in $I_g$ by
the symmetry $\iota$.

First we show that $\L$ is a lamination. Suppose not, so that two leaves cross.
Leaves of $\L$ are all of the form $w\ell$ for $w\in \SS_\ell$, so suppose
$u\ell$ crosses $v\ell$ for some pair of words with $|u|+|v|$ minimal. The
words $u$ and $v$ can have no common prefix, since if $u=au'$ and $v=av'$ then
$u'\ell$ and $v'\ell$ already cross. But if $w\in \SS_\ell$ is any word that starts
with $f$ resp. $g$, the leaf $w\ell$ is contained in $J_f$ resp. $J_g$
and lies to the left resp. right of $\ell$. The claim follows.

It remains to show that $\L$ is infinite if and only if $\ell$ is contained in
$I_f$. This condition is necessary, since if $\ell$ is not contained in $I_f$
(and therefore is not contained in $I_g$) then $\SS_\ell$ consists only of the empty word,
and $\L = \lbrace \ell\rbrace$. So suppose $\ell$ is contained in $I_f$.

One of the following must hold:
\begin{enumerate}
\item $J_f \subset I_f$ and $J_g \subset I_g$
\item $J_f \subset I_g$ and $J_g \subset I_f$
\end{enumerate}
The argument is almost identical in either case; for simplicity we will assume
we are in the first case. Since $J_f \subset I_f$ it follows that $f^n \in \SS_\ell$ 
and similarly $g^n$. The map $f$ is a strict contraction by Lemma~\ref{lemma:expanding_map}
and therefore the leaves $f^n\ell$ are all distinct. The proof follows.
\end{proof}

\subsection{Limit leaves} 

Recall our convention that the dynamical lamination $\L$ is not necessarily closed
in the space of unordered distinct points in $S^1$. Let $\overline{\L}$ denote its
closure (so that $\overline{\L}$ is an honest lamination of the circle in the usual 
sense of the word).

\begin{definition}[Limit leaf]\label{definition:limit_leaf}
A leaf $\ell' \in \overline{\L}$ is a {\em limit leaf} if there is a sequence of
words $w_n \in \SS_\ell$ with $|w_n| \to \infty$ so that $w_n\ell \to \ell'$
(we allow the possibility that $w_n\ell = \ell'$ for infinitely many $n$).
\end{definition}

For example, every leaf of $\overline{\L} - \L$ is a limit leaf.

\begin{lemma}[Preperiodic is periodic]\label{lemma:preperiodic_is_periodic}
Suppose $u,v$ are distinct words in $\SS_\ell$ with $u\ell = v\ell$. Then one
of the words $u,v$ has a nontrivial suffix $t$ so that $t\ell = \ell$.
Consequently every leaf of $\overline{\L}$ is a limit leaf.
\end{lemma}
\begin{proof}
Let's suppose without loss of generality that $u,v$ as above are chosen to
minimize $|u|+|v|$. Since $u$ and $v$ are distinct, at least one of them is nontrivial; 
without loss of generality, $u$. If $u\ell = \ell$ then we may take $t=u$.
Otherwise $u\ell$ and $v\ell$ both lie to the left, or both to the right of
$\ell$, so that the first letters of $u$ and $v$ must agree. In particular,
$u=au'$ and $v=av'$ for some nontrivial common prefix $a$, and now $|u'|+|v'| < |u|+|v|$
contrary to the choice of $u,v$.

If $t\ell = \ell$ for some nontrivial $t$ then $wt^n\ell = w\ell$ for any $n$ and any $w$, 
so that every $w\ell\in \L$ is a limit leaf. Since every leaf of $\overline{\L} - \L$
is a limit leaf, we are done.
\end{proof}

\begin{corollary}[Isolated limit]\label{corollary:isolated_limit}
If $\overline{\L}$ contains an isolated limit leaf, then every leaf of $\overline{\L}$ 
is a limit leaf.
\end{corollary}
\begin{proof}
Suppose $\ell' \in \overline{\L}$ is an isolated limit leaf. By definition,
$\ell' = \lim w_n\ell$ for some infinite sequence $w_n$, and since $\ell'$ is isolated 
we must have $w_n \ell = \ell'$ for all sufficiently large $n$. But then $w_{n+1}\ell = w_n\ell$
for distinct $w_n,w_{n+1}$ and therefore Lemma~\ref{lemma:preperiodic_is_periodic} applies.
\end{proof}

We shall show in the sequel (Proposition~\ref{proposition:finite_cut_leaf}) 
that when $\overline{\L}$ contains an isolated limit leaf, $\Lambda$ is a dendrite.

\subsection{Trajectories}\label{subsection:trajectories}

Let's now fix $s\in \partial \M$ and suppose after replacing $s$ by $-s$ if necessary
that at least one endpoint of $I_f$ lies in $J_g$ 
(see Lemma~\ref{lemma:simplifying_assumption_intervals}).

The maps $f$ and $g$ on $S^1$ are invertible on their domains. We may therefore define
$H:S^1 \to S^1$ by
$$H(x) := f^{-1}(x) \text{ for } x\in J_f - J_g \text{ and } H(x):= g^{-1}(x) \text{ for } x\in J_g - J_f$$
and for $x\in J_f \cap J_g$ define $H(x)$ to be {\em either} $f^{-1}(x)$ {\em or} $g^{-1}(x)$.
Thus $H$ is many-valued which is annoying but seems to be inevitable if we want
to maintain the rotational symmetry and capture the full dynamics of $f^{-1}$ and $g^{-1}$.

If $\mu$ is a leaf of $S^1$ other than $\ell$ then it is possible that $\mu$ is contained
in $J_f$ or in $J_g$ but not both. If $\mu$ is contained in $J_f$ define $H\mu = f^{-1}\mu$
and if $\mu$ is contained in $J_g$ define $H\mu = g^{-1}\mu$, and otherwise $H\mu$ is
undefined. We may define $H\ell$ to be either $f^{-1}\ell$ or $g^{-1}\ell$. Notice for
{\em any} leaf $\mu$ and any $n$ the leaf $H^n\mu$ is either undefined, or is defined
up to the involution $\iota$, the latter ambiguity occurring if and only if 
$H^m\mu = \ell$ for some $0 \le m < n$.

Let $H\L$ denote the union of all possible leaves $H\mu$ for all $\mu \in \L$ for which
$H\mu$ is defined.

\begin{lemma}\label{lemma:H_z_leaves_L_z}
If $\L$ is finite then $H\L = \partial I_f \cup \partial I_g$. 
If $\L$ is infinite then $H\L = \L \cup \partial I_f \cup \partial I_g$.
\end{lemma}
\begin{proof}
No matter what, the possible values of $H\ell$ are $\partial I_f$ and $\partial I_g$.
By Proposition~\ref{proposition:props_of_L} if $\L$ is finite, $\L = \lbrace \ell \rbrace$
and we are done. 

Otherwise $J_f$ is contained in $I_f$ and $J_g$ is contained in $I_g$. Every
leaf of $\L$ is of the form $w\ell$ for some $w\in \SS_\ell$ of minimal length, and
if $w=fu$ then $w\ell$ is contained in $J_f$ and $Hw\ell = u\ell \in \L$ and
similarly if $w=gu$.
\end{proof}

Define a (many-valued) map $T:S^1 \to \lbrace f^{-1},g^{-1}\rbrace^\N$ as follows.
For $x\in S^1$ let $H^n x$ be a valid sequence of iterates of $x$, in the sense that
$H^{i+1} x = H (H^i x)$ for all $i$. Each $H^{i+1} x$ is of the form $f^{-1} H^i(x)$
or $g^{-1} H^i(x)$, and we let $T(x)(i+1)$ be $f^{-1}$ or $g^{-1}$ accordingly.
We call $T(x)$ a {\em trajectory} of $x$. 

\begin{remark}
The representation of $x\in S^1$ by a (possibly) many-valued 
$T(x) \in \lbrace f^{-1},g^{-1}\rbrace^\N$ is no more complicated than (and completely
analogous to) the representation of a real number by its decimal expansion, for
which one has many-valued representations such as $1.0000\cdots = 0.9999\cdots$.
Notice that $x\in S^1$ has a unique trajectory $T(x)$ unless $x$ is one of the
countably many points such that $H^n x \in J_f \cap J_g$ for some $n$. Furthermore,
unless some $p\in J_f \cap J_g$ satisfies $H^n p \in J_f \cap J_g$ for some $n$
(which is false for generic $s$), every $x\in S^1$ has at most two
trajectories.
\end{remark}

For any $x$, for any choice of trajectory $T(x)$ and for any $n\in \N$, let $T(x)|n$ denote
the prefix of $T(x)$ of length $n$.

For $w \in \lbrace f^{-1},g^{-1}\rbrace^n$ define
$$S^1_w: = \lbrace x \in S^1 \text{ such that } T(x)|n = w \text{ for some } T(x)\rbrace$$
and for $T \in \lbrace f^{-1},g^{-1}\rbrace^\N$ define
$$S^1_T: = \lbrace x \in S^1 \text{ such that } T(x) = T \text{ for some } T(x)\rbrace$$

For every $w \in \lbrace f^{-1},g^{-1}\rbrace^n$ the set $S^1_w$ is a finite union of
closed intervals and isolated points, and if $u,v \in \lbrace f^{-1},g^{-1}\rbrace^n$
are distinct, then $S^1_u$ and $S^1_v$ are disjoint except possibly at finitely many
points $x$ for which some forward iterate $H^m(x)$ is contained in $J_f\cap J_g$ for
$m<n$. 

In particular, each $S^1_w$ and each $S^1_T$ is {\em closed} as a subset of
$S^1$ (this is a key property that justifies the many-valued representation
$x \to T(x)$). Furthermore, the collection of all $S^1_w$ for $|w|=n$ (ignoring 
possible isolated points) induces a finite partition of $S^1$ into closed intervals. 

\begin{lemma}[Totally disconnected]\label{lemma:distance_estimate}
Let $x,y \in S^1$ have trajectories $T(x)$ and $T(y)$ with $T(x)|n = T(y)|n$.
Then $|u(x)-u(y)| \le \diam(\Lambda)\cdot |s|^n$. 
In particular, for every $T \in \lbrace f^{-1},g^{-1}\rbrace^\N$ 
the set $S^1_T$ is totally disconnected and contained in the preimage $u^{-1}(p)$
of a single point $p \in \partial \Lambda$.
\end{lemma}
\begin{proof}
By abuse of notation we write $T(x)|n = T(y)|n = w^{-1}$ for some $w\in \SS$
with $|w|=n$. By the definition of $H$ and $T$ there are elements $x',y'\in S^1$
(not unique) with $w x'=x$ and $w y'=y$. Now, $|u(x')-u(y')|\le \diam(\Lambda)$
and therefore $|u(wx')-u(wy')| \le \diam(\Lambda)\cdot |s|^n$.

In particular, any two points $x,y \in S^1$ with $T(x) = T(y)$ satisfy $u(x)=u(y)$.
Since the fibers of $u$ are totally disconnected by 
Lemma~\ref{lemma:multiple_preimages_implies_cut_point}, the lemma is proved.
\end{proof}

Lemma~\ref{lemma:distance_estimate} has the following consequence.

\begin{lemma}[Images dense]\label{lemma:no_intervals_of_trajectories}
For every open interval $I \subset S^1$ there is $p\in \partial J_f$ and $w\in \SS$
for which $wp \in I$.
\end{lemma}
\begin{proof}
Suppose not. Then $T$ would be constant on $I$, contrary to Lemma~\ref{lemma:distance_estimate}.
\end{proof}

Let $\T\subset \lbrace f^{-1},g^{-1}\rbrace^\N$ denote the union of all trajectories
of all points in $S^1$. Define an equivalence relation $\sim$ on $\T$ by $T \sim T'$ 
if there is a finite sequence $x_i \in S^1$ for $0\le i<n$ and trajectories
$$T = T_0,T_1,\cdots,T_n = T'$$
so that for each $0\le i < n$, both $T_i$ and $T_{i+1}$ are trajectories of $x_i$.

We may give $\lbrace f^{-1},g^{-1}\rbrace^\N$ the product topology (for which
it is homeomorphic to a Cantor set) and then topologize $\T$ as a subspace, and
$\T/\sim$ as a quotient of this subspace.

\begin{lemma}[Compact]\label{lemma:compact_Hausdorff}
The space $\T/\sim$ is compact.
\end{lemma}
\begin{proof}
It suffices to show that $\T$ is compact, equivalently that it is a closed subspace
of $\lbrace f^{-1},g^{-1}\rbrace^\N$.
For every $n$ and for every $w\in \lbrace f^{-1},g^{-1}\rbrace^n$ the set $S^1_w\subset S^1$
is closed and therefore compact. Thus if $S^1_{w_n}$ is nonempty for every
$w_n: = T|n$, then $S^1_T$ is nonempty too, so that $T \in \T$.
\end{proof}

We have already seen that if $x$ and $y$ have common trajectories, then $u(x)=u(y)$.
Thus, by the definition of $\sim$ and $\T$, the map $u:S^1 \to \partial K$ factors
through a map from $\T/\sim$ to $\partial K$ that by abuse of notation we denote
$u:\T/\sim \; \to \partial K$. 

\begin{proposition}[Address is conditionally fiber]\label{proposition:address_is_fiber}
Suppose $s \in \partial \M$. The map $u:\T/\sim\;\to\partial K$ is continuous and 
surjective. 

Furthermore, if Conjecture~\ref{conjecture:A} holds for $s$, then 
$u:\T/\sim\;\to\partial K$ is a homeomorphism, and $u(x)=u(y)$ if
and only if some trajectory of $x$ is equal to some trajectory of $y$.
\end{proposition}
\begin{proof}
The map $u:\T \to \partial K$ is surjective since $u:S^1 \to \partial K$ is. 
It is continuous by Lemma~\ref{lemma:distance_estimate}; thus, by the definition
of the quotient topology, $u:\T/\sim\; \to \partial K$ is continuous too.
So we just need to check it is injective.

We show, under the hypothesis of Conjecture~\ref{conjecture:A} that if $x,y\in S^1$
have $u(x)=u(y)$ then some trajectory of $x$ is already equal to a 
trajectory of $y$. Choose trajectories $T(x)$ and $T(y)$ with maximal length
prefix in common, so that $T(x)|n = T(y)|n = w^{-1}$ for $w\in \SS$, and 
let $x',y' \in S^1$ be such that $wx'=x$ and $wy'=y$. Thus in particular
$wu(x')=wu(y')$ so that $u(x')=u(y')$. Since $n$ is maximal
it must be the case that one of $x'$ is in $J_f-J_g$ and one is in $J_g-J_f$.
Under the hypothesis that Conjecture~\ref{conjecture:A} holds this implies that
one of $u(x')$ and $u(y')$ is in $f\Lambda - g\Lambda$ and one is in
$g \Lambda - f\Lambda$, contrary to $u(x')=u(y')$. 

This implies that $u:\T/\sim\;\to\partial K$ is injective (and therefore a
homeomorphism). For, if $u(T)=u(T')$ where $T$ and $T'$ are trajectories of
$x$ and $y$, then (as we have just shown) $x$ and $y$ have a common trajectory $T''$,
and then $T\sim T'' \sim T'$. The proof follows.
\end{proof}

\subsection{Cut points}\label{subsection:cut_points}

Our goal is now to use the dynamics of $H$ and the structure of $\L$ to describe the
set of cut points of $\Lambda$. Note that a point $p \in K$ is a cut point for
$\Lambda$ if and only if it is a cut point for $K$. 

\begin{definition}[Cut leaf]\label{definition:cut_leaf}
A leaf $\mu:=\lbrace x,y\rbrace$ in $S^1$ with $u(x)=u(y)$ is called a
{\em cut leaf}.
A leaf $\mu$ with $H^n\mu$ defined for all $n$ is called a {\em dynamical cut leaf}.
\end{definition}

A point $p\in \partial K$ is a cut point of $K$ (and $\Lambda$) if and only
if the cardinality of $u^{-1}(p)$ is greater than 1. In particular,
if $\lbrace x,y\rbrace$ is a cut leaf then $u(x)=u(y)$ is a cut point and
vice versa. 

It turns out that every dynamical cut leaf is a cut leaf, and conditional on 
Conjecture~\ref{conjecture:A} the converse is also true:

\begin{lemma}[Conditional classification of cut leaves]\label{lemma:topological_cut_point_classification}
Let $\mu:=\lbrace x,y\rbrace$ be a leaf in $S^1$. The following are equivalent:
\begin{enumerate}
\item $x$ and $y$ have trajectories that are equal; and
\item $\mu$ is a dynamical cut leaf; i.e.\/ the iterates $H^n \mu$ are defined for all $n$. 
\end{enumerate}
Moreover, these equivalent conditions imply that $\mu$ is a cut leaf.

Furthermore, if $s \in \partial \M$ and Conjecture~\ref{conjecture:A} 
holds for $s$, then if $\mu$ is a cut leaf, the
two equivalent conditions above hold for $\mu$.
\end{lemma}
\begin{proof}
The equivalence of the first two conditions follows
directly from the definition: the trajectories $T(x)$ and $T(y)$ are equal if
and only if for all $n$ the points $H^n(x)$ and $H^n(y)$ are both in $J_f$,
or both in $J_g$. But this is exactly the condition that $H^n\mu$ is defined for
all $n$. Lemma~\ref{lemma:distance_estimate} now implies that $u(x)=u(y)$,
i.e.\/ $\mu$ is a cut leaf. 

Conversely, under the conditional assumption that $s \in \partial \M$ 
and Conjecture~\ref{conjecture:A} holds for $s$, 
Proposition~\ref{proposition:address_is_fiber} says that if $u(x)=u(y)$ then
$x$ and $y$ have trajectories in common, so that $\mu$ is a dynamical cut leaf.
\end{proof}

It turns out that one can decide if there are dynamical cut leaves algorithmically
by examining the combinatorial configuration of finitely many leaves of $\L$.
This is expressed in Theorem~\ref{theorem:interaction}, to be proved in the
sequel. We first prove some prepatory lemmas. Our first observation is that 
when $\L$ is finite, there are no dynamical cut leaves:

\begin{lemma}[Unique trajectories]\label{lemma:unique_trajectories}
Suppose that $\ell$ is not contained in $I_f$.
Then no two distinct points in $S^1$ can share the same trajectory;
equivalently, there are no dynamical cut leaves. In particular, 
if $s\in \partial \M$ and Conjecture~\ref{conjecture:A} holds for $s$ then 
$u:S^1 \to \partial K$ is a homeomorphism, and $K$ has no cut points.
\end{lemma}
\begin{proof}
Suppose $\ell$ is not contained in $I_f$. Then $I_f\cap J_f$ and $I_f \cap J_g$
both consist of connected closed intervals, and similarly for $I_g$. It
follows by induction that for every $w\in \lbrace f^{-1},g^{-1}\rbrace^n$ the
set $S^1_w$ is either empty or consists of a single connected interval or isolated point, 
and therefore for every infinite word 
$\omega \in \lbrace f^{-1},g^{-1}\rbrace^\N$ the set $S^1_\omega$ is either empty, 
or consists of a single point. 

Now apply Lemma~\ref{lemma:topological_cut_point_classification} and 
Proposition~\ref{proposition:address_is_fiber}.
\end{proof}

On the other hand, limit leaves are dynamical cut leaves: 

\begin{lemma}\label{lemma:limit_leaves_give_cut_points}
Every limit leaf of $\overline{\L}$ is a dynamical cut leaf (and therefore a cut leaf). 
Consequently if $\overline{\L}$ has limit leaves, $K$ has cut points.
\end{lemma}
\begin{proof}
Suppose $\mu = \lim w_n\ell$ is a limit leaf. Let $w_n = \lbrace x_n,y_n\rbrace$
and $\mu = \lbrace x,y\rbrace$ where $x_n \to x$ and $y_n \to y$. By the definition
of $w_n\ell$ the points $x_n$ and $y_n$ have trajectories $T(x_n)$ and $T(y_n)$ with
$T(x_n)|n = T(y_n)|n = w_n^{-1}$. Pass to a subsequence if necessary so that each
$w_n^{-1}$ is a prefix of $w_{n+1}^{-1}$. Then $x_m,y_m \in S^1_{w_n^{-1}}$ for all
$m\ge n$ and therefore the same is true for $x$ and $y$. In particular, $x$ and
$y$ have a trajectory in common, i.e.\/ $\mu$ is a dynamical cut leaf.
\end{proof}

The next proposition addresses the question of when leaves of $\L$ are cut leaves.
This turns out to be a very highly constrained situation.

\begin{proposition}[Finite cut leaf]\label{proposition:finite_cut_leaf}
The following are equivalent for splittable $s\in \M$:
\begin{enumerate}
\item some $\ell' \in \L$ is a cut leaf; and
\item every $\ell' \in \L$ is a cut leaf.
\end{enumerate}
Furthermore if either holds, then $\Lambda$ is a dendrite. In particular,
this holds if $\overline{\L}$ contains an isolated limit leaf.
\end{proposition}
\begin{proof}
If $\ell':=\lbrace x,y\rbrace$ is a leaf and $f\ell'$ is defined
then $fu(x) = u(f(x))$ and $fu(y)=u(f(y))$ so $\ell'$ is a cut
leaf if and only if $f\ell'$ is. The equivalence of the two conditions
follows from this.

Next we show either condition implies $\Lambda$ is a dendrite. 
Let $\ell=\lbrace x,y\rbrace$ which by hypothesis
is a cut leaf, and let $p=u(x)=u(y)$. Since $\ell$ is invariant under $\iota$
so is $p$, hence $p=0$. Thus the point $0$ disconnects $K$ into two components, 
one contained in $fK$ and one contained in $gK$ (we do not assume that $s\in \partial \M$
or that Conjecture~\ref{conjecture:A} holds). We shall show that $K$
(and therefore also $\Lambda$) is a dendrite.

The proof is very similar to the proof of Lemma~\ref{lemma:degenerate}.
Suppose $K$ is not a dendrite, equivalently that it has nonempty interior.
Let $A$ be a component of the interior of $K$ of diameter at least $D-\epsilon$, 
where $D$ is the supremum of the diameters of components of the interior of $K$.
Since $0$ disconnects $K$ into two components, $A$ is entirely contained in $fK$
or in $gK$; without loss of generality, say it is contained in $fK$. 
But then $f^{-1}A$ is also contained in the interior of $K$, and has
diameter at least $\diam(A)\cdot |s|^{-1}$, contrary to the
definition of $A$ providing $\epsilon$ is small enough (depending on $|s|$).
In particular, $K$ (and hence $\Lambda$) is a dendrite after all.

If $\overline{\L}$ contains an isolated limit leaf $\mu$ then since $\mu$ is a
limit of leaves in $\L$, it must actually be a leaf of $\L$. Since any
limit leaf is a dynamical cut leaf (and therefore a cut leaf) by
Lemma~\ref{lemma:limit_leaves_give_cut_points}, the hypothesis of the
proposition holds and $\Lambda$ is a dendrite.
\end{proof}

It is asserted in \cite{Solomyak}, p.1930 that $\Lambda$ is a dendrite if and only
if $f\Lambda \cap g\Lambda$ consists of a single point; Solomyak attributes this
to Bandt--Keller \cite{Bandt_Keller}. One direction is true; if 
$|f\Lambda \cap g\Lambda| = 1$ then $\Lambda$ is a dendrite and $s\in \partial \M$;
see Lemma~\ref{lemma:degenerate}.
However there are many examples of dendritic $\Lambda$ 
where $|f\Lambda \cap g\Lambda|=\infty$ and
$s$ is contained in the interior of $\M$. This happens for many algebraic $s$
for which there are words $u,v$ with $|u|=|v|$ starting with $f$ and $g$ respectively
for which $u\Lambda = v\Lambda = f\Lambda \cap g\Lambda$;
see Example~\ref{example:interior}.

One reason that it is difficult to give a complete description of the set of $s$ with
dendritic $\Lambda$ is that there are examples for which $\Lambda$ is not a
dendrite, but comes remarkably close in a certain sense.

\begin{example}[Topological $\R$-tree]\label{example:R_tree}
Let $s \sim 0.41964 + 0.60629i$ be a root of $z^3 + z^2 - z + 1$. Then
$\Lambda$ is homeomorphic to a closed disk, and $fgff \Lambda = gfgg \Lambda$ 
(in particular $fgff z = s^4 z = gfgg z$ for all $z\in \C$) and
furthermore $fgff\; \inte(\Lambda) = f\;\inte (\Lambda) \cap g\; \inte (\Lambda)$ where
$\inte$ denotes interior.

On the other hand, it turns out that $\Lambda$ contains a dense path-connected subset $T$
satisfying $T = fT \cup gT$ in which every two points are joined by a 
unique path and every point is a cut point (of course, $T$ is not closed).
The set $T$ in its path topology is a topological $\R$-tree, 
and a suitable limit of rescalings of $T$ gives an example of what is known as a 
{\em half-zipper} \cite{Calegari_Gwynne}.

Let's first describe one path $\alpha$ in $T$, and then we can take 
$T = \bigcup_{w\in \SS} w\alpha$. The path $\alpha$ joins the fixed point of $f$
to the fixed point of $g$, and is obtained as the limit of a sequence of
polygonal paths $\alpha_n$ joining $f^n(0)$ to $g^n(0)$ that interpolate
a sequence of points $w_i(0)$ where $|w_i|=n$ for each $i$, and where 
$w_0 = f^n$, and $w_N = g^n$ for some $N$. 
Here is a recursive procedure to define $w_{i+1}$ assuming $w_i$ is not equal to $g^n$.
We will find a string $\sigma$ in $w_i$ where either $|\sigma|=4$ or
$\sigma$ is a suffix of $w_i$ with $|\sigma|\le 3$ such that
either $\sigma$ is equal to one of $fgff$ or $gfgg$ (if $|\sigma|=4$), 
or $\sigma$ is a prefix of one of them (if $|\sigma|<4$). 
Call such a substring {\em flippable}. Then we will `flip the bits' of $\sigma$, 
i.e.\/ change $f$s to $g$s and vice versa; this will produce $w_{i+1}$. 

Here is how to find the correct flippable substring $\sigma$. 
Our first guess for $\sigma$ is $\tau$, the maximal
length substring of $w_i$ of length $\le 4$ that begins with $f$. If no such substring
exists, then $w_i = g^n$ and therefore $i=N$. If $\tau$ is flippable,
set $\sigma = \tau$; otherwise $\tau$ starts like one of $fgff$ or $gfgg$ but has
a first `incorrect' letter. Let $\tau'$ be the maximal length $\le 4$ substring of $w_i$
that starts at the first incorrect letter of $\tau$, and then update $\tau \leftarrow \tau'$.
Continue in this way until $\tau$ is flippable (this must happen at some point, because
the first letter of $\tau$ keeps moving to the right, and any string of length 
$1$ is flippable).

If $w_{i+1}$ is obtained from $w_i$ by flipping a string of length $4$ then
$w_{i+1}(0) = w_i(0)$. Otherwise $w_{i+1}(0)$ and $w_i(0)$ are joined by a
straight segment of $\alpha_n$ of length $d(w_{i+1}(0),w_i(0)) \le C\cdot|s|^n$. Furthermore,
by construction, $\alpha_{n+1}$ is obtained from $\alpha_n$ by replacing each such
straight segment by at most two segments of length $\le C\cdot|s|^{n+1}$. Thus
the $\alpha_n$ converge to a limit $\alpha_\infty$. One may check that
$\alpha_\infty$ is a Jordan arc; the arc $\alpha$ is its interior. See Figure~\ref{half_zipper}.

\begin{figure}[htpb]
\centering
\includegraphics[scale=0.5]{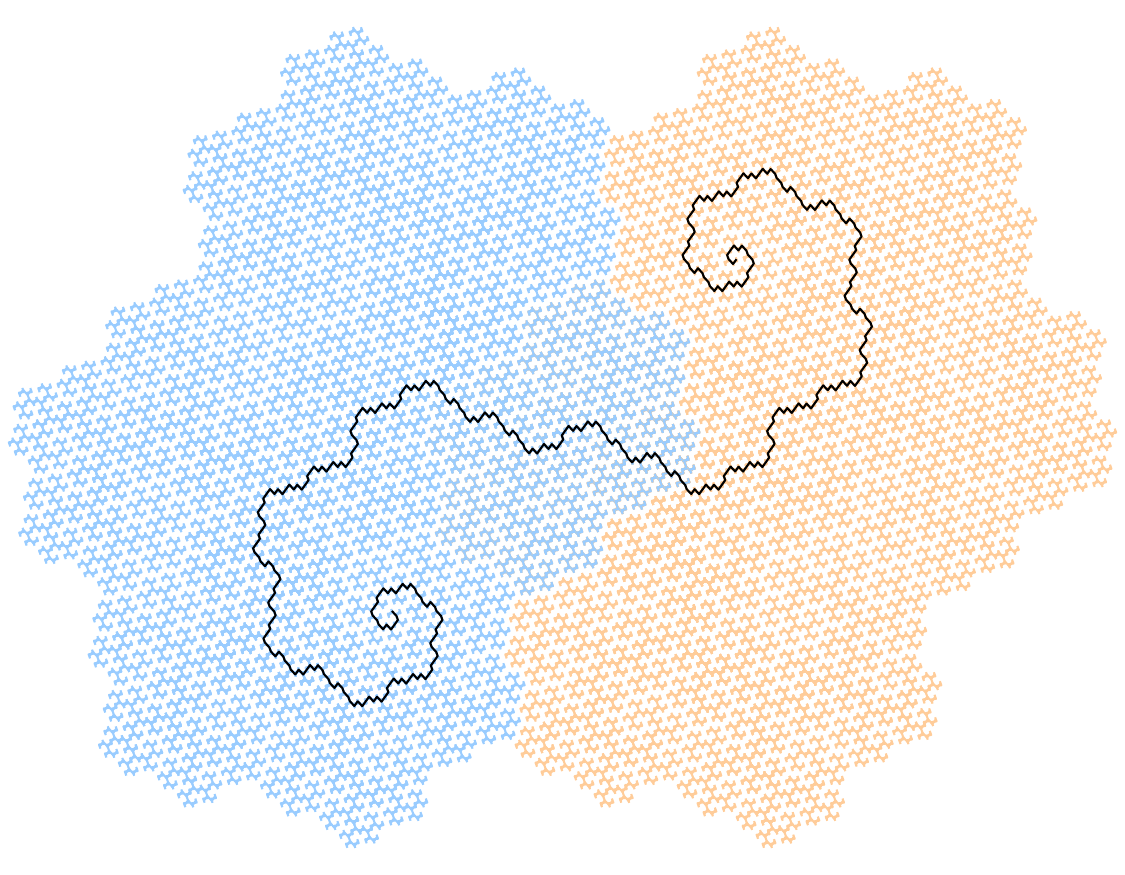}
\caption{For $s \sim 0.41964 + 0.60629i$ a root of $z^3 + z^2 - z + 1$ 
there is a topological $\R$-tree $T$ dense in $\Lambda$ with $T = fT \cup gT$.
Here $fT$ is in blue and $gT-fT$ is in orange.
The tree $T$ is the union of iterates $T = \bigcup_w w\alpha$ where $\alpha$ is
the interior of a certain Jordan arc (in black) joining the fixed points of $f$ and $g$.}
\label{half_zipper}
\end{figure}
\end{example}

\subsection{Cut leaves and limit leaves}\label{subsection:cut_versus_limit}

Lemma~\ref{lemma:limit_leaves_give_cut_points} says that every 
limit leaf of $\L$ is a dynamical cut leaf and therefore a cut leaf. 

What about the converse? Suppose $s\in \partial\M$ is such that there
is $p\in K$ for which $u^{-1}p$ has cardinality at least 4. For any two
points $x,y\in u^{-1}p$ the leaf $\mu:=\lbrace x,y\rbrace$ is a cut leaf; in particular
one may find a pair of cut leaves that cross. On the other hand, limit leaves
are contained in the lamination $\overline{\L}$ and therefore never cross.

Are there any $s\in \partial \M$ with this property? Whenever $f\Lambda \cap g\Lambda$
consists of a single point, $\Lambda$ is a dendrite, and unless $s$ is real, this
dendrite has countably infinitely many points $p$ for which $u^{-1}p$ has cardinality
at least $3$. Extensive computer experiments have failed to find an example where
some $u^{-1}p$ has cardinality $>3$ and therefore we are motivated to make the 
following


\begin{conjecture}[Order 3]\label{conjecture:cut_is_limit}
Let $s \in \partial \M$ and suppose
Conjecture~\ref{conjecture:A} holds for $s$. Then every cut leaf is a limit leaf.
\end{conjecture}

In \S~\ref{section:linear_dynamics} we shall describe a class of dynamical
systems that generalizes the action of $H$ on $S^1$. In this family one may 
find many examples with dynamical cut leaves that are not limit leaves based on 
cardinality arguments as above; see Example~\ref{example:cut_not_limit}.

\subsection{An algorithm to detect cut leaves}
\label{section:topological_classification_of_actions}

Lemma~\ref{lemma:topological_cut_point_classification} gives a
sufficient condition (and, conditional on 
Conjecture~\ref{conjecture:A} a necessary one) 
for $\Lambda$ to have cut points. However, 
there is no obvious way to turn this condition into a 
finite certificate one way or the other.

In this section we give a finite algorithm to 
determine whether or not there are dynamical cut leaves, and therefore
conditional on Conjecture A, to determine whether or not $\Lambda$ has cut points.

\begin{definition}
\label{definition:topological_leaf_hiding}
Let $\mu$ and $\nu$ be leaves which do not cross $\ell$. 
We say that $\mu$ {\em hides} $\nu$ if the geodesic
$\gamma_\mu$ separates the geodesic $\gamma_\ell$ from the
geodesic $\gamma_\nu$ in $\D$. 

We say that two leaves \emph{interact} if they cross, or if one 
of them hides the other.
\end{definition}

We have already seen in Lemma~\ref{lemma:degenerate} that if
$I_f$ (and therefore also $I_g$) are degenerate, then $\Lambda$ is a dendrite. 
Otherwise we may define $\ell_f$ and $\ell_g$ to be the leaves 
with endpoints $\partial I_f$ and $\partial I_g$ respectively. 

It may be that one of these leaves (and therefore both of them) cross $\ell$; 
this happens exactly when $\L = \lbrace \ell\rbrace$.
Otherwise by replacing $s$ by $-s$ if necessary we may assume that $\ell_f$ lies
entirely in $J_g$ (see Lemma~\ref{lemma:simplifying_assumption_intervals}).

\begin{theorem}[Interaction]\label{theorem:interaction}
Let $s\in \M$ be splittable. If $I_f$ (and therefore also $I_g$) is degenerate,
$\Lambda$ is a dendrite. 

Otherwise, assume (by replacing $s$ by $-s$ if necessary) that
$\ell_f$ has at least one endpoint in $J_g$. Let $w \in \SS$ have least length 
so that $w\ell$ interacts with $\ell_f$, if any 
(we allow the possibility that $w$ is empty if $\ell$ crosses $\ell_f$).
The following are true:
\begin{enumerate}
\item{some such $w$ exists;}
\item{the letters of $w$ alternate between $f$ and $g$;} 
\item{it cannot be the case that $\ell_f$ hides $w\ell$ from $\ell$;}
\item{if $w\ell$ hides $\ell_f$ from $\ell$ then there are dynamical cut leaves 
(and consequently $\Lambda$ contains cut points); and}
\item{if $\ell_f$ crosses $w\ell$ then there are no dynamical cut leaves.}
\end{enumerate}
In particular, 
if $s\in \partial \M$ and Conjecture~\ref{conjecture:A}
holds, then if $\ell_f$ crosses $w\ell$ there are no cut points in $\Lambda$.
\end{theorem}
\begin{proof}
If $I_f,I_g$ are degenerate, Lemma~\ref{lemma:degenerate} implies
that $\Lambda$ is a dendrite. So suppose $I_f,I_g$ are not degenerate, and
$\ell_f$ has at least one endpoint in $J_f$.

The proof reduces to a careful study of the action of $f$ and $g$ on leaves of
$\L$ together with $\ell_f$ and $\ell_g$. Under our assumption about
$\ell_f$, if $\ell$ does not cross $\ell_f$
then $\ell_f$ lies to the right of $\ell$, so if $w\ell$
interacts with $\ell_f$ the word $w$ must begin with $g$. For any
word $w\in \SS$ let $w^c$ denote the word obtained from $w$ by replacing each
$f$ with $g$ and vice versa; i.e.\/ $w^c$ is obtained from $w$ by `flipping the bits'
of $w$, as in Example~\ref{example:R_tree}. Thus if $w\ell$ interacts with $\ell_f$,
it follows that $w^c\ell$ interacts with $\ell_g$. We fix the word $w$ and this
notation for the remainder of the proof.

\medskip 

Since $I_f$ and $I_g$ are not degenerate, $\ell_f$ and $\ell_g$
are honest leaves. Let $I \subset S^1$ be the open interval bounded by 
the endpoints of $\ell_f$. 
Lemma~\ref{lemma:no_intervals_of_trajectories} says that $I$ 
contains $wp$ for $p\in J_f$ and some (least) $w\in \SS$. 
But then $w\ell$ interacts with $\ell_f$. This proves (1).

\medskip

The next step is to prove (4) (and along the way, (2)). Let $w$ be chosen with $|w|=n$ minimal so that
$w\ell$ interacts with $\ell_f$. We will show that if $w\ell$ hides $\ell_f$ then $\L$
has a limit leaf, and therefore by Lemma~\ref{lemma:limit_leaves_give_cut_points}
it will follow that $\Lambda$ has a cut point.

If $A$ is any set of leaves not crossing
$\ell$, then $\lambda \in A$ is {\em exposed} if no other $\mu \in A$ hides
$\lambda$ from $\ell$. For any such $A$ we can take the subset of exposed leaves.

\begin{lemma}[Exposed leaves]\label{lemma:exposed_leaves}
Let $A$ be the set of leaves of $\L-\ell$ of depth $\le n$ where $n=|w|$ minimal
that interacts with $\ell_f$. Then the set $A_E$ of exposed leaves of $A$ contains
exactly $2n$ elements which are all and only those of the form $u\ell$ where the
letters of $u$ alternate between $f$ and $g$.
\end{lemma}
\begin{proof}
First of all, $u\ell$ is to the left resp. right of $\ell$ if $u$ starts
with $f$ resp. $g$. In particular, $f\ell$ hides $ffu'\ell$ and $g\ell$ hides
$ggu'\ell$ for any $|u'|\le n-2$. Thus if $u\ell$ is exposed, $u$ cannot
begin with $ff$ or $gg$.

Second of all, if $|u|<n$ and starts with $f$ then $u\ell$ does not
interact with $\ell_g$ (by the definition of $n$). Thus if $u\ell$ hides $v\ell$
then $gu\ell$ hides $gv\ell$; see Figure~\ref{non_interact_hide}.

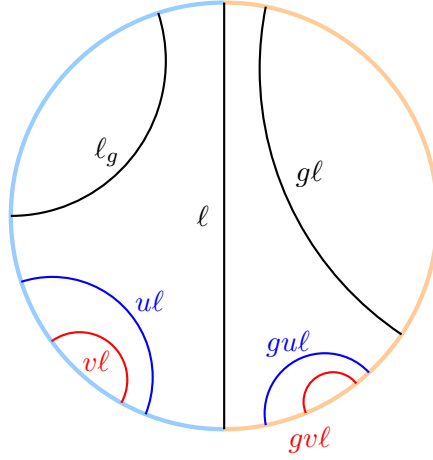
\begin{figure}[htpb]
\labellist
\small\hair 2pt
\pinlabel $\ell$ at 100 110
\pinlabel $\ell_g$ at 55 140
\pinlabel $g\ell$ at 150 130
\pinlabel $\color{blue}u\ell$ at 75 70
\pinlabel $\color{red}v\ell$ at 50 42
\pinlabel $\color{blue}gu\ell$ at 140 50
\pinlabel $\color{red}gv\ell$ at 150 5
\endlabellist
\centering
\includegraphics[scale=0.8]{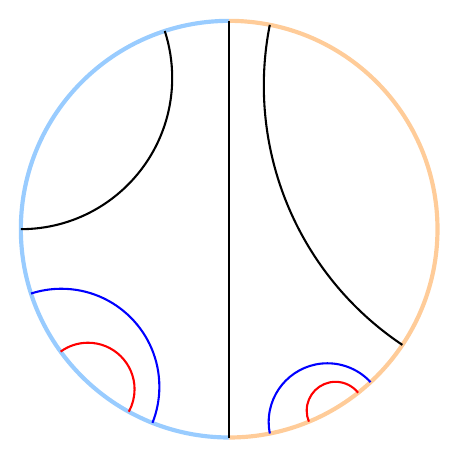}
\caption{If $u$ starts with $f$ and $u\ell$ hides $v\ell$ and does 
not interact with $\ell_g$ then $gu\ell$ hides $gv\ell$.}
\label{non_interact_hide}
\end{figure}

In particular, if $v\ell$ is hidden with $|v|<n$, then $fv\ell$ and $gv\ell$ are 
hidden. It follows that the only exposed leaves $u\ell$ must have $u$ alternate
between $f$ and $g$. By induction and the previous discussion none of these hide
each other. The proof follows.
\end{proof}

Note that Lemma~\ref{lemma:exposed_leaves} proves (2).

Now let's suppose $w\ell$ of least depth hides $\ell_f$, and similarly $w^c\ell$ of
least depth hides $\ell_g$. By construction, $w,w^c\in A_E$ are the unique leaves
of depth $n$. Let $B_E$ be the set of leaves obtained from $A_E$ by removing
$w\ell$ and $w^c\ell$, and adding $\ell_f$ and $\ell_g$. 
Let $P_f$ be the infinite hyperbolic `polygon' in $\D$ bounded by $\ell$ and the
leaves of $A_E$ on the left of $\ell$, and $Q_f$ the `polygon' bounded by $\ell$
and the leaves of $B_E$ on the left of $\ell$. Since $w^c\ell$ hides $\ell_f$
it follows that $P_f$ is contained in $Q_f$; and likewise $P_g$ is contained in $Q_g$.
The edges of each of $P_f$ and $Q_f$ alternate between leaves, and segments in
$J_f$; call the edges which are segments of $J_f$ the {\em ideal segments}. See
Figure~\ref{dynamics_of_P_and_Q}.

\begin{figure}[htpb]
\labellist
\small\hair 2pt
\pinlabel $P_f$ at 80 110
\pinlabel $Q_g$ at 140 110
\pinlabel $\ell_g$ at 32 140
\pinlabel $w^c\ell$ at 70 138
\pinlabel $w\ell$ at 153 82
\pinlabel $\ell_f$ at 188 80
\endlabellist
\centering
\includegraphics[scale=0.8]{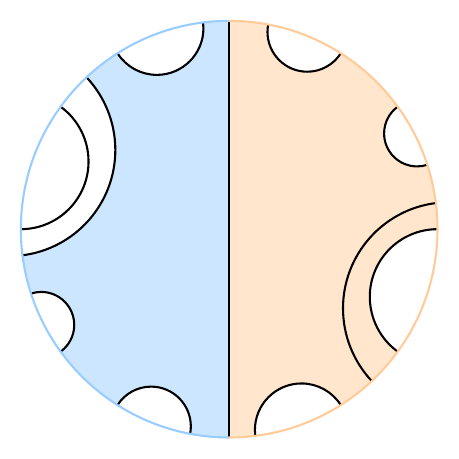}
\caption{The polygon $P_f$ is contained in $Q_f$. Since $f:Q_g \to P_f$ takes
ideal segments to ideal segments, the restriction $f:P_g \to P_f$ takes
distinct ideal segments inside distinct ideal segments.}
\label{dynamics_of_P_and_Q}
\end{figure}
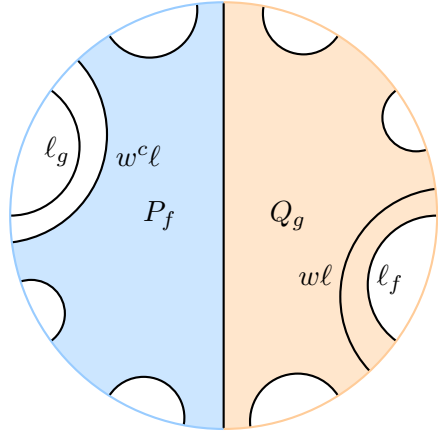

By construction, $f$ takes $Q_g$ to $P_f$, taking ideal segments to ideal segments;
in particular, $f$ takes $P_g$ inside $P_f$, taking distinct ideal segments of $P_g$
inside distinct ideal segments of $P_f$, and likewise for $g:P_f \to P_g$. In particular,
every element of the infinite sequence of leaves 
$$w\ell \to fw\ell \to gfw\ell \to fgfw\ell \to \cdots$$
has endpoints on {\em distinct} ideal segments of $P_f$ or $P_g$. In particular, this
sequence of leaves has a subsequence that converges to a nontrivial limit $\ell'$
contained in $\overline{\L}$. In other words, $\L$ has a limit leaf, which is 
necessarily a dynamical cut leaf by Lemma~\ref{lemma:limit_leaves_give_cut_points}. 
This proves (4).

\medskip

The same argument explains why $\ell_f$ cannot hide $w\ell$, for otherwise
(with notation as above) the map $f$ would take $Q_g$ to $P_f$ as before, but
now the sum of the lengths of the ideal segments of $Q_g$ is {\em strictly smaller} 
than the sum of the lengths of the ideal segments of $P_f$, so that $f$ would act as
a strict expansion, contrary to the fact that it acts as a strict contraction
by Lemma~\ref{lemma:expanding_map}. This proves (3).

\medskip

Finally we show that if $\ell_f$ crosses $w\ell$ there are no dynamical cut leaves.

Define polygons $P_f$ and $Q_f$ as
before (and similarly for $P_g$ and $Q_g$). Note that the hypothesis that 
$\ell_f$ crosses $w\ell$ means that neither of $P_g$ and $Q_g$ contains the other. 
Note that $H$ takes $P_g$ to $Q_f$ and $P_f$ to $Q_g$, taking
ideal segments to ideal segments.

We suppose to the contrary that a dynamical cut leaf $\mu$ exists, 
and without loss of generality we may assume it is to the left of $\ell$.
The image $H^n\mu$ is defined for all $n$ if and only if $H^n\mu$ never crosses $\ell$,
which implies that $H^n\mu$ never crosses any leaf of $\L$. In particular,
the endpoints of $\mu$ lie in the ideal segments of $P_f$. Suppose $\mu$ does
not hide $f\ell$. Then $\mu$ hides an interval $I\subset J_f$ from $\ell$ that
is disjoint from $f\ell$, and therefore $H\mu = f^{-1}\mu$ hides an interval 
$H I = f^{-1} I \subset J_g$ from $\ell$. If $H\mu$ does not hide $g\ell$ we
can repeat this process inductively obtaining a sequence of intervals $H^n I$ all
disjoint from $f\ell$ and $g\ell$. But $\diam(u H^n I) = |s|^{-1} \cdot \diam(u H^{n-1} I)$
and therefore this process must stop at some finite time. Hence, after replacing
$\mu$ with some finite $H^n \mu$ if necessary, we may assume $\mu$ hides $f\ell$ from $\ell$.
In particular, the endpoints of $\mu$ lie on {\em distinct} ideal segments of $P_f$.

Now, since $w^c\ell$ crosses $\ell_g$ there is a unique ideal segment 
$\sigma_f$ of $Q_f$ that contains an endpoint of $w^c\ell$ in its interior, and
a subsegment $\sigma'_f$ of $P_f$ contained in $\sigma_f$ (and likewise
there are $\sigma_g$ in $Q_g$ and $\sigma'_g$ in $P_g$). Let $\sigma:=\sigma_f\sqcup \sigma_g$
and $\sigma':=\sigma'_f \sqcup \sigma'_g$. Every point in an ideal segment of $Q_f$
is either in $\sigma_f -\sigma'_f$ or its is contained in an ideal segment of $P_f$ and
therefore its image under $H$ is contained in an ideal
segment of $Q_g$, and similarly with $f$ and $g$ interchanged. Thus, iterating $H$ defines a
first return map $R:\sigma' \to \sigma$. We may iterate $R$ to obtain a partially defined
first return map $R': \sigma'_f \to \sigma_f$. By construction, this map is 
orientation-preserving, and is either undefined everywhere, or 
moves the common vertex $p \in \sigma'_f \cap \sigma_f$ strictly off itself. 
To see this, note that some iterate $H^m p$ is an endpoint of
$\ell_f$ or $\ell_g$ in the interior of some ideal segment of $P_g$ or $P_f$. 
Since $R'$ is strictly expanding on its domain of definition 
in the Lesbesgue measure by Lemma~\ref{lemma:expanding_map}, it follows that
some iterate of $R'$ is undefined everywhere. See Figure~\ref{dynamics_of_P_and_Q_cross}.

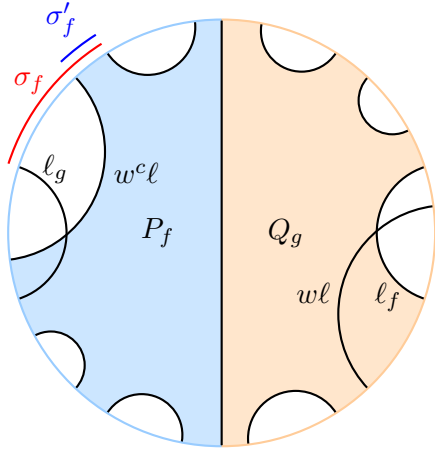
\begin{figure}[htpb]
\labellist
\small\hair 2pt
\pinlabel $P_f$ at 80 110
\pinlabel $Q_g$ at 140 110
\pinlabel $\ell_g$ at 32 140
\pinlabel $w^c\ell$ at 70 138
\pinlabel $w\ell$ at 153 82
\pinlabel $\ell_f$ at 188 80
\pinlabel $\color{red}\sigma_f$ at 20 180
\pinlabel $\color{blue}\sigma_f'$ at 35 210
\endlabellist
\centering
\includegraphics[scale=0.8]{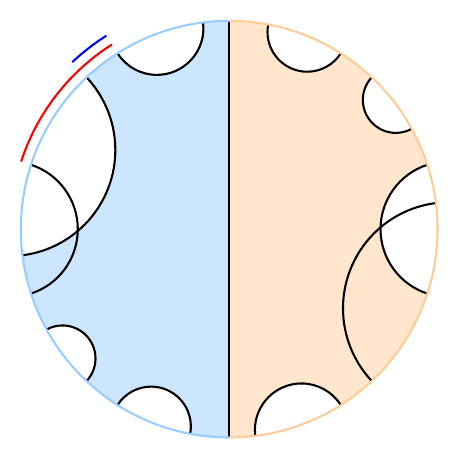}
\caption{When $w^c\ell$ and $\ell_g$ cross we obtain ideal segments $\sigma_f$ in
$Q_f$ and a subsegment $\sigma'_f$ in $P_f$. The first return map of $\sigma'_f$
to $\sigma_f$ is eventually undefined.}
\label{dynamics_of_P_and_Q_cross}
\end{figure}

This means that as we iterate $H$ on $\mu$ there is a first $n$ for which at least
one vertex of $H^n\mu$ is in $\sigma - \sigma'$. But $\mu$ (and therefore $H^n\mu$)
has vertices on distinct ideal segments of $P_f$ or $P_g$ (depending what side of
$\ell$ it is on). It follows that $H^n\mu$ crosses $\ell_f$ or $\ell_g$ transversely,
so that $H^{n+1}\mu$ crosses $\ell$ transversely and $H^{n+2}\mu$ is undefined.
This proves (5).

Finally, if $s\in \partial \M$ and satisfies Conjecture~\ref{conjecture:A} then
cut leaves are dynamical cut leaves by Lemma~\ref{lemma:topological_cut_point_classification}.
This completes the proof.
\end{proof}

\section{Linearization}\label{section:linear_model}

In this section, we show that the action of $H$ on $S^1$ is topologically
conjugate to a (discontinuous) linear map, depending on two parameters
$\lambda \in [1,2]$ and $\theta \in S^1$. 

\subsection{A linear model}

We parameterize $S^1$ as $\R/2\pi\Z$. 
If $A,B$ are closed intervals in $S^1$, there is a unique linear map from
$A$ to $B$ which preserves the circular order. We define a two (real) parameter
family of piecewise continuous linear endomorphisms of $S^1$.

\begin{definition}[Linear model]\label{definition:linear_model}
Let $J_F$ and $J_G$ be the intervals in $S^1$ to the left and to the right
of the vertical bisector; i.e.\/ their common endpoints are at $\pi/2$ and $3\pi/2$.
For $\lambda \in [1,2]$ and $\theta \in S^1$ define $G_{\theta,\lambda}^{-1}:J_G \to S^1$ and
$F_{\theta,\lambda}^{-1}:J_F \to S^1$ by 
$$G_{\theta,\lambda}^{-1}(t) = \lambda t + \theta \text{ and } F_{\theta,\lambda}^{-1} = \iota G_{\theta,\lambda}^{-1} \iota$$
Then define
$E_{\theta,\lambda}:S^1 \to S^1$ by
$$E_{\theta,\lambda}(x) = F_{\theta,\lambda}^{-1}(x) \text{ for } x\in J_F - J_G \text{ and }
E_{\theta,\lambda}(x) = G_{\theta,\lambda}^{-1}(x) \text{ for } x \in J_G - J_F$$
and for $x\in J_F \cap J_G$ define $E_{\theta,\lambda}(x)$ to be either 
$F_{\theta,\lambda}^{-1}(x)$ or $G_{\theta,\lambda}^{-1}(x)$. Thus $E_{\theta,\lambda}$ is
many-valued in the same way as $H$.

Define $I_F$ and $I_G$ to be the images of $J_F$ and $J_G$ under 
$F_{\theta,\lambda}^{-1}$ and $G_{\theta,\lambda}^{-1}$, and say that $I_F$ and $I_G$ are
degenerate if they are equal to all of $S^1$ (which happens if and only if $\lambda=2$).

Define $F_{\theta,\lambda}$ and $G_{\theta,\lambda}$ to be inverse to $F_{\theta,\lambda}^{-1}$
and $G_{\theta,\lambda}^{-1}$ on their domains of definition, unless $\lambda = 2$ in which case
$F_{\theta,\lambda}$ is undefined on the common image $F_{\theta,\lambda}^{-1}(\partial J_F)$.
\end{definition}
In the sequel we will suppress subscripts where this is unambiguous. 

\subsection{Linearization}\label{subsection:linearization}
The main theorem of this section is to show that,
conditional on Conjecture~\ref{conjecture:A}, for any $s\in \partial \M$ there 
are unique parameters $\lambda \in [1,2]$ and $\theta\in S^1$ for which $H$ is
topologically conjugate to $E_{\theta,\lambda}$.

Note that this theorem is {\em unconditional} in the sense that
it does not depend on Conjecture~\ref{conjecture:A}. It also applies to any
$s\in \M$ in the boundary or not which is splittable 
(see Definition~\ref{definition:splittable}) and satisfies $|s| < 1/\sqrt{2}$.

\begin{theorem}[Linearization]\label{theorem:linearization}
Suppose that one of the following holds for $s\in \M$:
\begin{enumerate}
\item{$s\in \partial \M$; or}
\item{$s\in \M$ is splittable and $|s|<1/\sqrt{2}$}
\end{enumerate}
and suppose further (after replacing $s$ by $-s$ if necessary) 
that $\ell_f$ has at least one endpoint in $J_g$.

Then there are unique parameters $\lambda \in [|s|^{-1},2]$ and $\theta \in S^1$
and a homeomorphism $\sigma:S^1 \to S^1$ commuting with $\iota$ that conjugates $H,f,g$
to $E_{\theta,\lambda},F_{\theta,\lambda},G_{\theta,\lambda}$ on their domains of definition.
\end{theorem}
Thus Theorem~\ref{theorem:linearization} provides a {\em linear model} for
the dynamics of $H$. 

\medskip

Recall by Theorem~\ref{theorem:strict_inequality} (to be proved in
Appendix~\ref{section:inequality_proof})
that for $s\in \partial \M$ we have $|s| \le 1/\sqrt{2}$ with equality if and
only if $s = \pm i/\sqrt{2}$. This is handled as a special case.

\begin{example}[$s=i/\sqrt{2}$]\label{example:i_on_sqrt_2}
Recall from Example~\ref{example:rectangle} that $\Lambda$ is the rectangle
$$\Lambda = \lbrace z = x + iy \text{ such that } -2 \le x \le 2 \text{ and } -\sqrt{2} \le y \le \sqrt{2}\rbrace$$
The image $f\Lambda$ is the subset of $\Lambda$ with $x \le 0$ and the
image $g\Lambda$ is the subset with $x\ge 0$; thus their intersection is the
interval $z = it$ for $t\in [-\sqrt{2},\sqrt{2}]$. In this case, 
$\Lambda = K$ and $\partial K$ is a polygonal Jordan curve of total length
$4+2\sqrt{2}$. The dynamics of $f^{-1}$ and $g^{-1}$ on $\partial K$ is
already linear with respect to Lesbesgue measure on $\partial K$. In particular,
$H$ is conjugate to $E_{\theta,\lambda}$ for $\lambda = \sqrt{2}$ and
$\theta = \pi/2$.
\end{example}
For the remainder of this section we therefore make the assumption only
that $s$ is splittable (equivalently that $H$ is defined at all) and
that there is a strict inequality $|s| < 1/\sqrt{2}$.

The proof of Theorem~\ref{theorem:linearization} will take up the rest
of this section. The outline is as follows.
First, we show how Milnor--Thurston's kneading theory gives us a monotone map
$s:S^1 \to S^1$ commuting with $\iota$ which is a {\em semiconjugacy} from
$H,f,g$ to $E,F,G$, and then we show this semiconjugacy is an honest conjugacy.

Let us remark that points of discontinuity are usually treated in this theory
in a somewhat ad hoc manner. A semiconjugacy (or conjugacy) to a linear map is typically
obtained by constructing an invariant measure without atoms, and then integrating it; thus
the (grand) orbits of any countable set of points are irrelevant. For this reason we
will usually ignore the many-valuedness of $H$ and $E$, or the special role of
the points $J_f\cap J_g$.

The maps $H$ and $E$ both commute with $\iota$. 
Since we are looking for a conjugacy that commutes with $\iota$
it is convenient to choose an orientation-preserving homeomorphism
$\alpha:[0,1] \to J_f$ and define $h:[0,1] \to [0,1]$ by 
$$h(t):= \alpha^{-1} H \alpha(t) \text{ if } H\alpha(t) \in J_f \text{ and }
h(t):= \alpha^{-1} \iota H \alpha(t) \text{ otherwise }$$

Evidently to prove the theorem it will suffice to find a conjugacy from $h$ to a 
piecewise linear map. First we find a semiconjugacy. We do this following
Preston \cite{Preston}, \S~11 (\S~7 in the 2003 revision), and we follow the
discussion and notation there.

\begin{definition}[Piecewise monotone map]
\label{definition:piecewise_monotone}
Let $I:=[0,1]$. A map $h:I \to I$ is {\em piecewise monotone} if there is
$m \ge 0$ and points $0 < c_1 < \cdots < c_m < 1$ called the
{\em points of discontinuity} so that $h$ is and strictly increasing on every 
complementary component of the points of discontinuity.

These components are called the \emph{laps} of $h$,
and the set of piecewise monotone maps is denoted $N(I)$.
\end{definition}
Note that the set $N(I)$ is closed under composition.

\begin{definition}[Counting laps]
\label{definition:counting_laps}
For a piecewise monotone map $h \in N(I)$, 
let $l(h)$ denote the number of laps of $h$.
The {\em lap growth rate} of $h$, denoted $\lambda(h)$, is
defined to be
$$\lambda(h) := \inf_{n\ge 1} l(h^n)^{1/n}$$ 
\end{definition}

Providing $\lambda(h)>1$, the map $h$ is semiconjugate to
a linear map:
\begin{theorem}[Preston \cite{Preston}, Theorem~11.4]\label{theorem:Preston_linear_model}
Let $h \in N(I)$ with $\lambda(h)>1$.  Then 
there exists a continuous, monotone (not necessarily strictly monotone) surjective map 
$\sigma : I \to I$ and $e \in N(I)$ with $e$ 
piecewise linear with slope $\lambda(h)$ such that 
$\sigma h = e\sigma$.
\end{theorem}

Our first task therefore is to show that $\lambda(h)>1$. In fact, we obtain an
explicit estimate in terms of $|s|$.
\begin{lemma}[Growth estimate]\label{lemma:growth_estimate}
For $h:I \to I$ as above,
$\lambda(h) \ge |s|^{-1} > 1/\sqrt{2}$.
\end{lemma}
\begin{proof}
We need to show that the number of points of discontinuity of 
$h^n$ grows like at least $|s|^{-n}$. Equivalently, we need to
show that the number of points of discontinuity of 
$H^n$ grows like at least $|s|^{-n}$.

For every interval $I \subset S^1$ the image $uI$ is a connected
subset of $\partial K$ of positive diameter. Applying either
$f^{-1}$ or $g^{-1}$ multiplies diameters by $|s|^{-1}$; thus, 
if $n$ is such for which  $|s|^{-n}\diam(uI) > \diam(K)$ then
$H^n|I$ must contain at least one point of discontinuity.

If we divide $S^1-\partial J_f$ into 
maximal disjoint subintervals $I_1,\cdots,I_k$ on 
which $H^n$ is continuous then $\diam(u I_i) < |s|^n\diam(K)$ for all $i$. 
On the other hand, $\cup_i uI_i = \partial K$ so
$\sum_i \diam(uI_i) \ge \diam(K)$. Thus $k\ge |s|^{-n}$ and the
lemma is proved.
\end{proof}

It follows that $H$ is semiconjugate to some unique $E_{\theta,\lambda}$ with
$\lambda \ge |s|^{-1}$ by a semiconjugacy commuting with $\iota$. 

To promote this semiconjugacy to a conjugacy we need to consider an analysis
of cases, depending on the number of discontinuity points of $h:I \to I$.

\begin{lemma}[Number of discontinuities]\label{lemma:number_of_discontinuities}
If $I_f$ entirely contains $J_f$ (equivalently, if $\L$ is infinite)
then $h$ has two points of discontinuity; otherwise it has one point of discontinuity.
\end{lemma}
\begin{proof}
The discontinuities of $h$ correspond to points of $\partial J_f$ in the interior of
$I_f$.
\end{proof}

First let's consider the case that $\L=\lbrace\ell \rbrace$, equivalently
that $h$ has one point of discontinuity. In this case $h$ is an example of
what is known as a Lorenz map. See \cite{Ding} for background.

\begin{definition}[Lorenz map]\label{definition:Lorenz}
A {\em Lorenz map} is a piecewise monotone map $h:I \to I$ whose domain
has one break point $c_1$, and for which there are $v_0,v_1$ so that
$h$ takes $[0,c_1)$ to $[v_0,1)$ and $(c_1,1]$ to $(0,v_1]$.

A Lorenz map is {\em topologically expanding} if the preimages of the 
break point under iterates of $h$ are dense in $[0,1]$.

A Lorenz map is {\em renormalizable} if there is a proper interval
$[a,b] \subset [0,1]$ containing $c_1$ and integers $n,m>1$ so that the map
$g:[a,b] \to [a,b]$ defined by 
$$g(x): = h^n(x) \text{ if } x \in [a,c_1) \text{ and } g(x): = h^m(x) \text{ if } x \in (c_1,b]$$
is conjugated to a Lorenz map by a (orientation-preserving) homeomorphism
$[a,b] \to [0,1]$.

A Lorenz map with constant slope is called a {\em $\beta$-transformation}.
\end{definition}
Evidently if $\L = \lbrace \ell \rbrace$ then $h$ is a 
topologically expanding Lorenz map. See Figure~\ref{h_graph_crossing}.

\begin{figure}[htpb]
\centering
\includegraphics[scale=0.7]{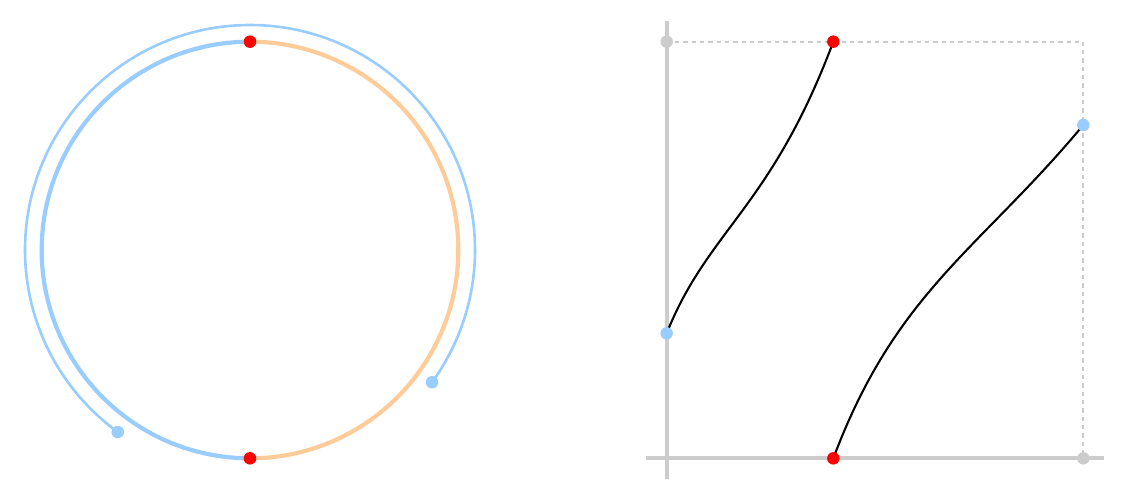}
\caption{When $\ell_f$ crosses $\ell$ the graph of $h$ has one discontinuity.}
\label{h_graph_crossing}
\end{figure}

The following theorem is a special case of \cite{Ding,Glendenning}. 

\begin{theorem}[\cite{Ding}, Appendix (10)]\label{theorem:lorenz}
An expanding Lorenz map which is not renormalizable 
is conjugate to a $\beta$-transformation, i.e.\/ to a 
piecewise linear map of constant slope.
\end{theorem}

\begin{proposition}[Crossing is linear]\label{proposition:crossing_linear}
Suppose $\L = \lbrace \ell \rbrace$; equivalently, suppose $\ell_f$ crosses $\ell$.
Then $h$ is conjugate to a $\beta$-transformation, and $H$ is conjugate to
$E_{\theta,\lambda}$ for some unique $\theta \in S^1$ and $\lambda \in [1,2]$.
\end{proposition}
\begin{proof}
By Theorem~\ref{theorem:lorenz} it suffices to show that $h$ is not renormalizable.
Suppose $h$ is renormalizable for some interval $[a,b] \subset [0,1]$ containing $c$,
and without loss of generality let's suppose $\diam(u[a,c_1)) \ge \diam(u(c_1,b]) \ge
\diam(u[a,b])/2$. The map $H^n$ sends $[a,c_1)$ continuously into $[a,b]$ or $\iota[a,b]$
for some $n\ge 2$.
On the other hand, $f^{-1}$ and $g^{-1}$ scale distances by $|s|^{-1}$ which is
strictly greater than $\sqrt{2}$ by hypothesis. Thus $\diam(u H^n[a,c_1)) > \diam(u[a,b])$
which is a contradiction.
\end{proof}

Now let's consider the case that $J_f \subset I_f$; equivalently, that
$\ell$ and $\ell_f$ do not cross. The map $h:I \to I$ has two points of discontinuity;
see Figure~\ref{h_graph_no_crossing}.

\begin{figure}[htpb]
\centering
\includegraphics[scale=0.7]{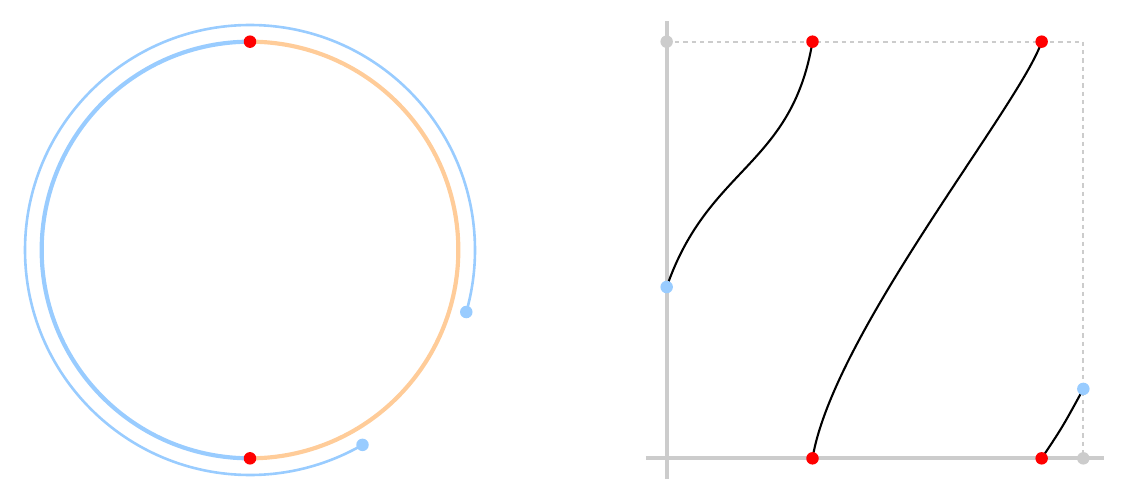}
\caption{When $\ell_f$ does not cross $\ell$ the graph of $h$ has two discontinuities.}
\label{h_graph_no_crossing}
\end{figure}

\begin{proof}[Proof of Theorem~\ref{theorem:linearization}]
If $s\in \partial \M$ then either $s=\pm i/\sqrt{2}$
in which case the theorem follows as in 
Example~\ref{example:i_on_sqrt_2} or $|s| < 1/\sqrt{2}$. Otherwise
$|s| < 1/\sqrt{2}$ is true by assumption.

If $\L$ is finite (equivalently, if $\ell$ crosses $\ell_f$) then
the theorem follows from Proposition~\ref{proposition:crossing_linear}.
Thus we may assume $\ell$ does not cross $\ell_f$, so that $h:I \to I$
has two points of discontinuity $0 < c_1 < c_2 < 1$ as in Figure~\ref{h_graph_no_crossing}.

By Theorem~\ref{theorem:Preston_linear_model} and Lemma~\ref{lemma:growth_estimate}
the map $h:I \to I$ is {\em semiconjugate} to a piecewise linear map
with constant slope by some monotone surjective map $\sigma:I \to I$.
The map $h$ has two discontinuity points
Let $T \subset (0,1)$ be the set of interior points of $I$ where $\sigma$ is locally constant.
Note that $T$ is a union of (countably many) open intervals.

\begin{lemma}[$T$ is invariant]\label{lemma:T_is_invariant}
The set $T$ is both forward and backward invariant under $h$ modulo the
discontinuity points $\cup c_i$. In other words, $h(T - \cup c_i) \subset T$
and $(h^{-1}(T)) \cap (0,1) \subset T$.
\end{lemma}
\begin{proof}
By assumption $\sigma$ conjugates $h$ to the piecewise linear map
$e:I \to I$ of constant slope $\lambda(h)$. In particular, $e \sigma (x,y) = \sigma h (x,y)$
for any open interval $(x,y)\subset I$ so $\sigma (x,y)$ is a single point (equivalently,
$(x,y)\subset T$) if and only if $\sigma h(x,y)$ is. 

If $p \in T - \cup c_i$ there
is an open interval $p \in (x,y) \subset (0,1) - \cup c_i$ for which
$h:(x,y) \to h(x,y)$ is a homeomorphism. Hence $h(p)$ is in $T$. Conversely,
for any $p \in T$ there is an open interval $p \in (x,y)$ in the interior of $I$,
and $h^{-1}(x,y) \cap (0,1)$ is a union of open intervals disjoint from $\cup c_i$. 
In particular, for each $q\in h^{-1}(x,y) \cap (0,1)$
there is an open interval $q \in (x',y')$ with $h:(x',y') \to (x,y)$ a homeomorphism onto
its image, and therefore $(x',y')$ (and consequently $q$) is in $T$.
\end{proof}

We would like to show that $T$ is empty, which would imply that $\sigma$ is a homeomorphism
and the theorem would be proved. So let's suppose $T$ is nonempty, and we will derive
a contradiction. Let's also introduce the notation $T^c:=(0,1)-T$. Note that by
Lemma~\ref{lemma:T_is_invariant} we have $(h T^c) \cap (0,1) \subset T^c$ 
and $(h^{-1} T^c) \cap (0,1) \subset T^c$. Notice too that by the definition of $\sigma$,
the set $T^c$ is relatively closed in $(0,1)$ and has no isolated points.

Suppose $T$ is disjoint from $\cup c_i$. Then $h(T) \subset T$ by 
Lemma~\ref{lemma:T_is_invariant} and furthermore $h$ is a homeomorphism when restricted
to each component of $T$. But $h$ is strictly expanding on $[0,1]$, so if $I$ is
a component of $T$ of largest diameter, $h(I)$ would be a component of even larger
diameter, which would be a contradiction. So without loss of generality we may assume $T$ 
contains one of the discontinuity points, say $c_1$. Since $T$ is open, it
contains an open neighborhood of $c_1$, and therefore by Lemma~\ref{lemma:T_is_invariant}
$T \cup \lbrace 0,1\rbrace$ contains open neighborhoods of $0$ and $1$. This implies
in turn by another application of Lemma~\ref{lemma:T_is_invariant} that $c_2$ is in $T$. 
Thus $T^c$ is a Cantor set properly contained 
in the interior of $(0,1)$ and disjoint from both $c_1$ and $c_2$.

The intersections of $T^c$ with each of the three intervals $(0,c_1)$, $(c_1,c_2)$ and
$(c_2,1)$ are contained in three minimal closed proper intervals that we label $X_1$, $X_2$
and $X_3$. We assume these are all nonempty; the case that one or more are empty is easier,
and may be treated by a similar argument. In particular,
Write $(0,1)- \cup X_i$ as the union of four disjoint open intervals $Y_0$, $Y_1$, $Y_2$
and $Y_3$ from left to right, each in $T$. For $i=1,2$ the interval $Y_i$ contains the
point $c_i$ of discontinuity, and therefore $h (Y_i-c_i) \subset Y_0 \cup Y_3$. Note
that every component of $T$ is either contained in one of the $Y_i$, or in the interior
of some $X_j$. See Figure~\ref{partition_of_I}.

\begin{figure}[htpb] 
\labellist
\small\hair 2pt
\pinlabel $Y_0$ at 35 15
\pinlabel $Y_1$ at 100 15
\pinlabel $Y_2$ at 195 15
\pinlabel $Y_3$ at 220 15
\pinlabel $X_1$ at 66 42
\pinlabel $X_2$ at 152 42
\pinlabel $X_3$ at 208 42
\endlabellist
\centering
\includegraphics[scale=0.7]{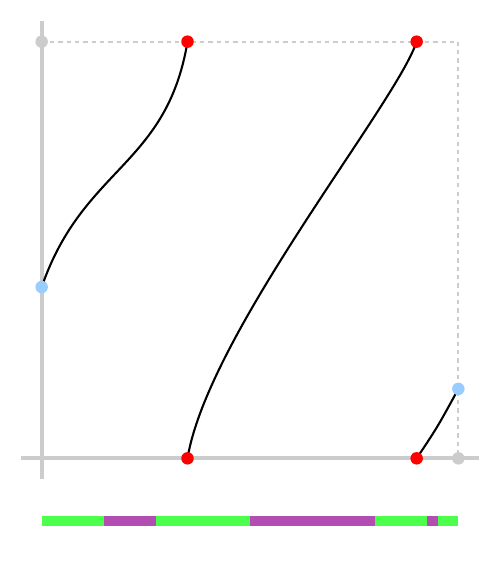}
\caption{The interval $(0,1)$ is partitioned into the closed intervals
$X_1,X_2,X_3$ and the open intervals $Y_0,Y_1,Y_2,Y_3$.}
\label{partition_of_I}
\end{figure}
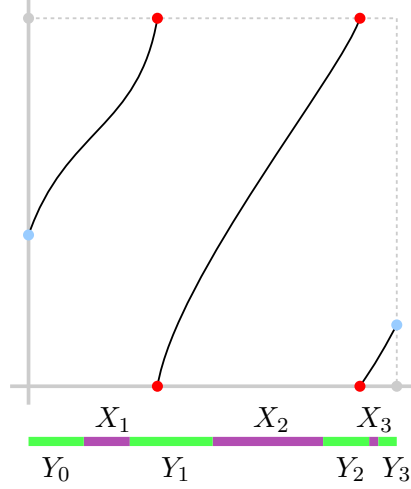

On the other hand, the iterates $h^jY_0$ and $h^jY_3$ must strictly increase in length 
as $j$ increases if they are
disjoint from the points of discontinuity, and therefore there are least integers $n,m\ge 1$
with $h^nY_0 \subset Y_1 \text{ or } Y_2$ and $h^mY_3 \subset Y_1 \text{ or } Y_2$.
Hence the restriction of $h^{n+1}$ to $Y_0$ has exactly one discontinuity
point which is of the form $h^{-n} c_i$ for $i=1$ or $2$, and $h^{n+1}$ maps
the complementary components of this point in $Y_0$ continuously into $Y_0$ and $Y_3$,
and likewise for $h^{m+1}|Y_3$.

Each $Y_i$ maps under $\alpha$ to an interval in $S^1$ and then by $u$ to a connected
subset of $\partial K$, and we suppose for
simplicity that $$\diam(u \alpha Y_0) \ge \diam(u \alpha Y_3)$$
(the case of the opposite inequality is treated perfectly analogously). 
The map $h^{n+1}|Y_0$ is continuous on the subintervals
to the left and to the right of its unique discontinuity point; 
let $Y$ be whichever subinterval has bigger
diameter under $u \alpha$. Then $\diam(u \alpha Y) \ge \diam(u \alpha Y_0)/2$ but
$h^{n+1}$ maps $Y$ continuously either into $Y_0$ or $Y_3$, and therefore 
$$|s|^{-n-1}\diam(u\alpha Y) =\diam (u \alpha h^{n+1} Y) \le \diam(u\alpha Y_0)$$
But this contradicts $n\ge 1$ and $|s|<1/\sqrt{2}$. This contradiction
shows that $T$ is empty after all, and $\sigma$ is a homeomorphism, and the theorem is proved.
\end{proof}

\begin{remark}
It is remarkable that the condition $|s|<1/\sqrt{2}$ has such strong dynamical 
consequences for $H$. If $s \in \M$ is splittable 
but has $|s| \ge 1/\sqrt{2}$ we still obtain a semi-conjugacy of $H$ 
to a linear model, and thereby obtain invariants $\lambda$ and $\theta$, but we
do not know that these invariants capture the full dynamics of $H$.
\end{remark}

\begin{lemma}[$\lambda =2$]\label{lemma:lambda=2} 
Let $s\in \partial \M$ have associated parameters $\theta,\lambda$. 
If $\lambda = 2$ then $\ell$ is a dynamical cut leaf, and consequently
$\Lambda$ is a dendrite.
\end{lemma}
\begin{proof}
If $\lambda = 2$ then $I_f$ and $I_g$ are degenerate. In particular, 
the forward iterates $H^n\ell$ are defined for all $n$ so that $\ell$ is a
dynamical cut leaf. The fact that $\Lambda$ is a dendrite follows from
Lemma~\ref{lemma:degenerate} or Proposition~\ref{proposition:finite_cut_leaf}.
\end{proof}

\section{Linear dynamics and laminations}
\label{section:linear_dynamics}

Recall in \S~\ref{section:linear_model} we described a 2-parameter
family of (discontinuous) linear dynamical systems on $S^1$.
For $\lambda \in [1,2]$ and $\theta \in S^1$ we defined a map
$E_{\theta,\lambda}:S^1 \to S^1$ in Definition~\ref{definition:linear_model}.
If we think of $S^1$ as $\R/2\pi$, then 
$E$ acts on $[-\pi/2,\pi/2]$ by
$$E:t \to \lambda t + \theta \text{ for } t \in [-\pi/2,\pi/2]
\text{ and } E:t \to \pi + \lambda (t-\pi) + \theta 
\text{ for } t \in [\pi/2,3\pi/2]$$
(recall that $E$ is many-valued on $\pm \pi/2$).

By Theorem~\ref{theorem:linearization} for any $s\in \partial \M$
the dynamics of $H,f,g$ on $S^1$ is conjugate to the dynamics of
$E_{\theta,\lambda},F_{\theta,\lambda},G_{\theta,\lambda}$ 
on $S^1$ for some unique $\theta,\lambda$.
Thus, conditional on Conjecture~\ref{conjecture:A} we may determine
the structure of the set of cut points of $\Lambda$ entirely from
the parameters $\theta$ and $\lambda$.
In the sequel we suppress subscripts whenever possible and just
write $E,F,G$. We also write $\SS^E$ for the free abstract semigroup
on the generators $F$ and $G$.

On the other hand, there are many parameters $\theta,\lambda$
that do not arise from $s\in \partial M$. For any
$\theta,\lambda$ we may construct leaves $\ell,\ell_F,\ell_G$
and a lamination $\L$ analogous to those constructed in
\S~\ref{subsection:dynamical_lamination}, and define limit leaves
and dynamical cut leaves in the obvious way. It does not make
sense to define cut leaves as distinct from dynamical cut leaves,
since there is no analog of the map $u$ in the general context. 

The lamination $\L$ is finite (and consists only of $\ell$) if
and only if $\ell$ crosses $\ell_F$. If $\ell$ does not
cross $\ell_F$, then as in Theorem~\ref{theorem:interaction}
we may find a word $w\in \SS^E$ of least length for which $w\ell$
interacts with $\ell_F$ and ask whether
$w\ell$ hides $\ell_F$ from $\ell$, in which case there are limit leaves
and therefore dynamical cut leaves, or $w\ell$ crosses $\ell_F$, 
in which case there are no dynamical cut leaves.

In this broader context there are many examples of parameters $\theta,\lambda$
for which there are dynamical cut leaves that are {\em not} limit leaves;
see Example~\ref{example:cut_not_limit}. Conjecture~\ref{conjecture:cut_is_limit} 
says that this cannot happen for $\theta,\lambda$ associated 
to $s\in \partial \M$.

\subsection{A simplifying assumption}
\label{section:linear_simplifying_assumption}

Throughout this section, we will mostly assume that $|\theta| \le \pi/2$
where convenient. In other words, the domain of $F$ is centered at a 
point which is inside its image. If $|\theta|>\pi/2$ one
obtains analogous results by proofs which are also entirely analogous.
Thus making this assumption saves us from having to switch between
cases. This simplifying assumption is analogous to the simplifying
assumption on $s$ we have frequently made that at least one endpoint 
of $I_f$ lies in $J_g$.

\subsection{Innermost leaves and the Douady--Hubbard--Thurston lamination}\label{subsection:innermost_leaves_and_DHT}

The set of pairs $\lambda \in [1,2]$ and $\theta \in S^1$ is parameterized
by an annulus. We let $\sigma$ denote a typical pair $(\theta,\lambda)$ 
and denote the parameter space of pairs by $\P$.
Within this annulus we distinguish interesting subsets
corresponding to the following features of the dynamical system
$E_{\theta,\lambda}$ on $S^1$, including

\begin{enumerate}
\item{$\P_\infty$, the parameters for which $\L$ is infinite
(equivalently, $\L$ has more than one leaf); and}
\item{$\P_\limit$, the parameters for which $\L$ contains a limit leaf.}
\end{enumerate}

Because of the symmetry of $\P$ under $\theta \to \theta+\pi$ we restrict attention
to the subset where $\theta \in [-\pi/2,\pi/2]$.
Figure~\ref{triangle_list} indicates $\P_\infty$ in blue and $\P_\text{limit}$
in green. 

\begin{figure}[htpb]
\centering
\includegraphics[scale=0.2]{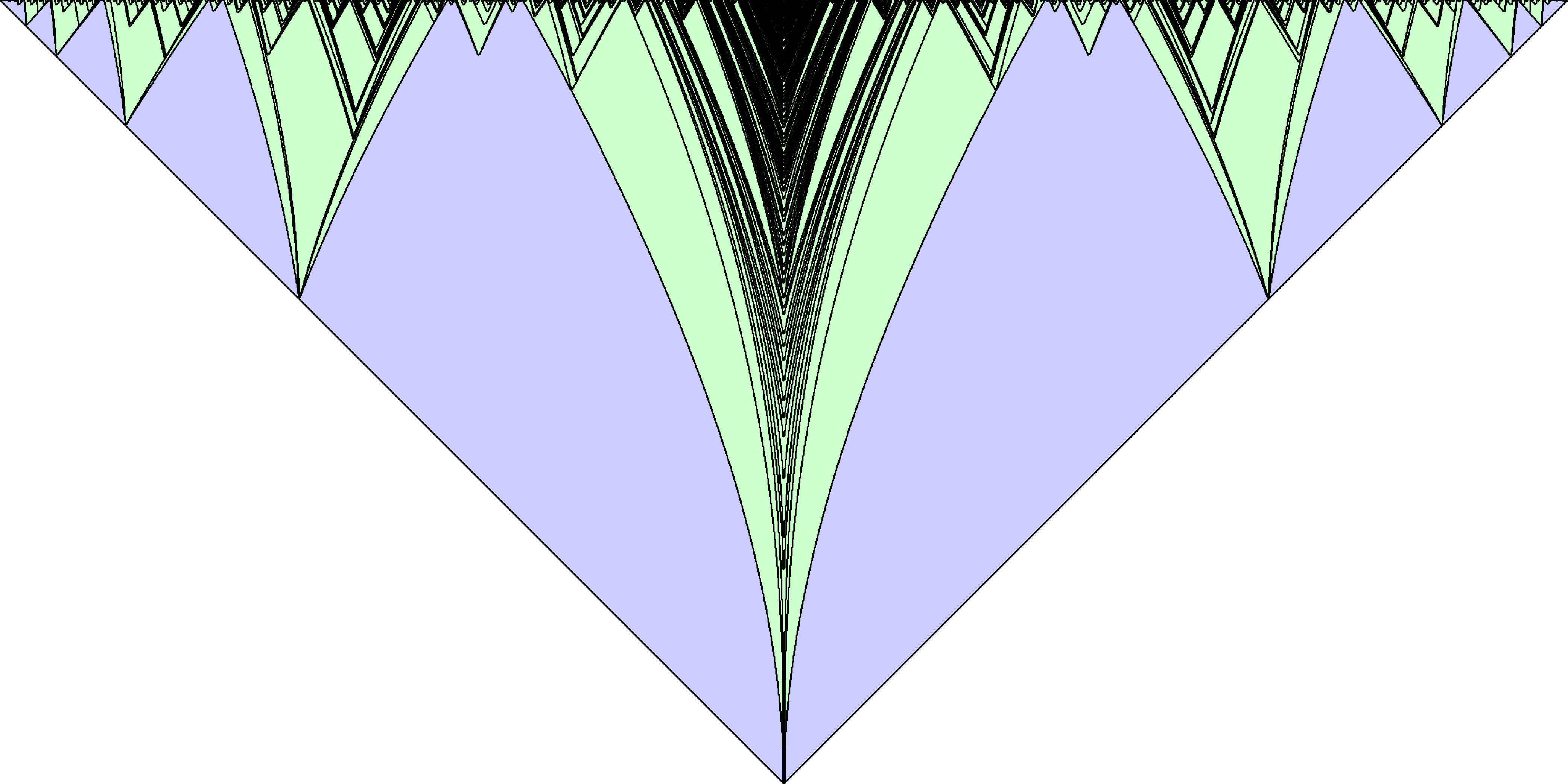}
\caption{Half a fundamental domain for the annulus $\P$. 
The angle $\theta\in [-\pi/2,\pi/2]$ parameterizes the horizontal
coordinate, and the multiplier $\lambda\in [1,2]$ parameterizes the vertical 
coordinate. The subspace $\P_\infty$ is in blue and
the subspace $\P_\limit$ is in green. The meaning of the `inner' green triangles
will be explained in the sequel.}
\label{triangle_list}
\end{figure}

These sets are invariant under $\theta \to - \theta$. 
The set $\P_\infty$ is very easy to describe: it is the set of $\lambda,\theta$
for which $\ell$ does not cross $\ell_G$. Since the endpoints of $\ell_G$ are
$\pm\lambda\pi/2 + \theta$ it follows that $\P_\infty$ is the blue triangle
$\lambda \ge 1 + 2|\theta|/\pi$.

\medskip

The set $\P_\limit$ is much more complicated, and its structure turns out (somewhat
surprisingly) to be related to the Douady--Hubbard--Thurston lamination model
for the abstract (ordinary) Mandelbrot set.

Recall that the set $\P_\limit$ parameterizes
$\theta,\lambda$ for which some leaf of $\L - \ell$ hides $\ell_G$. By
Theorem~\ref{theorem:interaction} this is the union of the `outermost' green
triangles in Figure~\ref{triangle_list}, and are associated to the words
$G,GF,GFG,GFGF,GFGFG,\cdots$ of shortest length for which $w\ell$ can hide $\ell_G$, by
Theorem~\ref{theorem:interaction}.

What is the meaning of the other (`inner') green triangles? 
Fix $\lambda,\theta$ and suppose $w\ell$ hides $\ell_G$ for some word $w$
beginning with $f$. Let $\L_w$ be the
finite set of leaves of the form $v\ell$ where $v$ is a suffix of $w$ or $v^c$ is a
suffix of $w$ where as before, 
$v^c$ is the word obtained from $v$ by replacing each $F$ by $G$ and
vice versa. Note that $v^c\ell = \iota v\ell$. 
In other words, $\L_w$ is the finite symmetric (under $\iota$) set of leaves obtained from
$\ell$ by applying the letters of $w$ (from the right) one by one, together with the
images of these leaves under $\iota$. 

\begin{definition}[Innermost]\label{definition:innermost}
For $\lambda,\theta$ and $w$ say that $w\ell$ is {\em innermost} in $\L_w$ if 
$w\ell$ hides $\ell_G$ from every other leaf of $\L_w$ (and likewise $w^c$
hides $\ell_F$ from every other leaf of $\L_w$).
\end{definition}

\begin{lemma}[Triangles]\label{lemma:triangles}
Fix $\lambda_0,\theta_0$ and $w$ and suppose $w\ell$ is innermost in $\L_w$, and let
$\Delta$ be the connected region in the $\lambda,\theta$ plane containing
$\lambda_0,\theta_0$ with these properties. Then $\Delta$ is a curvilinear green
`triangle' with one `horizontal' side $\Delta^*$ on the line $\lambda=2$, and two 
`vertical' sides $\Delta^\pm$
which are (contained in) real algebraic curves, each of which is a solution of 
$2\theta/\pi = p(\lambda)/(\sum_{j=0}^n \lambda^j)$
for some polynomial $p$ with integer coefficients of degree $|w|+1$ in $\lambda$.
The `vertical' side $\Delta^-$ resp. $\Delta^+$ consists of exactly the parameters 
in $\Delta$ for which $w\ell$ shares its left resp. right endpoint in common with 
$\ell_G$ (as oriented by the interval $J_F$ that contains them both).
\end{lemma}
\begin{proof}
Consider the finite dynamical set of words $\SS_w$ of the form $v$ where $v$ or $v^c$ 
is a suffix of $w$, so that $\L_w$ is the orbit of $\ell$ under $\SS_w$. Since
$w\ell$ resp. $w^c\ell$ are innermost, and hide $\ell_G$ resp. $\ell_F$ from the other
leaves, and since $G$ resp. $F$ is defined and continuous on the complement of $\ell_G$ resp.
$\ell_F$, it follows that $\L_w$ varies {\em continuously} as a function of
$\lambda,\theta$ until some endpoint of $w\ell$ collides with $\ell_G$. Collision from
the left resp. right determines the curve $\Delta^-$ resp. $\Delta^+$. We claim these
are (contained in) real algebraic curves of the desired kind.

Let's change variables on $S^1$ by choosing coordinates $\tau:=2\theta/\pi$
so that $S^1 = \R/4\Z$. In these coordinates we have, on their 
respective domains of definition,
$$G^{-1}: x \to \lambda x + \tau  \text{ and }
F^{-1}: x \to \lambda(x-2) + 2 + \tau$$
and similarly,
$$G: x \to \lambda^{-1}(x-\tau) \text{ and } F: x \to \lambda^{-1}(x - \tau - 2) + 2$$
In these coordinates, the endpoints of $\ell$ are $\pm 1$ and therefore
the endpoints of $\ell_F$ are $2 - (2\pm 1)\lambda + \tau$. 
Likewise the endpoints of $w_n\ell$ are inductively of the form 
$\tau\cdot \sum_{j=1}^n \lambda^{-j} + p(\lambda^{-1})$ for
some degree $n$ integer polynomial $p$ whose coefficients depend on
which lift of $\tau$ one chooses at each stage. Equating these gives
$$\tau\cdot \sum_{j=1}^n \lambda^{-j} + p(\lambda^{-1}) = \tau + 2 - (2\pm 1)\lambda$$
and rearranging gives the desired identity. 
\end{proof}

Note that the `vertex' of $\Delta$ where $\Delta^-$ and $\Delta^+$ intersect 
occurs exactly when $w\ell = \ell_G$.

Let $\P_w$ denote the subset of $\P$ for which $w\ell$ hides $\ell_G$ and is
innermost in $\L_w$. Thus $\P_w$ consists of a finite collection of triangles
$\Delta$ as in Lemma~\ref{lemma:triangles}. These are the green triangles in
Figure~\ref{triangle_list}.

\begin{lemma}[Triangles nest]\label{lemma:triangles_nest}  
Let $\Delta,\Delta'$ be triangles in $\P_w,\P_{w'}$. Then either $\Delta$ and $\Delta'$
are disjoint or one is entirely contained inside the other.
\end{lemma}
\begin{proof}
Suppose there is a parameter $\lambda,\theta$ in $\Delta \cap \Delta'$. For that
parameter we can simultaneously form $\L_w$ and $\L_{w'}$. Both $w\ell$ and
$w'\ell$ hide $\ell_G$ from $\ell$ and therefore without loss of generality
$w\ell$ hides $\ell_G$ from $w'\ell$ so that in fact $w\ell$ is innermost in
$\L_w \cup \L_{w'}$. But then evidently $\Delta \subset \Delta'$ by
Lemma~\ref{lemma:triangles}.
\end{proof}

Recall for each component $\Delta$ of each $\P_w$ that $\Delta^*$ denotes the
`horizontal' edge of $\Delta$, which is an interval in the circle $\lambda = 2$.
Lemma~\ref{lemma:triangles_nest} implies that 
the collection of intervals $\Delta^*$ ranging over all components
of all $\P_w$ are pairwise disjoint or nested in the circle $\lambda=2$. 
Notice by the way that the
endpoints of these intervals are all numbers of the form $p\pi/(2^{n+2}-2)$ for
various integers $p$, by Lemma~\ref{lemma:triangles}.

Take a fundamental domain $\theta \in [-\pi/2,\pi/2]$ in the circle $\lambda = 2$
and map this to $\R/2\pi\Z$ by $\theta \to \pi + 2\theta$. The endpoints of each 
interval $\Delta^*$ map to (the endpoints of) a leaf in $\R/2\pi\Z$, and by
Lemma~\ref{lemma:triangles_nest} this collection of leaves is non-crossing, and
form a lamination of $S^1$ that we denote $\L_{DHT}$
with closure $\overline{\L}_{DHT}$. See Figure~\ref{triangles_wrapped_around_M}.

\begin{figure}[htpb]
\centering
\includegraphics[scale=0.2]{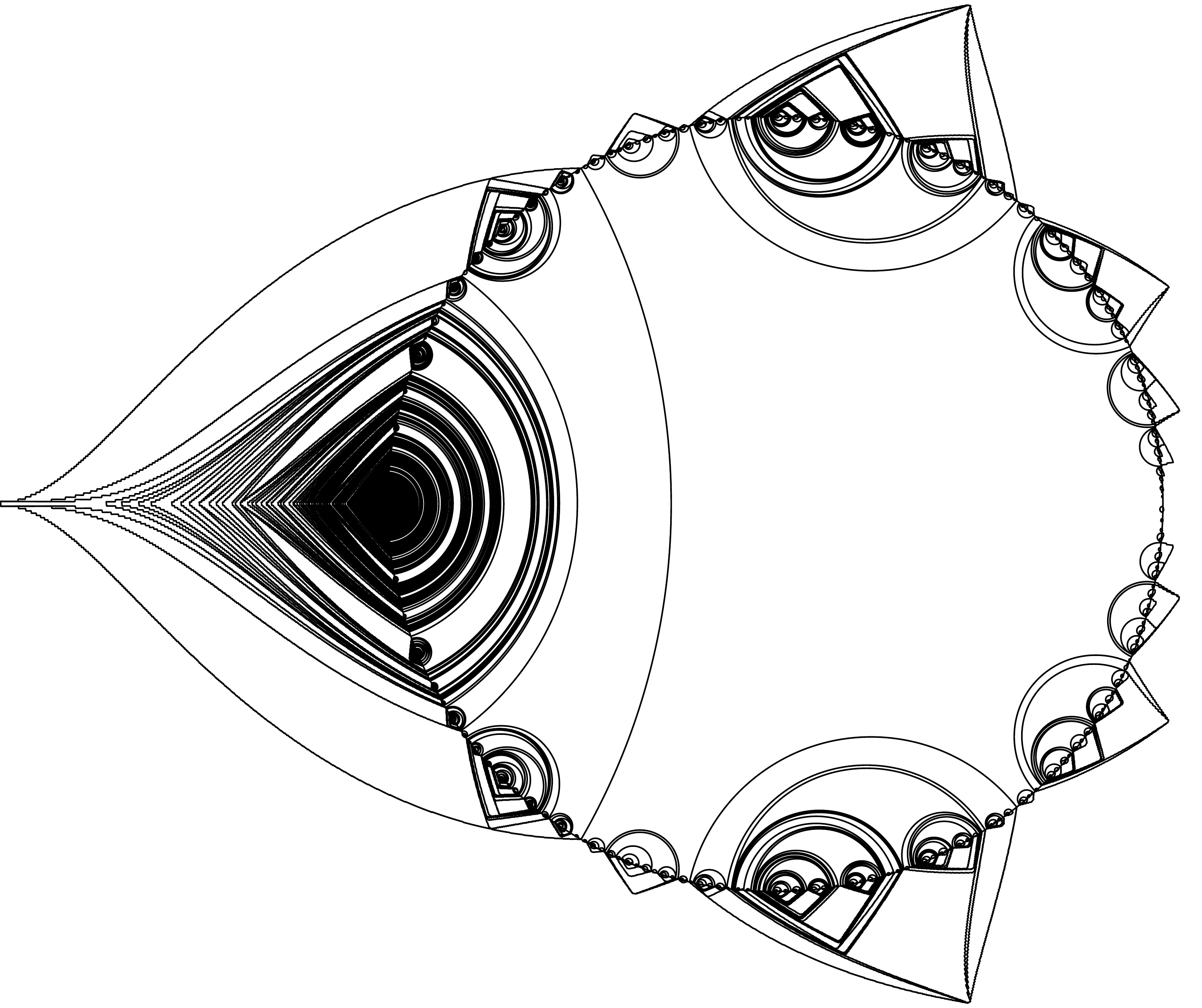}
\caption{Half a fundamental domain for the annulus $\P$ wrapped around a circle. 
Endpoints of intervals $\Delta^*$ become endpoints of leaves of the lamination $\L_{DHT}$.}
\label{triangles_wrapped_around_M}
\end{figure}

This lamination should look very familiar. It is the Douady--Hubbard--Thurston
quadratic minor lamination. See \cite{Thurston, Douady_Hubbard_1, Douady_Hubbard_2}
for a description of this lamination in terms of quadratic minors, and its relation
to the external ray parameterization of the ordinary Mandelbrot set $\M_\ord$.

\begin{theorem}[$\overline{\L}_{DHT}$ and $\M_\ord$]\label{theorem:DHT_lamination}
The lamination $\overline{\L}_{DHT}$ is the Douady--Hubbard--Thurston 
quadratic minor lamination, and the quotient of $\overline{\D}$ by $\overline{\L}_{DHT}$
is the abstract Mandelbrot set.
\end{theorem}
\begin{proof}
For $\lambda=2$ we may define a map $E':S^1 \to S^1$ by $E'=G^{-1}$ on
$J_G$ and $E'=\iota F^{-1}$ on $J_F$. Then $E'$ is continuous, of constant slope $2$.
In fact, if we let $R_\theta:S^1 \to S^1$ be rotation
by $\theta$, then $R_\theta$ conjugates $E'$ to $z \to z^2$. 

Let $\L$ be the $E$-invariant dynamical lamination $\L$.
By the $\iota$-invariance of $\L$, the image $R_\theta(\L)$ is the invariant
dynamical lamination for $z \to z^2$ with critical leaf $R_\theta(\L)(\ell)$,
in the terminology of \cite{Thurston}. 

Each arc in $S^1$ bounded by a leaf of $\L_{DHT}$ is associated to an
interval $\Delta^*$ for some component $\Delta$ of some $\P_w$. When $\lambda=2$
(i.e.\/ on the interval $\Delta^*$) the `leaves'
$\ell_F$ and $\ell_G$ degenerate to antipodal points; after conjugacy by $R_\theta$
the leaf $\ell_G$ becomes the critical value $v$ 
(i.e.\/ the image of the critical leaf under $z \to z^2$). Since
the degenerate leaf $\ell_G$ is just the point $\pi+\theta$ its image under
$R_\theta$ is $\pi+2\theta$; this is the `meaning' of the map from $\theta$
coordinates to Douady--Hubbard--Thurston coordinates on $S^1$.

\medskip 

Let the $\theta$ coordinates of $\Delta^*$ be $[\theta^-,\theta^+]$.
This interval corresponds to a leaf 
$\lbrace \pi+2\theta^-,\pi+2\theta^+\rbrace$
in $\L_{DHT}$ by definition. We explain the meaning of this leaf in
dynamical terms.

For each $\theta \in [\theta^-,\theta^+]$ we get a dynamical lamination 
$\L(\theta)$ for which $w\ell$ hides the point $\ell_g$ and is innermost in
$\L_w(\theta)$. If we think of
$\ell_G$ morally as the `leaf' $G^{-1}\ell$ then we claim that the
sequence of `leaves'
$$\ell_G = G^{-1}\ell, w\ell, (wG)w\ell, (wG)^2w\ell,\cdots$$
are all nested between $\ell_G$ and $\ell$, and therefore 
converge to a limit $\mu_\theta:=\lim_{n \to \infty} (wG)^nw\ell$ which is a
stable fixed leaf for $wG$. To see this, note that when $\theta =\theta^\pm$,
the point $\ell_G$ is one of the endpoints of $w\ell$ which is fixed by $wG$.
Since $(wG)w\ell$ can't cross $\ell$, and since it has one endpoint on
$\ell_G$, it hides $w\ell$ from $\ell$, and so on for the other iterates by
induction. Thus the claim is proved for this value of $\theta$, and evidently
it is therefore true by continuity throughout $[\theta^-,\theta^+]$, since
for no intermediate value can any iterate $(wG)^n\ell$ ever cross $\ell$.

In the same way we can apply powers $(wG)^n$ to each leaf of $\L_w(\theta)$ and 
obtain a finite lamination $\L'_w(\theta)$ containing $\mu$ as an innermost leaf.
Applying letters of $wG$ in order from right to left permutes the leaves of
$\L'_w(\theta)$ in a finite cycle.

Applying $R_\theta$ to $\L'_w(\theta)$ gives a finite lamination $\L''_w(\theta)$ 
on which $z \to z^2$
is periodic of period $|w|+1$. The leaf $\nu(\theta):=R_\theta \mu$ is innermost in this
finite lamination. Since this is true for every $\theta$ in the interval 
$[\theta^-,\theta^+]$, and since periodic leaves of $z \to z^2$ of fixed period 
are isolated, it follows that $\nu(\theta)$ and $\L''_w(\theta)$ are 
{\em independent} of $\theta$ in the interval $[\theta^-,\theta^+]$; i.e.\/ we
may refer unambiguously to $\nu$ and $\L''_w$. Taking $\theta=\theta^\pm$
we see that the endpoints of $\nu$ are precisely 
$\lbrace \pi + 2\theta^-,\pi + 2\theta^+\rbrace$. Thus $\nu$ is itself the
leaf of $\L_{DHT}$ associated to the interval $\Delta^*$.

By construction, the finite lamination $\L''_w$ is forward invariant. Since $\nu$ is
innermost, it is the image of the `largest' leaf of $\L''_w$ (the `major'); in other
words, $\nu$ is the minor in the terminology of Thurston. 
There is a unique (backward) invariant dynamical lamination 
containing $\L''_w$ for which no leaf crosses the major (or its image under
$\iota$). In particular $\nu$ is a leaf of the Douady--Hubbard--Thurston lamination,
and running this construction backwards, all the quadratic minor leaves arise this way.
\end{proof}

Let us spell out this correspondence in a few special cases.

\begin{example}[$w=F$]
First consider the case $w=F$. The region $\P_w$ consists of a single 
triangle $\Delta$ in this case. The degenerate leaf $\ell_G$ is just
the point $\pi+\theta$. The leaf $F\ell$ has endpoints $(\pi/2 +\theta -\pi)/2 + \pi$
and $(-\pi/2 +\theta - \pi)/2 + \pi$.
Thus $F\ell$ has $\ell_G$ as an endpoint when $\theta = \pm \pi/6$ so that
$[\theta^-,\theta^+] = [-\pi/6,\pi/6]$ in this case.

For $\theta = -\pi/6$ the map $FG$ has attracting fixed points at 
$\pi+\theta = 5\pi/6$ and $3\pi/2$, and the leaves $(fg)^nf\ell$ converge to
the leaf $\mu_{-\pi/6}:=\lbrace 5\pi/6,3\pi/2\rbrace$ which is evidently 
a stable fixed leaf of $FG$. Applying $R_\theta$ takes
$\mu_{-\pi/6}$ to $R_{-\pi/6}\mu_{-\pi/6} = \lbrace 2\pi/3,4\pi/3\rbrace$; see
Figure~\ref{DHT_example_1}.

\begin{figure}[htpb]
\centering
\includegraphics[scale=0.7]{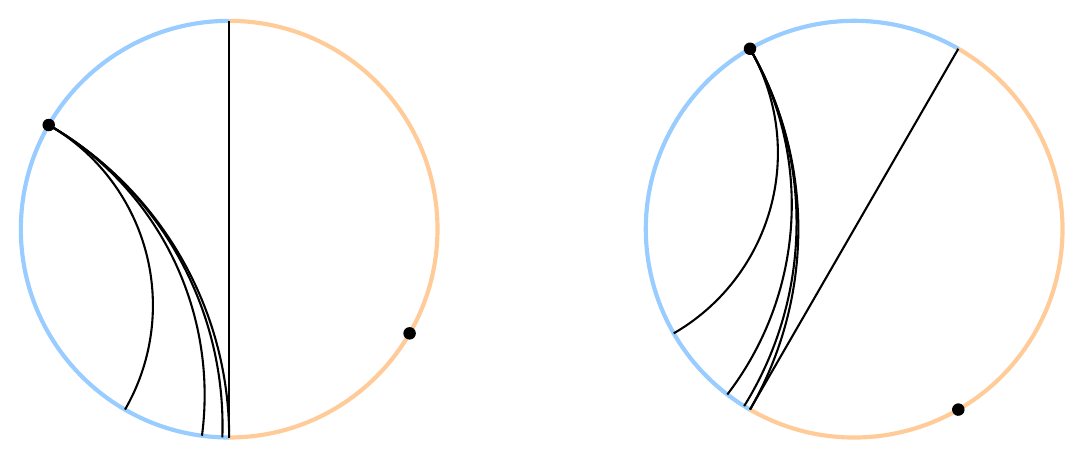}
\caption{The stable fixed leaf $\mu_{-\pi/6}=\lbrace 5\pi/6,3\pi/2\rbrace$ of $FG$ 
is taken by $R_{-\pi/6}$ to the unstable fixed leaf $\lbrace 2\pi/3,4\pi/3\rbrace$
of $z \to z^4$.}
\label{DHT_example_1}
\end{figure}

For $\theta = \pi/6$ the map $FG$ has attracting fixed points at
$\pi+\theta = 7\pi/6$ and $\pi/2$, and now $(FG)^nF\ell$ converges to the
leaf $\mu_{\pi/6}:=\lbrace \pi/2,7\pi/6\rbrace$ which is a stable fixed leaf of $FG$. 
Applying $R_\theta$ takes
$\mu_{\pi/6}$ to $R_{\pi/6}\mu_{\pi/6} = \lbrace 2\pi/3,4\pi/3\rbrace = R_{-\pi/6}\mu_{-\pi/6}$.

The leaf $\nu:=R_{\pi/6}\mu_{\pi/6}=R_{-\pi/6}\mu_{-\pi/6}$ is an 
(unstable) fixed leaf of the square of $z \to z^2$. It corresponds to
the indifferent periodic point of period 2 in the Julia 
set of $z \to z^2+c$ for $c=-0.75$.
\end{example}

\begin{example}[$w=FG$]
Let's consider another example, $w=FG$. The region $\P_w$ consists of
three triangles; one of them intersects $\lambda=2$ in the interval 
$[\theta^-,\theta^+]=[3\pi/14,5\pi/14]$.
For $\theta$ in this interval the leaves $(FGG)^nFG\ell$ converge to
an attracting fixed leaf $\mu_\theta$ for $fgg$, and $R_\theta\mu_\theta$ 
is a repelling periodic leaf of period $3$ for $z \to z^2$. The associated
Julia set is $z \to z^2+c$ for $c\sim -0.125 -0.649i$, the root point of the
bulb whose center is the corabbit.
\end{example}

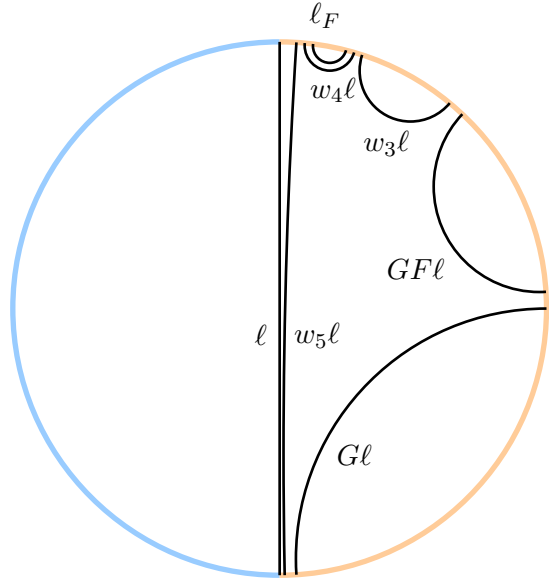
\begin{figure}[htpb] 
\labellist
\small\hair 2pt
\pinlabel $\ell$ at 103 100
\pinlabel $w_5\ell$ at 124 100
\pinlabel $G\ell$ at 138 55
\pinlabel $GF\ell$ at 161 125
\pinlabel $w_3\ell$ at 150 172
\pinlabel $w_4\ell$ at 130 192
\pinlabel $\ell_F$ at 127 220
\endlabellist
\centering
\includegraphics[scale=1.0]{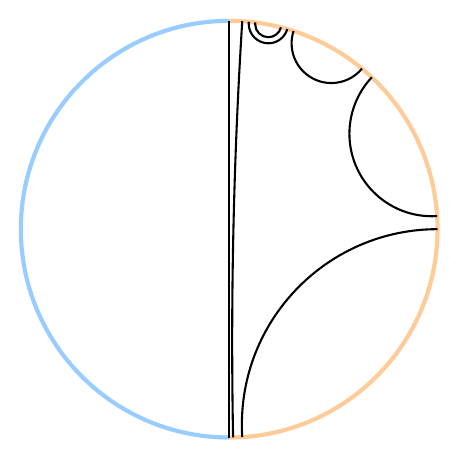}
\caption{An example in $\Delta_4$. The sequence of leaves $w_n\ell$ spirals
positively and accumulates on an ideal 5-gon.}
\label{DHT_example_2}
\end{figure}

\begin{example}[Diagonal cut leaves]\label{example:cut_not_limit}
Suppose for some $\theta,\lambda$ that the lamination $\L$ has a complementary
ideal polygon $P$ with at least 4 boundary leaves, all of which are limit leaves
(and consequently also dynamical cut leaves). In particular,
all the ideal vertices of $P$ have the same trajectories, and therefore the
diagonals are all dynamical cut leaves; however such diagonals cannot be nontrivial
limits of ordinary leaves, and therefore if they are not in $\L$ they are not limit
leaves at all.

For each $n$, let $w_n$ be the prefix of $(GF)^\infty$ of length $n$, so that
$w_0$ is empty, $w_1 = G$, $w_2 = GF$, $w_3=GFG$ and so on. For each $n$ there is a triangle
component $\Delta_n$ of $\P_{w_n}$ for which $w_n \ell$ hides $\ell_G$ from $\ell$, and
the lamination $\overline{\L}$ contains a limit $(n+1)$-gon whose leaves are limits
from outside of the leaves $w_{(n+1)m+i}\ell$ for fixed $i$ mod $(n+1)$ and $m \to \infty$.
See Figure~\ref{DHT_example_2} for an example with $n=4$.

For example, taking $\lambda \sim 1.915$ and $\theta \sim 1.295$ gives the lamination
$\overline{\L}$ depicted in Figure~\ref{4_sided_polygon_L}. This parameter is contained
in the triangle $\Delta_3$ and therefore $\overline{\L}$ has (ideal) 4-gons 
whose boundary leaves are limit leaves. In particular, the diagonals are 
dynamical cut leaves which are not limit leaves. Extensive computer experiments have failed
to find $s\in \partial \M$ for which $(\lambda,\theta)$ is contained in any
$\Delta_n$ for $n>2$; this motivates Conjecture~\ref{conjecture:cut_is_limit}, but we
do not have any real theoretical justification for this.

\begin{figure}[htpb]
\centering
\includegraphics[scale=0.2]{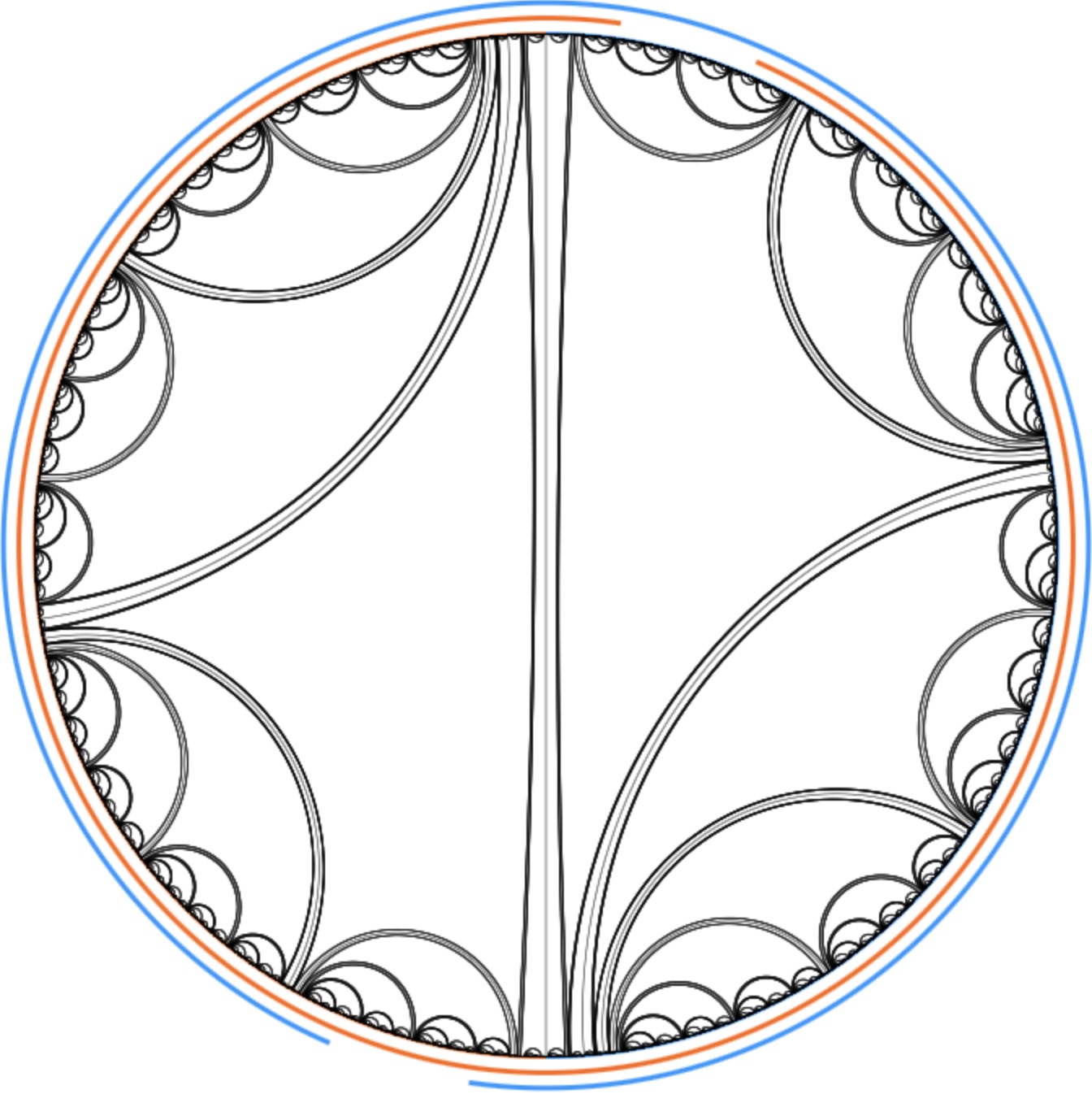}
\caption{The lamination $\L$ has 4-sided polygons of limit leaves. The diagonals are
dynamical cut leaves which are not limit leaves.}
\label{4_sided_polygon_L}
\end{figure}
\end{example}

\subsection{Inclusion graphs}\label{subsection:inclusion_graphs}

The goal of this section is to give an algorithm that recursively generates a
directed graph so that dynamical cut leaves in $S^1$ 
correspond to infinite directed paths in the graph. For $\lambda,\theta$ associated
to a parameter $s \in \M$ these cut leaves give cut points in $K$, and if $s\in \M$
and satisfies the conclusion of Conjecture~\ref{conjecture:A} these are 
all the cut points of $K$.

Throughout the remainder of this section we 
fix $\theta,\lambda$, the intervals $I_F,I_G,J_F,J_G$ and the (partially defined) maps
$E,F,G$. We assume that $\theta,\lambda \in \P_\infty$, i.e.\/ that $J_F \subset I_F$
and $J_G \subset I_G$ otherwise there are no dynamical cut leaves at all.

The technical tool that keeps tracks of compact families of leaves is the notion
of a {\em thick leaf}:
\begin{definition}[Thick leaf]\label{definition:thick_leaf}
Let $C,D$ be disjoint subsets of $S^1$, each a closed interval or a single point.
The \emph{thick leaf}
$(C,D)$ is the set of all leaves (i.e.\/ unordered distinct pairs of points) in $S^1$
with one element in $C$ and one in $D$.
\end{definition}
We say that a leaf $\mu:=\lbrace x,y\rbrace$ is {\em contained} in the thick leaf
$(C,D)$ if one of $x,y$ is in $C$ and the other is in $D$.

If $C$ and $D$ are both contained in $I_F$ then the thick leaf $F(C,D)$ is defined,
and similarly for $G(C,D)$ if $C$ and $D$ are both contained in $I_G$. Likewise, if
$C$ and $D$ are both contained in $J_F$, or both contained in $J_G$, then 
$E(C,D)$ is defined.

Two thick leaves are particularly impotant. The leaf $F\ell$ has endpoints $\lbrace
F\pi/2,F3\pi/2\rbrace \subset J_F$, and we define 
$C_F:=[\pi/2,F\pi/2]$ and $D_F:=[F3\pi/2,3\pi/2]$. Likewise define 
$C_G:=\iota C_F$ and $D_G:=\iota D_F$.

\begin{lemma}[Image lands in thick]\label{lemma:eventually_in_C_f_D_f}
Let $\mu:=\lbrace x,y\rbrace$ be a cut leaf with trajectory $T = T(x) = T(y)$. 
Then there is some finite $n$ so that the image of $\mu$ under the 
prefix $T|n$ lies in $(C_F,D_F)$ or $(C_G,D_G)$.
\end{lemma}
\begin{proof}
Let $\ell'$ be the leaf of $S^1$ with endpoints on the unique fixed points
of $F^{-1}$ and $G^{-1}$. Note that $F^{-1}$ does not typically fix the other
endpoint of $\ell'$ and similarly for $G^{-1}$.
Note also that $\ell'$ crosses $\ell$, and is (therefore) not a 
leaf of $\L$. Suppose $\mu$ does not cross $\ell'$. Then there is a unique
interval bounded by the endpoints of $\mu$ that does not contain an endpoint
of $\ell$, and both $F^{-1}$ and $G^{-1}$ expand the length of this
interval by $\lambda$. It follows that there is a finite forward iterate 
$E^n\mu$ for which the endpoints of $\mu$ are on different sides of $\ell'$.
Since a cut leaf and its forward images can never cross $\ell$, we may
assume $E^n\mu$ is to the left of $\ell$ (say).

Then either $E^n\mu$ is contained in $(C_F,D_F)$ or it is hidden from $\ell$
by $F\ell$. Again by uniform expansion of $F^{-1}$ on $J_F$ 
there is some least $m$ so that $E^{n+m}\mu = F^{-n}E^n\mu$ 
is to the left of $\ell$ and to the right of $F\ell$.
Since by hypothesis the endpoints of $E^n\mu$ are on opposite sides
of $\ell'$ the same is true of $E^{n+m}\mu$; thus $E^{n+m}\mu$
has endpoints on either side of $F\ell$ in $J_F$ and
is contained in $(C_F,D_F)$.
\end{proof}

Thick leaves may be {\em sliced}:
\begin{definition}[Slicing big leaves]\label{definition:cutting}
Let $(C,D)$ be a thick leaf and let $\mu$ be any leaf. 
Let $(C,D)\backslash \mu$ be the set of leaves contained in $(C,D)$ that do not
cross $\mu$. 

Then $(C,D)\backslash\mu$ is the union of zero, one, or two thick leaves, depending
on whether $\mu$ has endpoints in $(C,D)$, and whether it links leaves in $(C,D)$ or not.
We call this collection of thick leaves the result of {\em slicing} $(C,D)$ by $\mu$.
\end{definition}

There are several combinatorial possibilities for $(C,D)\backslash \mu$:
\begin{enumerate}
\item{if $\mu$ crosses $(C,D)$ but does not have an endpoint in $C$ or $D$ then
$(C,D)\backslash \mu$ is empty;}
\item{if $\mu$ has an endpoint $x$ in $C$ but not $D$ (say)
then $(C,D)\backslash \mu$ consists of a single thick leaf of the form $(C',D)$ where
$x$ divides $C$ into two intervals, and one of them is $C'$;}
\item{if $\mu$ has one endpoint in $C$ and one in $D$ then $(C,D)\backslash\mu$ 
consists of two thick leaves, each of the form $(C',D')$ where $C',D'$ are subintervals
of $C,D$ divided by the endpoints of $\mu$; and}
\item{if $\mu$ has two endpoints in $C$ (say) then $(C,D)\backslash \mu$ consists of
two thick leaves, each of the form $(C',D)$ where $C'$ is one of two subintervals
of $C$ obtained by removing a subinterval bounded by $\mu$.}
\end{enumerate}

\begin{definition}[Intersection of thick leaves]\label{definition:thick_intersect}
Two thick leaves {\em intersect} if there is a leaf contained in both of them. 
A thick leaf $(C,D)$ {\em contains} a thick leaf $(C',D')$ if $C' \subset C$ and
$D' \subset D$ (up to permutation). 
\end{definition}
Note that if $D,E$ are disjoint leaves then $(C,D)$ and $(C,E)$ do not intersect
(as thick leaves), even though the set of geodesics in $\D$ consisting of leaves
contained in one and the other certainly will intersect in $\D$.

We are now ready to recursively define the {\em inclusion graph} $\IG$ 
associated to a parameter $\theta,\lambda$. 

The recursion operates on two queues $V,W$ (i.e.\/ finite data buffers for
which items are added and removed on a first in, first out (FIFO) basis)
whose elements are pairs $((C,D),n)$ where $(C,D)$ is a thick leaf, 
and $n$ is a non-negative integer. The stack $V$ is initialized to 
consist of the pair 
$$V := ((C_F,D_F),0) , ((C_G,D_G),0)$$
and $W$ is initialized to the empty queue
(actually, although items are added to $W$ they are never taken off).

One pass of the recursion takes an element off the end of the queue $V$ 
and either puts a new element on $W$, or adds new elements to the start
of $V$ as follows.

First, pop an element $((C,D),n)$ off the end of $V$. The case $n=0$ is
handled specially; for now let's assume $n>0$. 
We claim that the following holds:
\begin{enumerate}
\item{$E^n$ is defined on the thick leaf $(C,D)$; and}
\item{either $E^n(C,D)$ is contained in $(C_F,D_F)$ or $(C_G,D_G)$, or it
is contained on one side of $\ell$ and no leaf in $E^n(C,D)$ crosses
$F\ell$ or $G\ell$.}
\end{enumerate}
This claim will follow by induction.

Assuming this claim, here is how the recursion proceeds.
If $E^n(C,D)$ is contained in $(C_F,D_F)$ or $(C_G,D_G)$ then push 
$((C,D),n)$ on $W$. Otherwise, $E^n(C,D)$ is on one side of $\ell$ and
no leaf in $E^n(C,D)$ crosses $F\ell$ or $G\ell$ (note that it is
possible that every leaf in $E^n(C,D)$ has both endpoints in
$C_G$ for example). It follows that $E^{n+1}(C,D)$ is
defined and lies entirely on one side of $\ell$. Slice $E^{n+1}(C,D)$ by
$F\ell$ or $G\ell$ depending which side of $\ell$ it is on.
This produces between $0$ and $2$ new thick leaves; their preimages
under $E^{n+1}:(C,D) \to E^{n+1}(C,D)$ are thick leaves contained 
in $(C,D)$. For each of these $(C',D')$ push $((C',D'),n+1)$ to the 
start of the queue $V$.

Notice by construction that $E^{n+1}(C',D')$ is indeed defined, and
by construction it is either contained in $(C_F,D_F)$ or $(C_G,D_G)$ or
it is contained on one side of $\ell$ and no leaf in $E^{n+1}(C',D')$
crosses $F\ell$ or $G\ell$. This proves the induction step of our claim. 

It remains to discuss the special case $n=0$ which applies only to the 
initial elements on $V$. For simplicity we explain the case
of $((C_F,D_F),0)$; the other case is analogous. 
In this case we form $E(C_F,D_F) = F^{-1}(C_F,D_F)$, 
slice the result by $G\ell$ to obtain new thick leaves, take their images 
$(C',D')$ under $F$ (i.e.\/ $E^{-1}$), 
and push $((C',D'),1)$ back on $V$. Evidently these
new elements (if any) satisfy the base case of the induction, so the claim
is proved and the algorithm is defined. 

When the stack $V$ is empty, the list $W$ consists entirely of elements
of the form $((C,D),n)$ where $(C,D)$ and $E^n(C,D)$ are both
thick subleaves of $(C_F,D_F)$ or $(C_G,D_G)$. We do {\em not} assume
$W$ is finite; nevertheless, because $V$ is a FIFO stack, it makes sense
to run the recursion for possibly countably infinitely many steps until 
$V$ is empty.

Let us now define the {\em inclusion graph}.

\begin{definition}[Inclusion graph]\label{definition:IG}
The inclusion graph $\IG$ has one vertex for each element of $W$
and a directed edge from $((C,D),n)$ to $((C',D'),n')$ if $(C',D')$ is
contained in $E^n(C,D)$.
\end{definition}

\begin{example}[Inclusion graphs]\label{example:IG}
Figure~\ref{inclusion_graph_examples} gives examples of inclusion graphs arising
from various points on $\partial \M$. Note that directed loops in $\IG$ correspond
to cut points with periodic orbits under $H$.
\begin{figure}[htpb]
\centering
\includegraphics[scale=0.15]{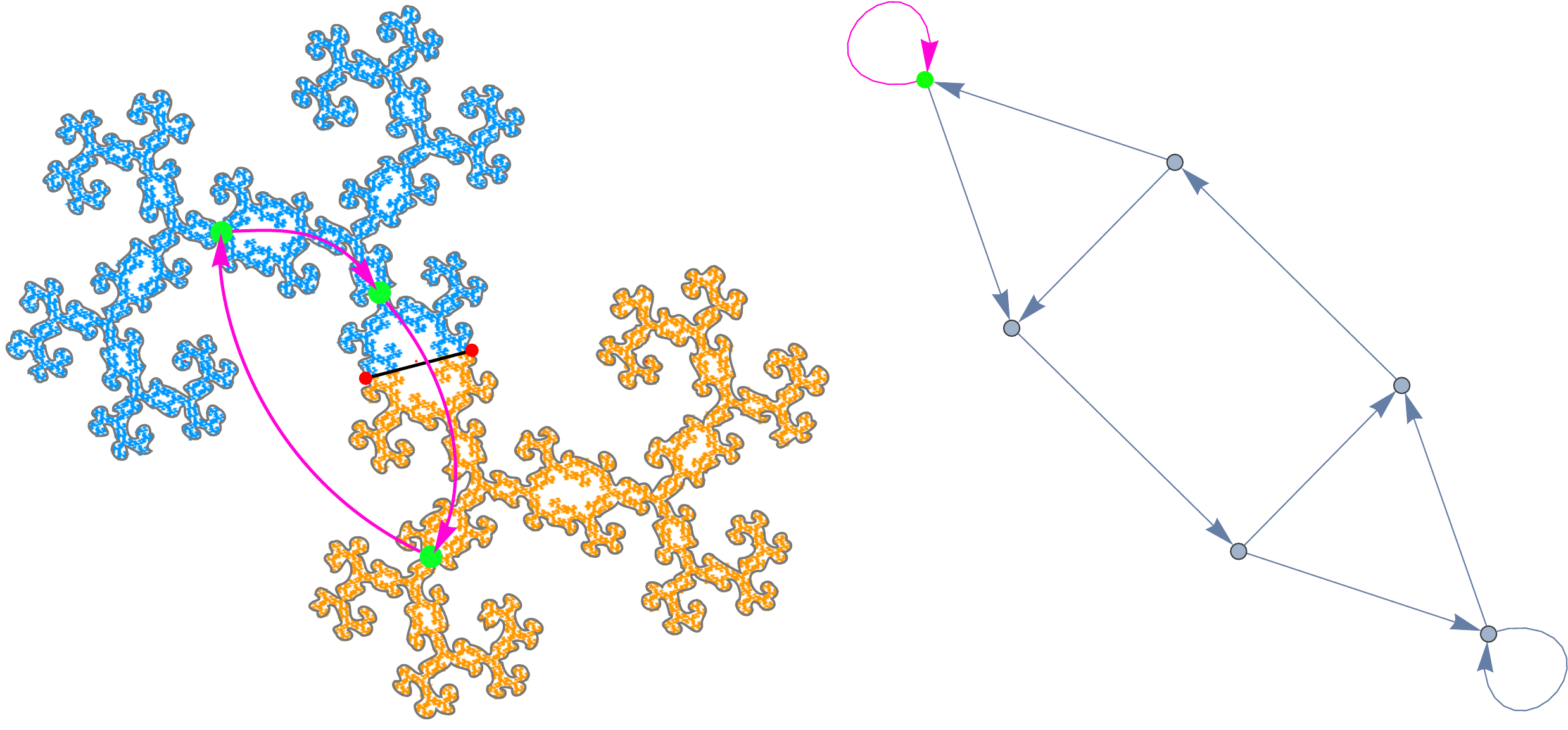}
\includegraphics[scale=0.15]{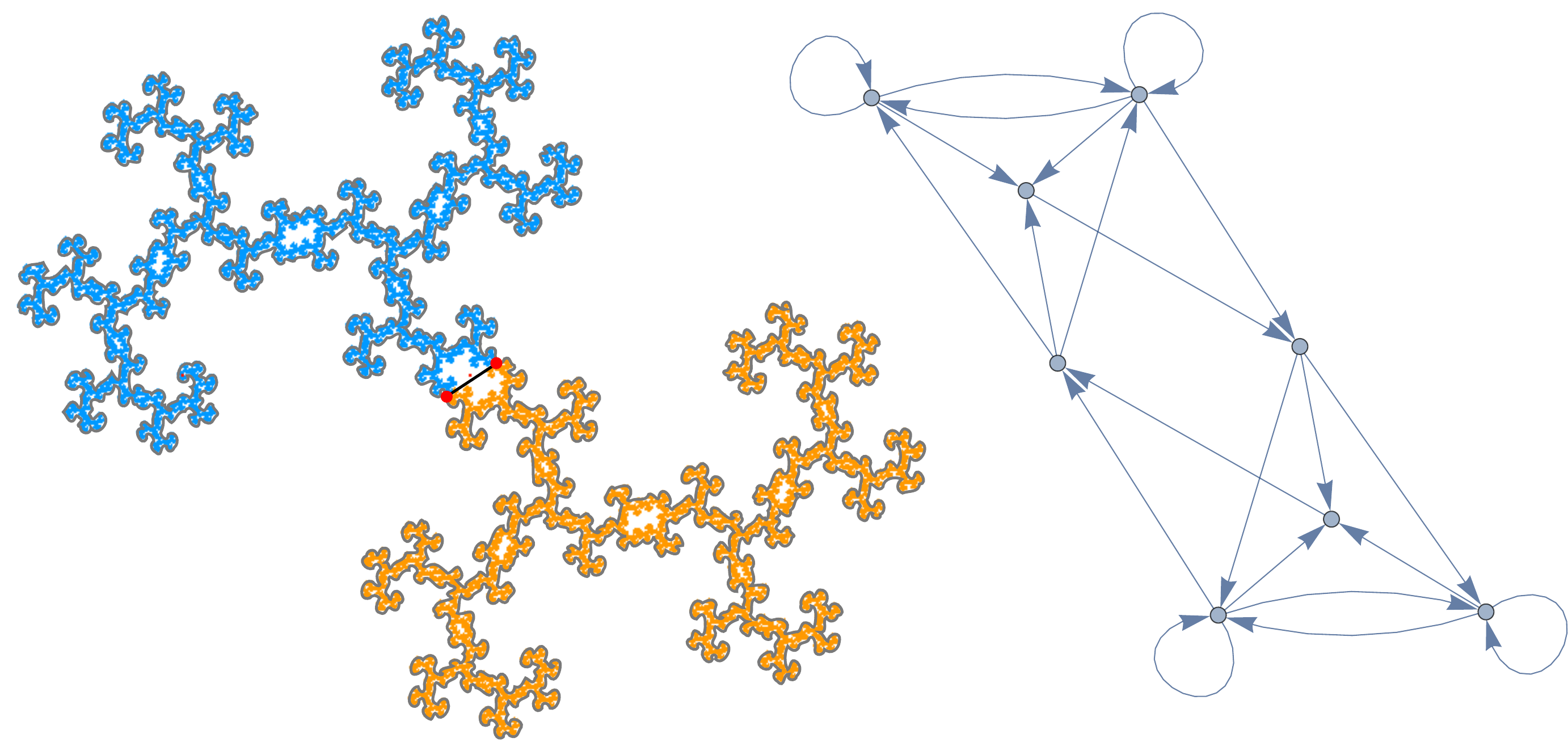}

\includegraphics[scale=0.15]{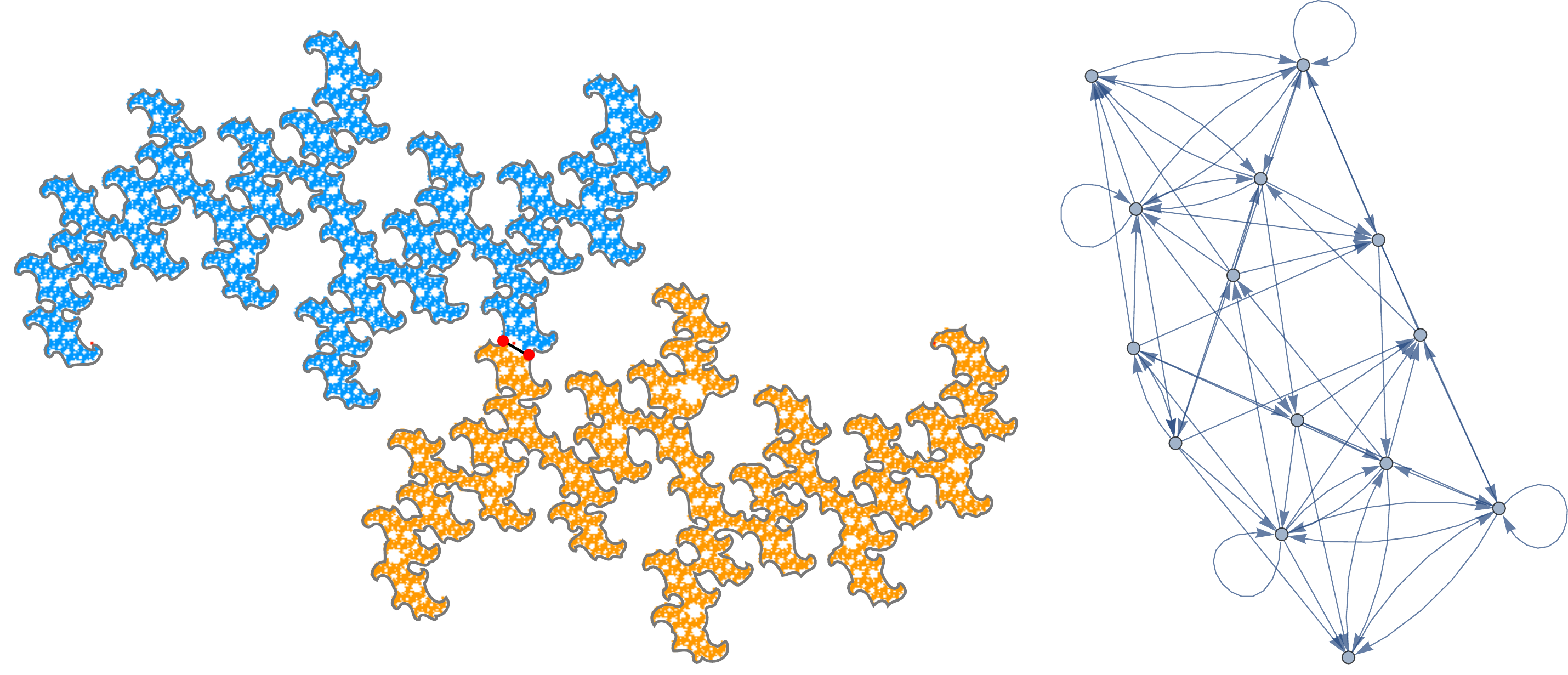}
\includegraphics[scale=0.15]{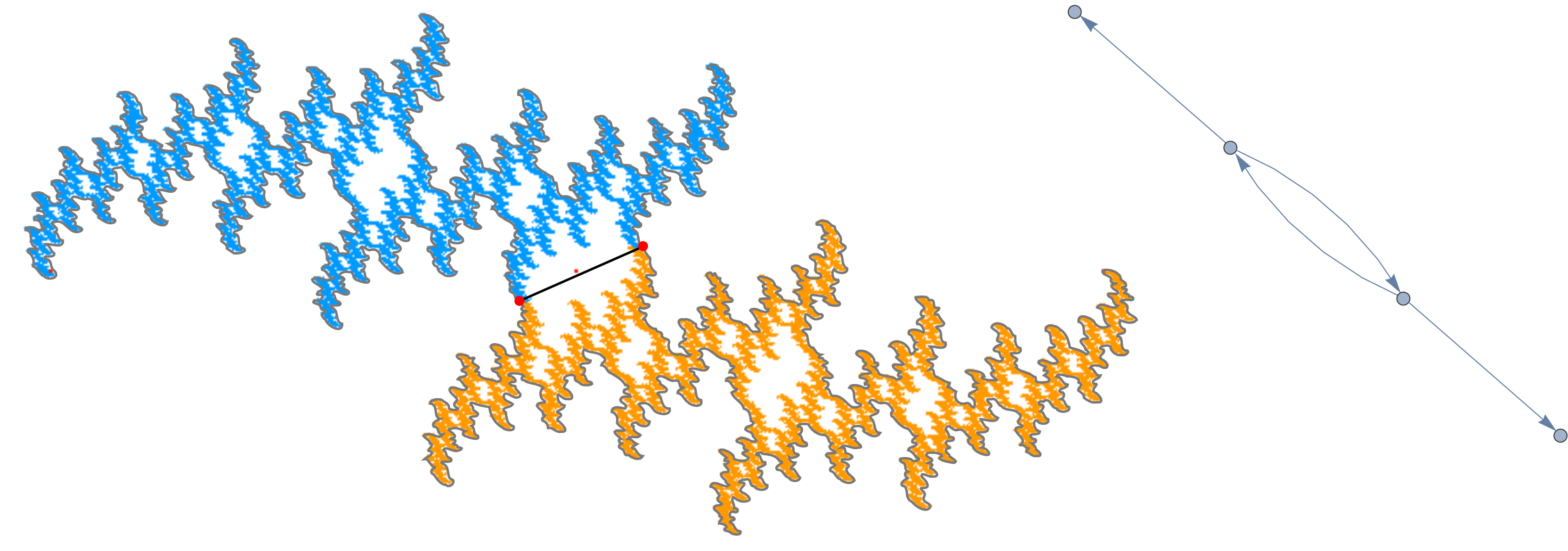}
\caption{Inclusion graph examples. Directed loops in $\IG$ give rise to cut points
with periodic orbits under $H$.}
\label{inclusion_graph_examples}
\end{figure}
\end{example}

\begin{theorem}[Inclusion graph]\label{theorem:inclusion_graph}
Every infinite path in $\IG$ determines a unique dynamical cut leaf, and
every dynamical cut leaf is the image of one of these under some finite
word $w$.
\end{theorem}
\begin{proof}
Suppose we have an infinite directed path $((C_i,D_i),n_i)$ for which
$(C_{i+1},D_{i+1})$ is contained in $E^{n_i}(C_i,D_i)$ for all 
$i=0,1,2,\cdots$. By abuse of notation we form the nested sequence of thick leaves
$$\cdots \subset E^{-n_2-n_1-n_0}(C_3,D_3) \subset E^{-n_1-n_0}(C_2,D_2)\subset E^{-n_0}(C_1,D_1) \subset (C_0,D_0)$$
and let $Z$ be the intersection of the family. By construction this is nonempty,
and every leaf in $Z$ is a dynamical cut leaf. Since every element of $Z$ has
the same forward trajectory, $Z$ consists of a single leaf by 
uniform expansion of $F^{-1}$ and $G^{-1}$ on their domains of definition,
analogous to the conclusion of Lemma~\ref{lemma:distance_estimate}.

Conversely, let $\mu$ be a dynamical cut leaf with trajectory $T$. 
By Lemma~\ref{lemma:eventually_in_C_f_D_f} some finite forward image
$E^n\mu$ is contained in $(C_F,D_F)$ or $(C_G,D_G)$, and repeated forward
images must return to $(C_F,D_F)$ or $(C_G,D_G)$ infinitely often. The
successive returns determine an infinite directed path in $\IG$ with associated
dynamical cut leaf $Z$, and $\mu = wZ$ for some finite $w$ as claimed.
\end{proof}

\begin{question}
Is there a pair $\lambda,\theta$ for which $\IG$ is infinite? Is there such a
pair arising from some $s\in \partial \M$?
\end{question}

\section{Computation}\label{section:computing_parameters}

\subsection{Relationship of $s$ to $\lambda$ and $\theta$}\label{subsection:computing_parameters}

Suppose $s \in \partial \M$ and for the sake of argument, let's suppose further
that $s$ satisfies Conjecture~\ref{conjecture:A}. The number $s$ determines
$\lambda$ and $\theta$, albeit very indirectly.

We have already seen (Lemma~\ref{lemma:growth_estimate}) that $\lambda \ge |s|^{-1}$. 
In this section we will derive an inequality relating $\arg(s)$ to $\lambda$ and
$\theta$.

Since $E$ commutes with $\iota$ it is convenient to work on the quotient $S^1/\iota$.
Let's make a linear change of coordinates so that this quotient circle is $\R/\Z$.
The dynamics of $E$ on this quotient circle may be given in suitable coordinates by 
$$T: x \to \lambda \lbrace x \rbrace + \phi \pmod \Z$$
where $\lbrace x \rbrace$ denotes the fractional part of a real number, 
and where $\phi = (1-\lambda)/2 + \theta/\pi$. The advantage of
expressing things in these coordinates is that $T$ is manifestly the composition of
a $\beta$-transformation of $S^1$ with a rigid rotation; such a map is known as
a {\em generalized $\beta$-transformation} or sometimes as a {\em linear mod one
transformation}, see. e.g \/ \cite{Bruin}.

It turns out that for certain points $p$ in $\partial K$ the dynamics of $H$ on $p$
becomes rather simple:

\begin{proposition}[Hull points rotation]\label{proposition:convex_rotation}
Let $p \in \partial K$ be contained in the boundary of the convex hull of $K$. 
Let $u x = p$. Then by abuse of notation, the action of $T$ on the forward
$T$-orbit of $x$ is topologically semiconjugate to a rotation by $-\arg(s)/\pi$.
\end{proposition}
\begin{proof}
We shall prove this for a specific point $p$ in the boundary of the convex hull of
$K$, but the proof in general is the same. 

\begin{definition}
For $s\in \partial \M$ a {\em leftmost} point $p\in K$ is any point whose real
coordinate is the most negative.
\end{definition}

Evidently one choice of $p$ is given by $p = \sum_{j\ge 0} \epsilon_j s^j$ where
$\epsilon_j = -1$ if the real part of $s^j$ is non-negative, and $\epsilon_j = 1$ otherwise. 
This choice of $p$ is unique unless some finite power $s^n$ is totally imaginary.
If we define 
$$k_n: = \begin{cases} f \text{ if } \epsilon_n = -1 \\
g \text{ if } \epsilon_n = 1
\end{cases}$$ then we may think of $p$ as the image of $0$ under the right-infinite 
product $p = k_0 k_1 k_2 \cdots 0$; i.e.\/ if we define $w_n: = k_0 k_1 \cdots k_{n-1}$ 
(so that $|w_n|=n$) then $p = \lim_{n \to \infty} w_n 0$.

Any leftmost point is contained in the boundary of the convex hull of $K$; in fact,
a supporting halfspace for $p$ is the half space $A$ bounded by the vertical line with
the same real coordinate as $p$. Let $R \subset A$ be the straight ray starting at $p$ and 
parallel to the negative real axis (so $R$ is perpendicular to $\partial A$ at $p$). 
The corresponding prime end determines $x\in S^1$ with $u(x) = p$.

For any $n$ we may determine $p_n:= u H^n(x)$ by $p_n = \sum_{j\ge n} \epsilon_j s^{j-n}$.
Note that $p_n = w_n^{-1} p$ for all $n$. Furthermore, 
the point $p_n$ is also contained in the boundary of the convex hull of $K$, and
a supporting halfspace for $p_n$ is the half space $A_n:= w_n^{-1} A$, since if
there were a point $q\in \Lambda$ in the interior of $A_n$ then $w_n q$ would lie in the
interior of $A$. In particular,
the rays $R_n: = w_n^{-1}R$ are all disjoint, and therefore their circular order at
infinity (which agrees with the circular order on the orbit $H^n(x)$ in $S^1$) is the
same as the circular order on the set of angles $n \arg{s^{-1}}$. 

This proves the theorem in the case of leftmost $p$; the general case is perfectly
analogous.
\end{proof}
Notice that if $\arg(s)/\pi$ is rational the specific leftmost $p$
we construct in Proposition~\ref{proposition:convex_rotation} is periodic.

On the other hand, for a generalized $\beta$-transformation $T$ one has the following
proposition, whose statement and proof we learned from Toby Hall:

\begin{proposition}[Rotation interval]\label{proposition:rotation_interval}
For $\beta \in [1,2]$ and $\phi \in [0,1]$
let $T: x \to \beta \lbrace x \rbrace + \phi \pmod \Z$ be a generalized 
$\beta$-transformation, and let $\tilde{T}: x \to [x] + \lambda \lbrace x \rbrace + \phi$
be a real-valued lift, where $[x]$ denotes the floor (i.e.\/ the integer part) of $x$. 
Define continuous monotone nondecreasing maps $\tilde{T}_0,\tilde{T}_1: \R \to \R$ by
$$\tilde{T}_0(x):= \begin{cases}
\tilde{T}(x) \text{ for } 0 \le {x} \le 1/\beta \\
[x] + 1 + \theta \text{ for } 1/\beta \le {x} < 1
\end{cases}$$
and
$$\tilde{T}_1(x):= \begin{cases}
[x] + \beta-1 + \theta  \text{ for } 0 \le {x} \le (\beta-1)/\beta \\
\tilde{T}(x) \text{ for } (\beta-1)/\beta \le {x} < 1
\end{cases}$$
Let the rotation numbers of $\tilde{T}_0$ and $\tilde{T}_1$ be $r_0$ and $r_1$
respectively. Then for every $r \in [r_0,r_1] \pmod \Z$ there is an $x \in S^1$ so that
the action of $T$ on the forward $T$-orbit of $x$ is semi-conjugate to rotation by $r$,
and conversely.
\end{proposition}
\begin{proof}
If $T$ is semi-conjugate to a rotation on the orbit of $x$ we say that $x$ has a
`rotation-compatible orbit'.

Observe that the maps $\tilde{T}_0$ and $\tilde{T}_1$ 
are contained in a 1-parameter family of maps
$\tilde{T}_t:\R \to \R$ for $t\in [0,1]$ that all commute with integer translation, 
and that agree with $\tilde{T}$ outside a family of `plateaux' where the maps are locally
constant; different $\tilde{T}_t$ are obtained by moving these plateaux up and down. 

Evidently $\tilde{T}_0(x) \le \tilde{T}(x) \le \tilde{T}_1(x)$ for every
$x$, so if there is a point $x \in S^1$ with a rotation-compatible orbit, after
lifting $x\to \R$ the rotation number of $\tilde{T}$ on $x$ is in the interval
$[r_0,r_1]$.

Conversely, for every $r\in [r_0,r_1]$ there is some $t$ for which
$\tilde{T}_t$ has rotation number $r$, and on any orbit for $\tilde{T}_t$ that is disjoint
from the plateaux, the maps $\tilde{T}_t$ and $\tilde{T}$ agree. When $r$ is
irrational one of the endpoints of the plateau can never have a forward orbit
that lands in the plateau (or else the rotation number would be rational); when 
$r$ is rational, there must be an unstable periodic orbit for index reasons 
which can't intersect the plateau interior. 
\end{proof}

\begin{corollary}\label{corollary:rotation_interval}
With notation as above, $-\arg(s)/\pi$ is contained in the interval $[r_0,r_1]$
associated to the generalized $\beta$-transformation with parameters $\lambda$
and $\phi$ where $\phi = (1-\lambda)/2 + \theta/\pi$.
\end{corollary}

Unfortunately, Corollary~\ref{corollary:rotation_interval} does not constrain
$\theta$ very much given $\lambda$ and $-\arg(s)/\pi$. Figure~\ref{rotation_number_intervals} 
shows the values of $r_0$ (in blue) and $r_1$ (in red) 
for $\phi \in [0,1]$ and for $\lambda = \sqrt{2}$.

\begin{figure}[htpb]
\labellist
\small\hair 2pt
\pinlabel $\phi$ at 250 5
\pinlabel $r$ at 5 340
\endlabellist
\centering
\includegraphics[scale=0.4]{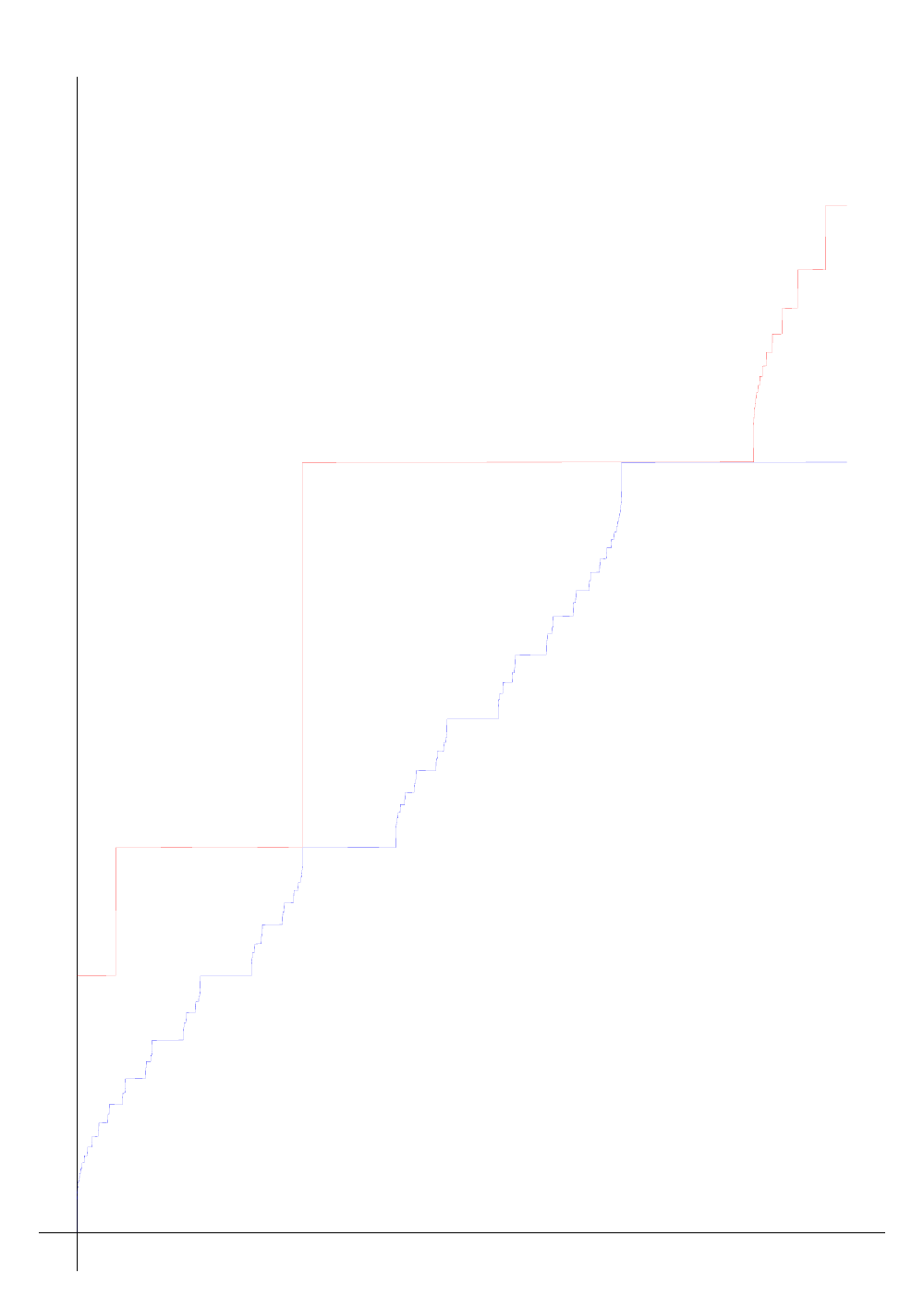}
\caption{The graph of $r_0$ (in blue) and $r_1$ (in red) for $\phi \in [0,1]$ and
$\lambda = \sqrt{2}$. There are many intervals of values of $\phi$ for which both
$r_0$ and $r_1$ are constant.}
\label{rotation_number_intervals}
\end{figure}

\subsection{Computing $\lambda$ and $\theta$}\label{subsection:computing_lambda_theta}

To actually compute $\lambda$ and $\theta$ we need to get our hands dirty. 
For $s\in \partial \M$ we fix a depth $N$, define $\Lambda_N$ to be the
set of points $w$ of the form $w = \sum_{j=0}^{N-1} \epsilon_j s^j$ where each
$\epsilon_j = \pm 1$, and define a potential function
$$F(z): = \sum_{w\in \Lambda_N} |z - w|^{-2}$$
We may then compute a connected contour of $F(z)=C$ close to $\partial K$ and derive a mesh 
$P:=P(C,N)$ of points together with a circular order on it. In practice we used $N=16$.
It is important to choose $C$ big enough 
that the contour comes very close to 
$\partial K$, but not so big that it is sensitive to the `mesh size' of
the approximation $\Lambda_N \sim \Lambda$. The constant $C$ was chosen dynamically as
a function of $s$ to be the largest value to meet some obvious numerical criteria. 

For each point $p\in P$ we compute whether it is closer to $f\Lambda$ or $g\Lambda$
and thereby define $H' p$ to be $f^{-1}p$ or $g^{-1}p$ accordingly. Then we take the
result, and find a nearby element of $P$ by first moving $H' p$ along the gradient of
the potential to the correct contour, then moving perpendicularly along the gradient until
we get sufficiently close to a point in $P$. The result is $H p$, and we obtain in this
way a map $H:P \to P$. We expect (and experiments bear this out) that $H'$ is 
monotone increasing over $P$ in the circular order 
except at two discontinuities (where $H' p$ switches from
$f^{-1}p$ to $g^{-1}p$ or vice versa). 

Having obtained a discrete combinatorial approximation $H:P \to P$ to the `true' map
$H:S^1 \to S^1$ we may compute the average slope of $H$
over these monotone intervals as an approximation to $\lambda$, 
and for each of the midpoints $m$ of these monotone intervals, determine $H(m)-m$ as an
approximation to $\theta$ (in units for which the circularly ordered set has total
length $2\pi$). The reliability of these estimates depends on how close our map $H$ is to
being linear on its monotone intervals (on a scale large compared to the mesh size);
in our experiments, the slope of $H$ seemed to deviate only to order $1\%$, so we took our
estimates to be reliable, although admittedly we have no theoretical justification for this,
or an estimate of the error. 

A more principled method would be to estimate the $H$-invariant measure on $P$
by the lap-counting formula of Flatto--Lagarias \cite{Flatto_Lagarias} and then use this
measure to compute $\lambda$ and $\theta$.

\begin{figure}[htpb]
\centering
\includegraphics[scale=0.2]{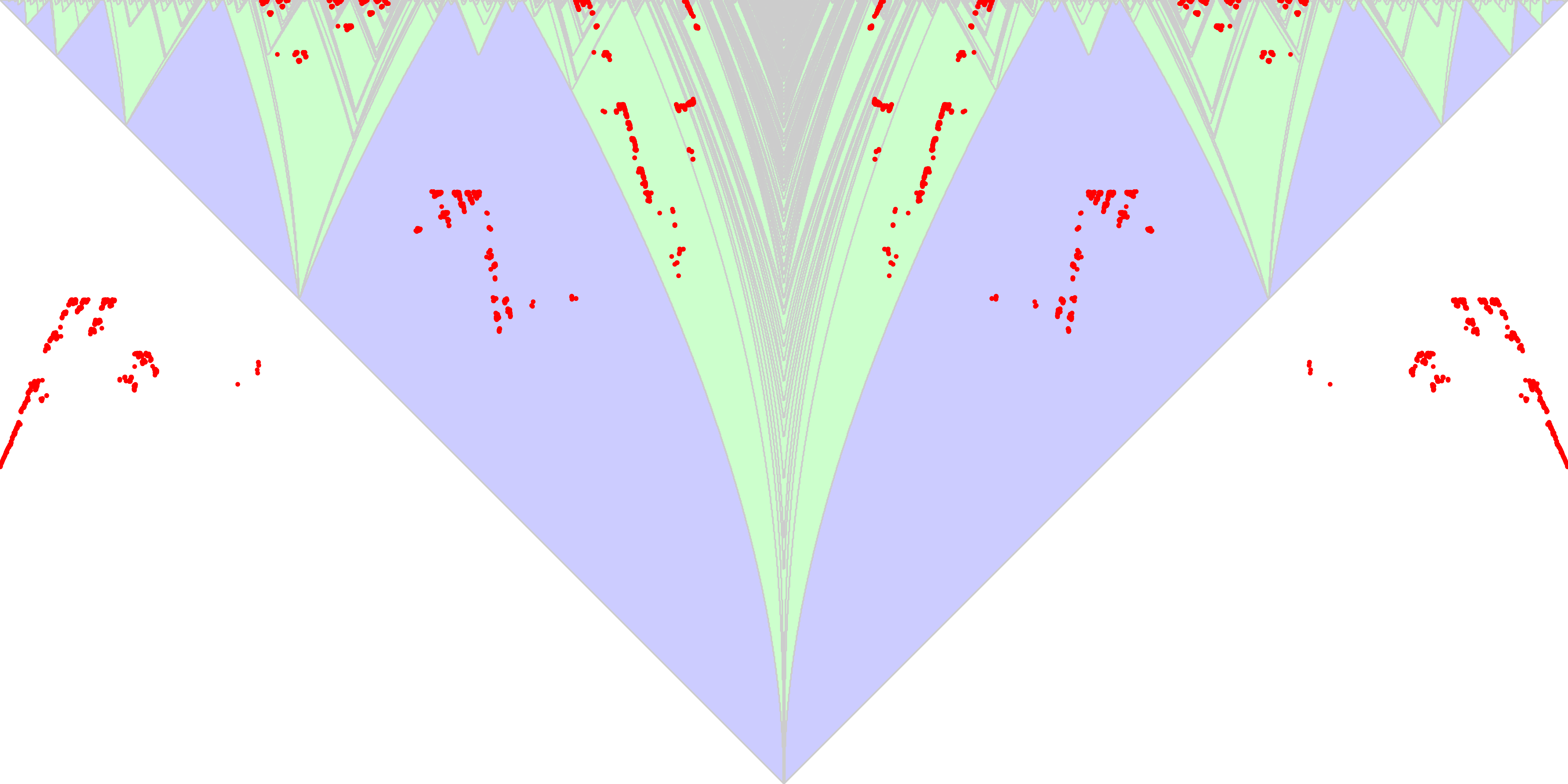}
\caption{Values of $\lambda,\theta$ for $s\in \partial \M$.}
\label{triangles_and_values}
\end{figure}

The results are depicted in Figure~\ref{triangles_and_values} for half a fundamental
domain. The $\lambda,\theta$ parameters (in red) are superimposed over 
Figure~\ref{triangle_list}. The only parameters for which there are dynamical cut
leaves are contained in $\Delta_2$ and $\Delta_3$ (up to symmetry), consistent with
Conjecture~\ref{conjecture:cut_is_limit}.

The figure reveals a great deal of apparent structure; for example, zooming in near
$\lambda =2$, $\theta = 0.5$ one sees a set that resembles the attractor of a linear IFS,
suggesting the existence of some sort of dynamics in the parameter plane. These particular
points are all contained in $\Delta_3$, and the ones with $\lambda=2$ correspond to 
external rays landing on the $1/3$ bulb of $\Mord$. For comparison, let $\CC$ be the
Cantor set in $[0,1]$ which is the attractor of the two generator linear IFS
$$\phi_0: t \to (t-1/6)/4 + 1/6 \quad \phi_1: t \to (t-1/4)/2 + 1/4$$
The image of $\CC$ (offset slightly from the line $\lambda = 2$) under the change
of coordinates $\theta = \pi(t-1/2)$ is imposed on Figure~\ref{zoomed_vals_with_Cantor}.

\begin{figure}[htpb]
\centering
\includegraphics[scale=0.3]{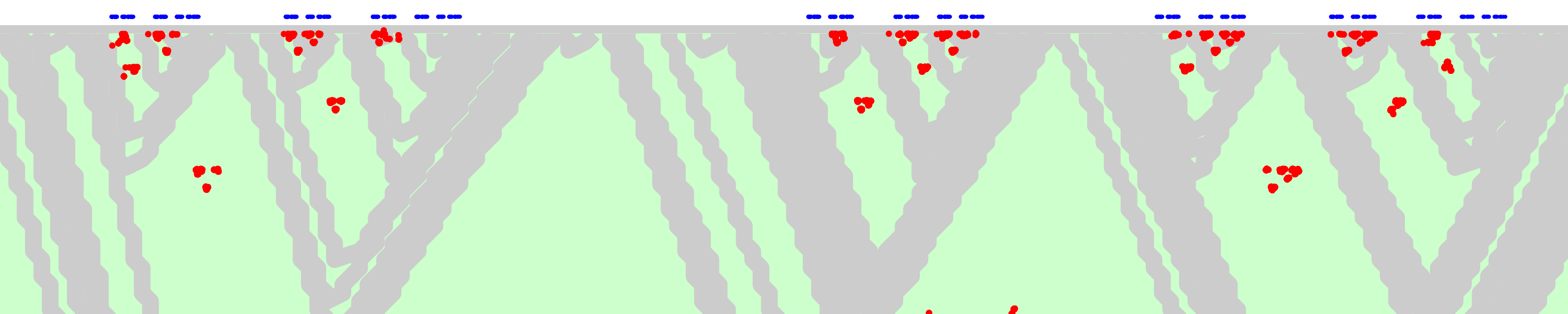}
\caption{Zooming in Figure~\ref{triangles_and_values} near $\lambda=2$, $\theta = 0.5$ for $s\in \partial \M$,
together with the (slightly offset and linearly rescaled) Cantor set $\CC$ (in blue).}
\label{zoomed_vals_with_Cantor}
\end{figure}

The numerical evidence makes it remotely plausible, though not compelling, that the set of $\theta$ 
values with $\lambda=2$ in the $1/3$ bulb maps exactly to $\CC$ under $t = \theta/\pi + 1/2$.
Computational shortcomings of our method makes it easier to compute $\theta$ values in some
ranges than others, and the lacuna near $t=1/4$ seems to correspond to 
values of $s \sim 0.49116 + 0.45727i$ for which accuracy is very difficult to achieve.

\subsection{Conjecture A}\label{subsection:Conjecture_A}

Recall that Conjecture~\ref{conjecture:A} asserts that $D_f = J_f$ for every
non-real $s \in \partial \M$; equivalently, that $\partial K$ is the union of a
single arc contained in $f\partial K$ and a single arc contained in $g\partial K$,
meeting in exactly two points, or exactly one cut point $0 = fK \cap gK$. In this
section we briefly present the numerical evidence for the conjecture.

The evidence is organized around a finite-depth version of this two-point meeting
condition. Recall from Definition~\ref{definition:limit_set} the decreasing family
of disk covers $D_0 \supset D_1 \supset \cdots$ with $\bigcap_n D_n = \Lambda$,
where $D_0$ is the disk of radius $1/(1-|s|)$ centered at $0$ and
$D_n = fD_{n-1} \cup gD_{n-1}$.

\begin{definition}[Clean at depth $n$]\label{definition:clean}
For $s \in \M$ and $n \ge 1$, let $\partial D_n$ denote the outer boundary of
$D_n$, and set
$$X_n := fD_{n-1} \cap gD_{n-1} \cap \partial D_n,$$
the set of points of $\partial D_n$ lying in both halves. We say that $s$ is
{\em clean at depth $n$} if $|X_n| = 2$.
\end{definition}

\begin{remark}
The involution $z \mapsto -z$ preserves $D_n$ and interchanges its two halves, so
$X_n$ is invariant under $z \mapsto -z$. Hence $|X_n|$ is odd only if $0 \in X_n$;
in particular $|X_n| = 1$ if and only if $X_n = \lbrace 0 \rbrace$. However,
since each $D_{n+1}$ is strictly contained in the interior of $D_n$, if $|X_n|=1$
then $D_{n+1}$ is disconnected so that $s$ is not in $\M$ after all.

The generic clean case is $|X_n| = 2$, an antipodal pair
$\lbrace p, -p \rbrace$ cutting $\partial D_n$ into a single arc in $fD_{n-1}$ and a
single arc in $gD_{n-1}$; this is the finite-depth form of the partition
$D_f = J_f$. The disk-cover computation of \S~\ref{subsection:computing_lambda_theta}
returns $|X_n|$ as the number of discontinuities of its piecewise-linear
reconstruction of the map $H$.
\end{remark}

The results of computation are summarized as follows. We let $s$ range over a sample of 
$19663$ numerically computed parameters of $\partial \M$. Computed at depth $n = 16$, 
all but $142$ of these --- that is, $19521$ of them, or $99.28\%$ --- are clean; 
the $142$ exceptions return $|X_{16}| > 2$. In a focused resample of $1426$ boundary 
parameters (near the $n=16$ exceptions) computed at depth $n = 23$, every one is clean.

\appendix

\section{Proof of Theorem~\ref{theorem:strict_inequality}}\label{section:inequality_proof}

In this section we prove Theorem~\ref{theorem:strict_inequality}, that the entire
circle $|s|=1/\sqrt{2}$ lies in the interior of $\M$ except for the two points
$s=\pm i/\sqrt{2}$. By the symmetries $s\mapsto\bar{s}$ and $s\mapsto-s$ of $\M$
(\S\ref{subsection:symmetries}) it is enough to certify $s$ in the quadrant
$\arg(s)\in[0,\pi/2)$. 

Most of the interval is certified numerically by the method of traps of \cite{Calegari_Koch_Walker},
although there are four special points where different methods are necessary. These points
require special treatment, and fall into three cases:
\begin{enumerate}
\item{At $s=1/\sqrt{2}$ the limit set $\Lambda$ degenerates to a line segment, and there
is no ordinary trap for $s$. Instead, a neighborhood of this point can be certified as interior to $\M$
by an {\em affine trap} c.f.\/ \cite{Calegari_Koch_Walker} \S~10.3.}
\item{At $s_0: = 1/2 + i/2$ and $s_1:=1/4 + \sqrt{-7}/4$ the limit set $\Lambda$ is a
Jordan disk which has the property that the interiors of $f\Lambda$ and $g\Lambda$ are
disjoint and their boundaries meet along a connected arc, so that there is no ordinary 
trap for $s$. However, the special nature of these points means that they are 
{\em renormalization points} (or {\em landmark points} in the terminology of
Solomyak \cite{Solomyak}) and neighborhoods of these points can be certified as interior to $\M$
by {\em limit traps} c.f.\/ \cite{Calegari_Koch_Walker} \S~9.}
\item{At $s=i/\sqrt{2}$ the limit set $\Lambda$ is a rectangle, and in fact $s$ is not 
an interior point of $\M$ at all. To certify that $s= e^{i\theta}/\sqrt{2}$ is interior 
for an entire open interval $\theta \in (\theta_*,\pi/2)$ we introduce a novel 
technique, the method of {\em logarithmic spirals}.}
\end{enumerate}
Every other point lies in a regime in which ordinary traps may be used to certify
interiority. Each of these cases required computer assistance; this was carried out
by the program {\tt schottky}.

\medskip

Here is the structure of the remainder of this appendix. In \S~\ref{subsection:log_spiral}
we explain the method of logarithmic spirals, and explain how this can be used to
certify an explicit open interval of $|s|=1/\sqrt{2}$ abutting $i/\sqrt{2}$. In 
\S~\ref{subsection:limit_traps} we recall the definition of limit traps and explain how
they can be used to certify explicit open neighborhoods of $s_0$ and $s_1$. 
In \S~\ref{subsection:affine_traps} we recall the definition of affine traps and explain
how they can be used to certify an explicit open neighborhood of $1/\sqrt{2}$.
Finally in \S~\ref{subsection:explicit_certificate} we briefly discuss how ordinary trap
certificates may be found away from the four exceptional points, and conclude the
proof.

\subsection{The logarithmic--spiral construction near $i/\sqrt{2}$}\label{subsection:log_spiral}

In this subsection we show $e^{i\theta}/\sqrt{2}$ is in the interior of $\M$ for
an explicit open interval of values $\theta \in (\theta_*,\pi/2)$.
It is convenient to adopt notation throughout this subsection adapted to this specific
regime. For $t$ small and non-negative, define $s(t) = e^{i(\pi/2 - t)}/\sqrt{2}$, and denote
$f_t:z \to s(t)z - 1$ and $g_t:z \to s(t)z + 1$ with attractor $\Lambda_t$. Likewise,
for $w\in \SS$ we let $w_t$ denote the 
similarity of $\C$ corresponding to $w$ for $s=s(t)$.

When $t=0$ (i.e.\/ for $s=i/\sqrt{2}$) the set $\Lambda_0$ is a rectangle which
for every integer $m$ is tiled by interior disjoint subrectangles $w_0\Lambda_0$ as $w$ ranges over the
elements of $\SS$ with $|w|=m$. 

Now let's take a small positive $t$ and look at the images $w_t\Lambda_0$ for various
words $w_t$. Note that the Hausdorff distance from $\Lambda_0$ to $\Lambda_t$ is
of order $t$ so the Hausdorff distance from $w_t\Lambda_0$ to
$w_t\Lambda_t$ is of order $t\cdot 2^{-|w|/2}$.

When $t=0$ the rectangles $w_0\Lambda_0$ for fixed $|w|$ abut each
other but do not overlap. When $t>0$ the rectangles $w_t\Lambda_0$ for fixed
$|w|$ overlap each other nontrivially; see Figure~\ref{subdivide_rectangle_tilt}. 

\begin{figure}[htpb]
\centering
\includegraphics[scale=0.6]{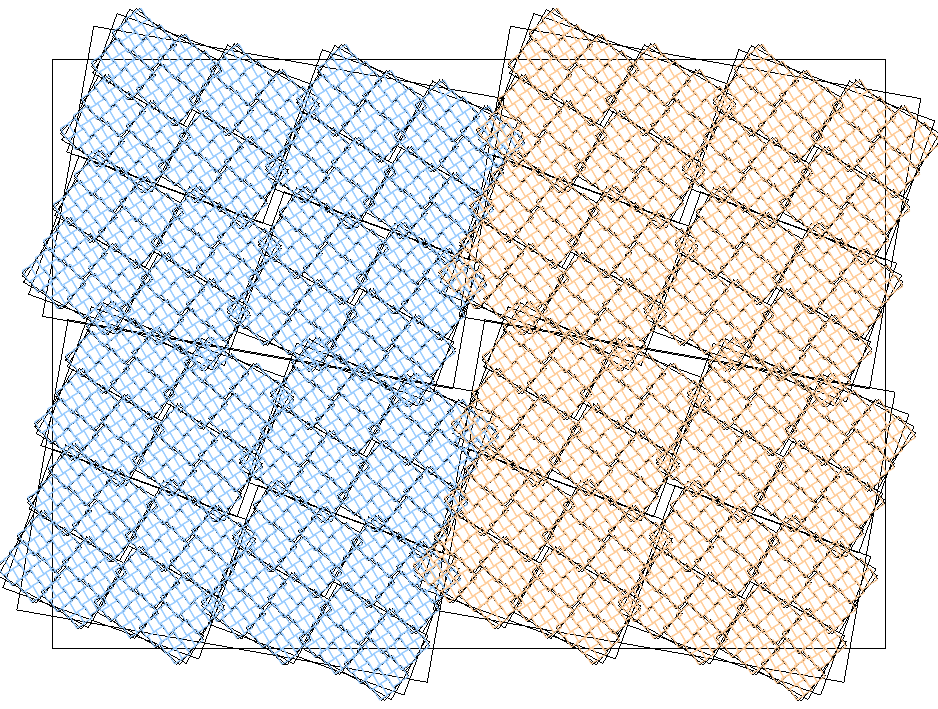}
\caption{The images $w_t\Lambda_0$ when $t$ is small but strictly positive.}
\label{subdivide_rectangle_tilt}
\end{figure}

The set $\Lambda_t$ is the Hausdorff limit 
$$\Lambda_t = \lim_{n \to \infty} \bigcup_{|w|=n} w_t\Lambda_0$$
Thus although none of the sets $\cup_{|w|=n} w_t\Lambda_0$ are full, their
filled sets have the property that their images under $f_t$ and $g_t$ overlap
each other nontrivially. We shall show that this property persists in the limit;
in other words, the filled set $K_t$ 
of $\Lambda_t$ has the property that $f_t K_t$ and $g_t K_t$ 
overlap nontrivially. Furthermore, we shall show that
this property of $K_t$ is stable under perturbation, and therefore
holds for all $s'$ sufficiently close to $s(t)$, so that $s(t)$ 
is in the interior of $\M$.

\begin{figure}[htpb]
\labellist
\small\hair 2pt
\pinlabel $L_t^0$ at 250 180
\pinlabel $f_tL_t^0$ at 650 180
\pinlabel $g_tL_t^0$ at 850 180
\endlabellist
\centering
\includegraphics[scale=0.4]{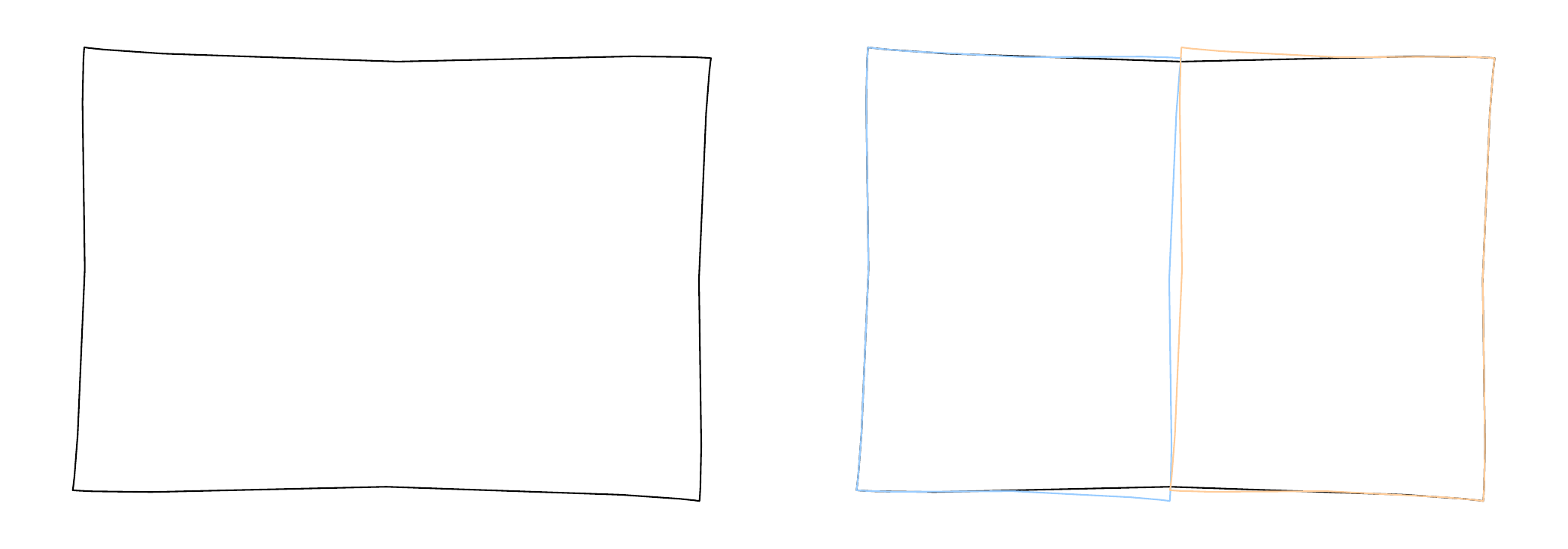}
\caption{The set $L_t^0$ (in black) for $t \sim 0.0087$
whose boundary is made from 8 segments of logarithmic
spirals landing at the fixed points of $(gffg)_t$, $(ggff)_t$, $(fggf)_t$ and $(ffgg)_t$,
and its images under $f_t$ (in blue) and $g_t$ (in orange) whose union has
a full set that completely encloses $L_t^0$.}
\label{subdivide_rectangle_spiral}
\end{figure}

For every $t$ positive and sufficiently small, we shall construct a 
Jordan domain $L_t^0$ so that the union $f_tL_t^0\cup g_tL_t^0$ is connected, 
and its full set $L_t^1$ contains
$L_t^0$. It will follow by induction that every union $\cup_{|w|=n}w_tL_t^0$ is connected,
and that its full set $L_t^n$ contains $L_t^j$ for all $j\le n$. In other words,
the full set $K_t$ is the {\em increasing} union of the sets $L_t^n$.
Furthermore, it will turn out that $g_tL_t^0$ and $f_tL_t^0$ intersect {\em transversely}.

The set $L_t^0$ and its images under $f_t$ and $g_t$ is depicted in
Figure~\ref{subdivide_rectangle_spiral}. The `corners' $p_i$ for $i = 0,1,2,3$
of $L_t^0$, starting at the bottom right and proceeding in the anticlockwise direction,
are the fixed points of the elements $(gffg)_t$, $(ggff)_t$,
$(fggf)_t$ and $(ffgg)_t$ respectively. These four words are the cyclic
rotations of one another, and consequently their fixed points are permuted by
$f_t$ and $g_t$: if $w'$ is obtained from $w$ by moving its last letter to the
front then $\mathrm{fix}(w') = c_t\,\mathrm{fix}(w)$, where $c$ is that letter, and so
$$p_1 = g_t\,p_0, \qquad p_2 = f_t\,p_1, \qquad p_3 = f_t\,p_2, \qquad p_0 = g_t\,p_3.$$
Moreover the symmetry $z\mapsto -z$ (which conjugates $f_t$ to $g_t$) preserves the
configuration, so $p_2 = -p_0$ and $p_3 = -p_1$.

Each `edge' of $L_t^0$ is made from a pair of segments of logarithmic spirals, each of
which ends at one of the two corners bounding that edge and is taken into itself by the
corresponding $w_t$; each such segment is a piece of the invariant logarithmic spiral
of the similarity $w_t$ about its fixed point (of multiplier $s^4$). An edge is thus
determined by its two endpoints together with the common endpoint of its two spiral
segments, which we call the `midpoint' of the edge. Write $m_0,m_1,m_2,m_3$ for the
midpoints of the right, top, left and bottom edges $p_0p_1$, $p_1p_2$, $p_2p_3$,
$p_3p_0$ respectively.

We do \emph{not} choose the four midpoints independently. We choose only the midpoints
$m_1$ and $m_3$ of the top and bottom edges, and we choose them symmetrically,
$$m_3 = -m_1,$$
in accordance with the symmetry $z\mapsto -z$. The midpoints $m_2$ and $m_0$ of the
left and right edges are then \emph{determined} by the requirement that the left edge
be the $f_t$--image of the top edge and the right edge be the $g_t$--image of the
bottom edge:
$$m_2 = f_t\,m_1, \qquad m_0 = g_t\,m_3.$$
Since $f_t p_1 = p_2$
and $f_t p_2 = p_3$, and since $f_t$ carries the invariant spiral of $(ggff)_t$ about
$p_1$ to that of $(fggf)_t$ about $p_2$, and the invariant spiral of $(fggf)_t$ about
$p_2$ to that of $(ffgg)_t$ about $p_3$, the choice $m_2 = f_t m_1$ makes the
entire left edge of $L_t^0$ equal to $f_t(\text{top edge of }L_t^0)$; symmetrically $m_0 = g_t m_3$
makes the entire right edge of $L_t^0$ equal to $g_t(\text{bottom edge of }L_t^0)$.

There remains a single degree of freedom, the choice of the top midpoint
$m_1\in\C$. If we write $\mu = (p_1+p_2)/2$, we can define $m_1$ to be
$\mu$ moved a distance of order $t$ towards the centre of $L_t^0$ (the origin); that is,
we `push' the top and bottom midpoints slightly into the interior. The push is chosen
so that $g_tL_t^0$ and $f_tL_t^0$ overlap transversely near the middle of the bottom
(and, symmetrically, the top) edge, and so that the union $f_tL_t^0\cup g_tL_t^0$ has a
full set containing $L_t^0$.

The orange corner $g_t p_2$ (i.e.\/ the lower left corner of $g_tL_t^0$) lies in the 
interior of the blue `rectangle' $f_tL_t^0$ at a distance $O(t^2)$ 
from the boundary; Figure~\ref{subdivide_rectangle_spiral_zoom} zooms in to 
Figure~\ref{subdivide_rectangle_spiral} near this point. 
These facts are routine to verify.

\begin{figure}[htpb]
\centering
\includegraphics[scale=0.5]{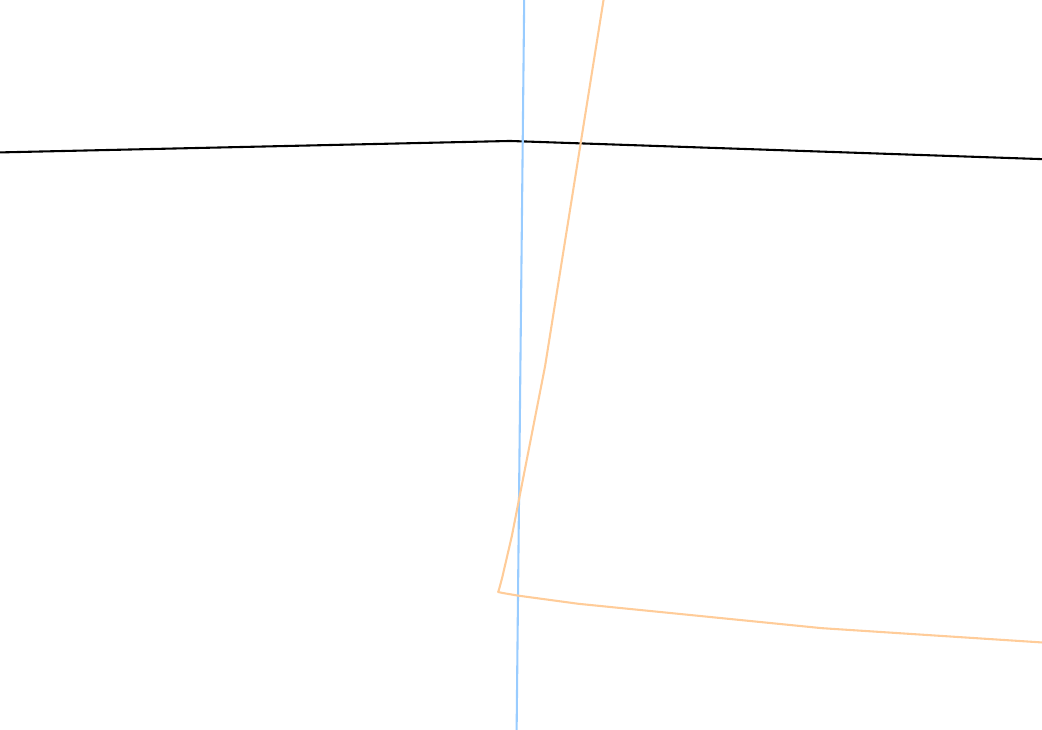}
\caption{A magnification near the middle of the bottom of 
Figure~\ref{subdivide_rectangle_spiral} showing
the intrusion of the (orange) corner $g_t p_2$ into the (blue) `rectangle'
$f_tL_t^0$.}
\label{subdivide_rectangle_spiral_zoom}
\end{figure}

Note that $\partial f_tL_t^0$ overlaps a neighborhood of the left edge
in $\partial L_t^0$, and $\partial g_tL_t^0$ overlaps a neighborhood
of the right edge in $\partial L_t^0$, since $\partial L_t^0$ is made of 
segments of invariant logarithmic spirals.

For $s'$ sufficiently close to $s(t)$ (of order $|s'-s(t)| = O(t^2)$)
we may construct an analogous set $L_{s'}^0$ from eight segments of invariant
logarithmic spirals, which is entirely contained in the full set of 
$f_{s'}L_{s'}^0\cup g_{s'}L_{s'}^0$
and for which $g_{s'}L_{s'}^0$ and $f_{s'}L_{s'}^0$ still intersect 
transversely (since this is an open condition). 
In particular, $\Lambda_{s'}$ is connected, so that $s'\in \M$
and therefore $s(t)$ is in the interior of $\M$ for all sufficiently small $t$.

\medskip

We now record how this construction is verified in practice, and how far towards the
bulk of the circle it reaches. The one free parameter---the top midpoint $m_1$---is a
\emph{complex} number, so we are free to vary it in \emph{two} real dimensions, rather
than only pushing it radially inward as the informal description suggests. Everything
else in $L_s^0$ is then determined by $m_1$: the bottom midpoint by $m_3 = -m_1$, the
side midpoints by $m_2 = f_t m_1$ and $m_0 = g_t m_3$, and the eight edges by the
corresponding invariant spirals. Thus a single point $m_1\in\C$ determines $L_t^0$
together with its images $f_tL_t^0$ and $g_tL_t^0$.

In Figure~\ref{subdivide_rectangle_spiral_zoom} one can see the desired configuration:
there is a triangular `notch' bounded by a blue edge (the right side of $f_tL_t^0$),
an orange edge (the left side of $g_tL_t^0$) and a black edge (the bottom side of $L_t^0$).
The existence of this notch can be confirmed by a numerically stable certificate, and
reduces to two inequalities that are decided by
the positions of finitely many points, and are therefore checkable exactly in interval
arithmetic. Let $Q = g_t p_2$ be the lower-left corner of $g_tL_t^0$; note that $Q$
depends on $t$ but not on $m_1$. If $Q$ lies on the interior side of the right edge of
$f_tL_t^0$, then the bottom edge of $g_tL_t^0$ crosses that right edge at a point $P$;
and if $P$ lies below the bottom edge of $L_t^0$, then the overlap of $f_tL_t^0$ and
$g_tL_t^0$ straddles the bottom edge of $L_t^0$, so that the full set of
$f_tL_t^0\cup g_tL_t^0$ contains it. Together with the symmetric statement at the top
edge, this yields $L_t^1\supseteq L_t^0$ with $f_tL_t^0$ and $g_tL_t^0$ meeting
transversely. Both inequalities are open conditions, so whenever they hold with a
positive margin they persist under a perturbation of $t$, and $s(t)$ is in the interior
of $\M$.

Maximizing this margin over the two-dimensional parameter $m_1$ --- a short numerical
search carried out in {\tt schottky} \cite{schottky} --- we find that the certificate can
be met with a strictly positive margin throughout the range
$$\arg(s)\in(\theta_*,\,\tfrac{\pi}{2}), \qquad \theta_*\approx 84.5^\circ \quad
(t\approx 0.096)$$
In this way the logarithmic-spiral certificate reaches down
to $\arg(s)\approx 84.5^\circ$, where it meets the range accessible to ordinary trap
certificates described in the sequel.

\subsection{Limit traps for $s_0:=1/2 + i/2$ and $s_1:=1/4 + \sqrt{-7}/4$}
\label{subsection:limit_traps}

In this subsection we explain how to certify interiority in a neighborhood of the
two points
$$s_0:=1/2 + i/2 \quad \text{ and } \quad s_1:=1/4 + \sqrt{-7}/4$$
at angles $\arg s_0 = 45^\circ$ and $\arg s_1 \approx 69.295^\circ$ respectively.

At each of $s_0$ and $s_1$ the attractor $\Lambda$ is a Jordan disk, the union
$\Lambda=f\Lambda\cup g\Lambda$ of two Jordan disks meeting along an arc in their 
common boundary. This common arc is symmetric under $z \to -z$, and therefore its
midpoint is $0$. The fact that $0 \in f\Lambda \cap g\Lambda$ in either case
(which implies that $s_0,s_1 \in \M$) follows from the identities
$$0 = fg f^\infty 0 = gf g^\infty 0 \text{ for } s_0$$
and
$$0 = f(gffg)^\infty 0 = g(fggf)^\infty 0 \text{ for } s_1$$

These identities certify that $s_0$ and $s_1$ are {\em renormalization points}
in the terminology of \cite{Calegari_Koch_Walker} \S~9.2 (or {\em landmark points}
in the terminology of Solomyak \cite{Solomyak}). Explicitly, there are finite
words $u,v$ with common length $a:=|u|=|v|$ so that $u$ starts with $f$ and $v$ starts with $g$,
and finite words $\alpha,\beta$ with common length $b:=|x|=|y|$ so that 
$0 = u \alpha^\infty 0 = v \beta^\infty 0$ for $s$.
For $s_0$ we have $u=fg$, $v=gf$, $\alpha=f$, $\beta=g$, $a=2$, $b=1$ whereas 
for $s_1$ we have $u=f$, $v=g$, $\alpha=gffg$, $\beta=fggf$, $a=1$, $b=4$.
Let's fix $\sigma \in \lbrace s_0,s_1\rbrace$ and fix $u,v,\alpha,\beta,a,b$ as above.
For any finite resp. infinite word $w$ we let $p_w(s)$ be the polynomial resp.
power series defined by $w_s z = z s^{|w|} + p_w(s)$, and for $u,v,\alpha,\beta$ define
$$P(s): = p_{u\alpha^\infty}(s) - p_{v\beta^\infty}(s)$$
so that (by definition) $P(\sigma)=0$.

Now \cite{Calegari_Koch_Walker} Lemma~9.2.5 says that if $x,y$ are words with
common length $c:=|x|=|y|$, and if there is a constant $C \in \C$ for which the
renormalized displacement
$$V(C;x,y): = \sigma^{-a-c}\left(p_u(\sigma) - p_v(\sigma)\right) + 
\sigma^{-c}\left(p_x(\sigma) - p_y(\sigma)\right) +
\sigma^{-a-c}C P'(\sigma)$$
is trap-like for $\sigma$, then the words $u\alpha^nx$ and $v\beta^ny$ form an
ordinary trap for $\sigma + C\sigma^{bn}$ for all sufficiently large $n$. In this
case one says that $C$ {\em admits a limit trap}.

Since this condition on $C$ is numerically stable, an entire open region of 
$C$-values may be certified. To certify $\sigma$ itself as an interior point, it 
suffices to cover an entire fundamental domain of the elliptic curve 
$E_\sigma:=\C^*/\langle \sigma^b \rangle$ by open regions of $C$-values as above.

Figure~\ref{asymp_cores} shows the results of a computer-aided search using
{\tt schottky} to successfully cover each fundamental domain by a finite union of
certified open regions. An estimate of the error guarantees an honest interval
of $\theta$ values around these landmark points for which $e^{i\theta}/\sqrt{2}$ is
in the interior of $\M$.

\begin{figure}[htpb]
\centering
\includegraphics[scale=1]{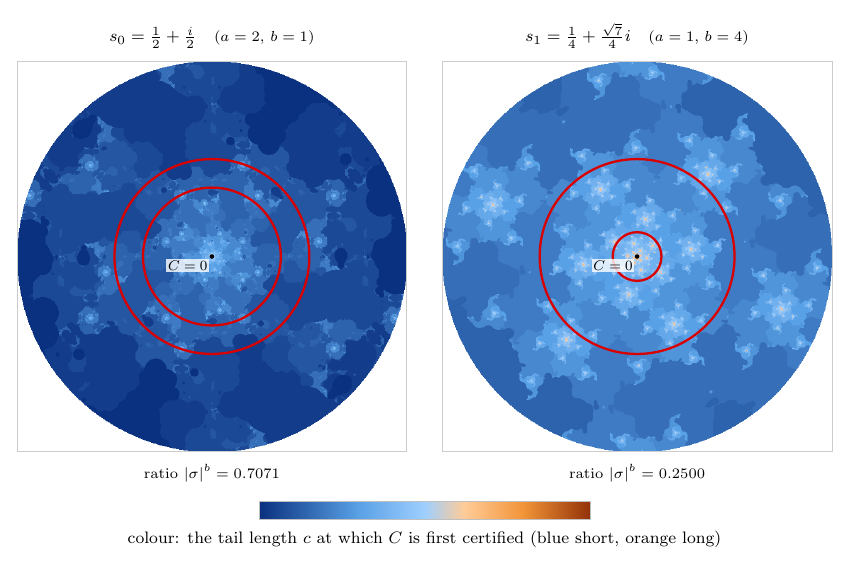}
\caption{Fundamental domains for the elliptic curves $E_\sigma$ are covered by
finitely many limit traps, certifying $\sigma$ as an interior point of $\M$
for $\sigma=s_0,s_1$.}
\label{asymp_cores}
\end{figure}

We conclude with two brief comments about the points $s_0,s_1$ and their associated
limit sets.

\begin{remark}[Planar 2-reptiles]
In fact, it is possible to characterize the three distinguished 
parameters $s_0$, $s_1$ and $i/\sqrt2$ intrinsically. For each of these three values,
the limit set $\Lambda$ is a
Jordan domain equal to the union $\Lambda = f\Lambda\cup g\Lambda$ of two interior--disjoint
and isometric copies of itself. In particular, in either case $\Lambda$ is an example of a {\em planar
$2$-reptile} in the language of \cite{Ngai_Sirvent_Veerman_Wang}. Rational planar $2$-reptiles
are classified in that paper up to isometry (Theorem~1.3); there are exactly 6 possibilities. 
Of these, in exactly three cases the two tiles are related by a translation --- the
rectangle (corresponding to $s=i/\sqrt{2}$), the twindragon ($s_0$) and the
tame twindragon ($s_1$); in the three remaining cases (the L\'evy dragon, the Heighway
dragon and the right isosceles triangle) the two tiles are related by a rotation, and
are irrelevant for our study. 

The paper \cite{Ngai_Sirvent_Veerman_Wang} leaves open the possibility that there might be
{\em irrational} planar $2$-reptiles beyond those on their list, although they conjecture
this does not occur. In fact, the numerical certificates obtained in this appendix 
{\em prove} this conjecture for the subclass of (not a priori rational) $2$-reptiles whose
tiles are related by a translation. It is plausible that a $2$-parameter generalization
of the method of traps and limit traps could certify this conjecture unconditionally.
\end{remark}

\begin{remark}[Half-zippers]
The coincidences $u\alpha^\infty 0 = v\beta^\infty 0$ for
$s_0$ and $s_1$ give rise to a construction of a canonical topological $\R$-tree
$T$ dense in the Jordan domain $\Lambda$ parallel to that in Example~\ref{example:R_tree}.
In each case the boundary of the Jordan domain $\Lambda$ may be glued up by isometries to produce
a conformal Riemann sphere in which the image of $T$ becomes a half-zipper in the sense of
\cite{Calegari_Gwynne}. The coincidence $f(fffg)^\infty 0 = g(fgff)^\infty 0$ for
$i/\sqrt{2}$ gives rise to a third example. Together with Example~\ref{example:R_tree} 
numerical evidence suggests this is an exhaustive list (up to $s \to -s$ and $s \to \bar{s}$).

These three topological $\R$-trees are depicted in Figure~\ref{dragon_R_trees}.

\begin{figure}[htpb]
\centering
\includegraphics[scale=0.4]{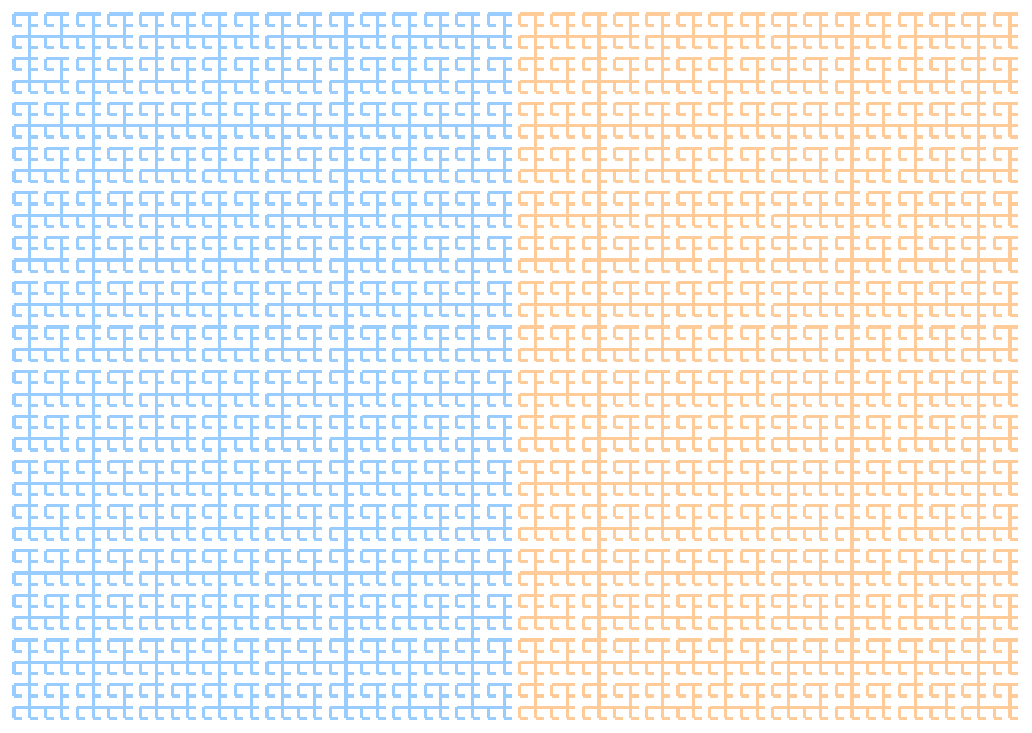}
\includegraphics[scale=0.4]{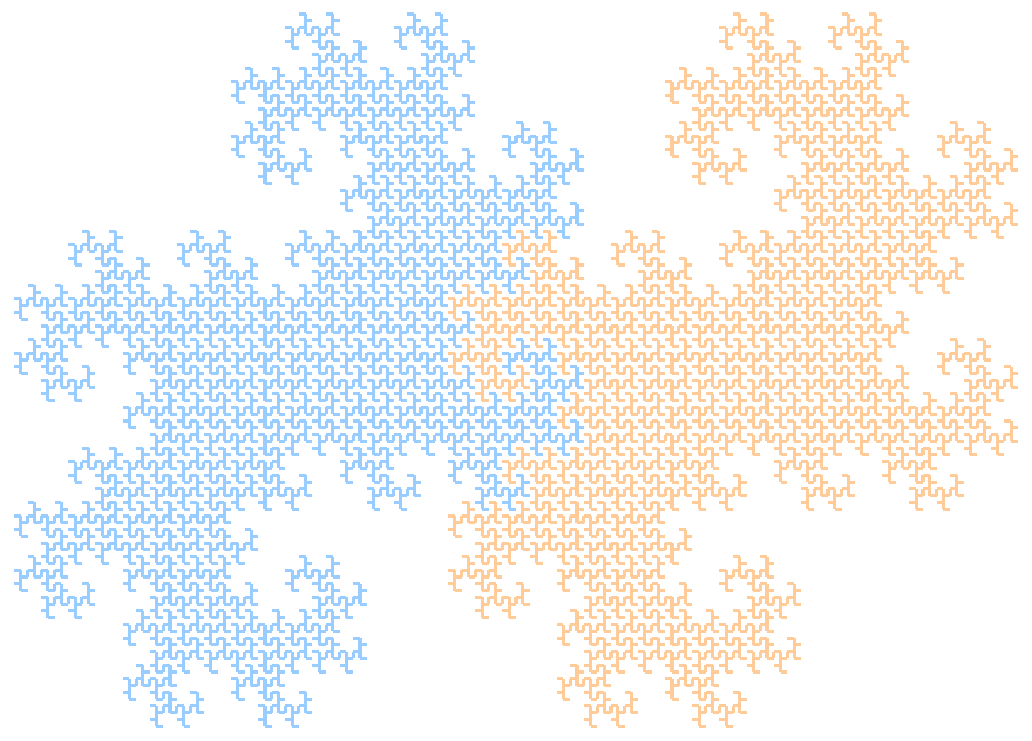}
\includegraphics[scale=0.4]{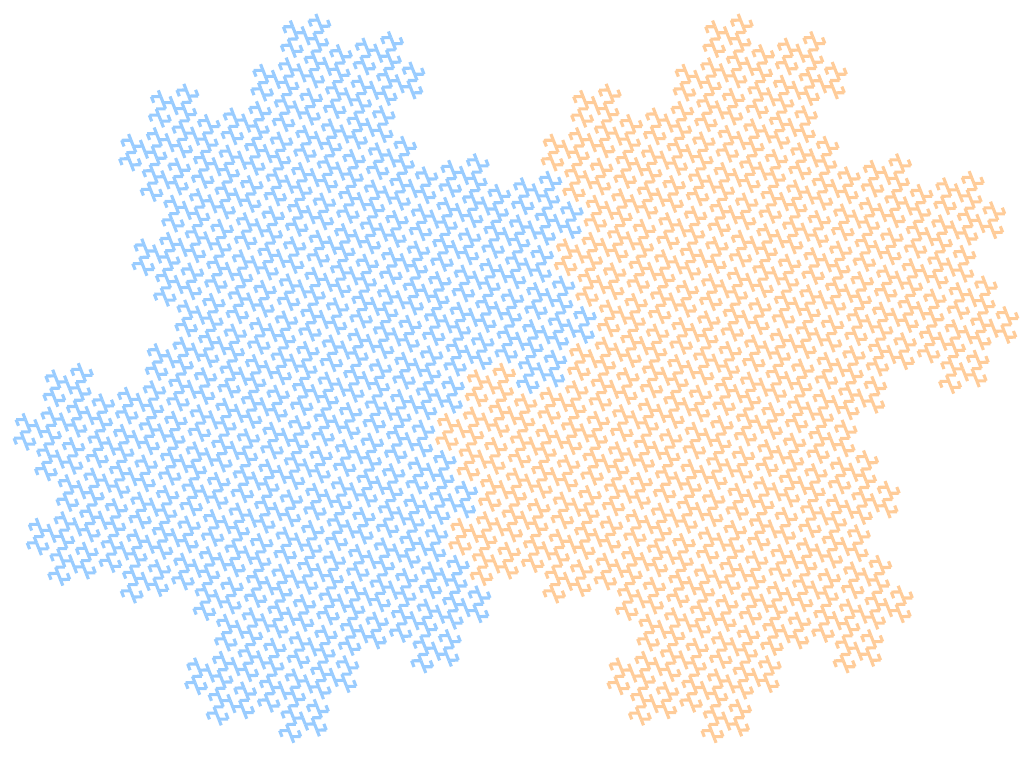}

\caption{Topological $\R$-trees dense in $\Lambda$ for $i/\sqrt{2}$, $s_0$ and $s_1$ 
giving rise to half-zippers.}
\label{dragon_R_trees}
\end{figure}

In fact, the half-zipper associated to $s_1$ is identical to the double-branched cover of
the half-zipper associated 
to the geometric mating of $z \to z^2 + i$ with itself discussed in \cite{Calegari_Loukidou}
Example~7.7 (see especially \cite{Calegari_Loukidou} Figure~25). This is a
Latt\`es example, derived from the elliptic curve with complex multiplication
$E=\C/\langle 1,\eta\rangle$ for $\eta = 1/2 + \sqrt{-7}/2 = 2s_1$.
\end{remark}

\subsection{Real points and affine traps}\label{subsection:affine_traps}

The method of affine traps is explained in \cite{Calegari_Koch_Walker} \S~10.2--3 and
can be used to certify interiority in $\M$ for real points $s$ for which ordinary
traps are unavailable. The idea is (roughly) to consider the analog of our IFS for
$s'$ a nilpotent imaginary thickening of a real parameter $s$. Explicitly, let's
write $s': = s+i\epsilon$ and $z:=x+i\epsilon y$ in $\R[i\epsilon]/(i\epsilon)^2$. Then
if we write $x+i\epsilon y$ as $(x,y)$, the maps $f$ and $g$ become 
$$f^{(1)}:(x,y) \to (sx-1,x+sy) \quad g^{(1)}:(x,y) \to (sx+1,x+sy)$$
which can now be reinterpreted as a 2 dimensional real linear IFS with attractor
$L_{s'} \subset \R^2$. As before, $L_{s'}$ is connected if and only if $fL_{s'}$
intersects $gL_{s'}$. Note that a slightly different formula appears in
\cite{Calegari_Koch_Walker} \S~10.2 reflecting the different normalization of $f$ and
$g$ throughout that paper.

Affine traps are the obvious generalization of ordinary traps in this context:
if we can find words $u,v$ starting with $f,g$ respectively
for which $uL_{s'}$ and $vL_{s'}$ form a trap (in $\R^2$) then the real number
$s$ is an interior point of $\M$; the precise statement is
\cite{Calegari_Koch_Walker}, Proposition~10.3.2, and Lemma~10.3.4 contains the
estimates that let us compare the limit set $L_{s'}$ and
the limit set $L_{s+i\epsilon}$ with $y$ coordinate scaled by $1/\epsilon$ and certify
that $s+i\epsilon$ admits an honest trap for all sufficiently small nonzero $\epsilon$. 

Affine traps are abundant at $s=1/\sqrt{2}$ and certify interiority of $e^{i\theta}/\sqrt{2}$
for $\theta$ in a definite neighborhood of $0$.

\subsection{Ordinary traps in the bulk}
\label{subsection:explicit_certificate}

Having certified interiority in $\M$ for explicit neighborhoods of $1/\sqrt{2},s_0,s_1$
and for $e^{i\theta}/\sqrt{2}$ for $\theta \in (84.5^\circ,90^\circ)$ it remains to
certify $e^{i\theta}/\sqrt{2}$ for $\theta$ in the remainder of $[0^\circ,90^\circ)$, 
a subset that we refer to informally as `the bulk'. This was carried out by
the program {\tt schottky}, completing the proof of Theorem~\ref{theorem:strict_inequality}. 
Two remarks are in order.

\begin{remark}
The statement that every $s=e^{i\theta}/\sqrt{2}$ for $\theta \in (0^\circ,84.5^\circ) -
\lbrace 45^\circ,\tan^{-1}(1/\sqrt{7})\rbrace$ can be certified as interior to
$\M$ {\em by a trap} is strictly stronger than simply the statement that these $s$ lie
in the interior of $\M$. The main result of \cite{Calegari_Koch_Walker}, namely Theorem~7.2.7,
proves that a dense subset of the interior of $\M$ may be certified by traps, but there
is no characterization --- even conjecturally --- of the set of interior $s$ in $\M$ that
{\em cannot} be certified by traps. This set includes all $s$ for which
$\Lambda$ is convex, which holds if and only if $s=re^{\pi ip/q}$ for
coprime integers $p,q$ and real $r\ge 2^{-1/q}$; see \cite{Calegari_Koch_Walker}
Lemma~7.2.3, but we have already seen it is strictly bigger --- for example, 
it includes the points $s_0$ and $s_1$.
\end{remark}

\begin{remark}
The computer search to certify all $\theta$ in the bulk was carried out {\em before}
the special points $s_0$ and $s_1$ were identified. More than half of the computational
time spent on the certification was devoted to trying (and failing) to certify
microscopic neighborhoods of these two points. The fact that the program failed to
certify these points by traps (as logic dictates it must) may be seen as an independent
check of its reliability.
\end{remark}

\end{document}